\documentclass[a4paper, 10pt, american]{amsart}

\usepackage{times,latexsym,amssymb}
\usepackage{amsmath,amsthm,bm}
\usepackage{xcolor}
\usepackage[colorlinks,pdfpagelabels,pdfstartview = FitH,bookmarksopen= true,bookmarksnumbered = true,linkcolor = blue,plainpages = false,hypertexnames = false,citecolor = red,pagebackref=false]{hyperref}
\usepackage{amsbsy}
\usepackage{amstext}
\usepackage{amssymb}
\usepackage{stmaryrd}
\usepackage{esint}
\usepackage{stmaryrd}
\usepackage{graphicx}
\usepackage{floatflt}
\usepackage{float} 
\usepackage{appendix}
\usepackage{mathrsfs}
\usepackage{psfrag}
\usepackage{caption}

\SetSymbolFont{stmry}{bold}{U}{stmry}{m}{n}

\allowdisplaybreaks
\newtheorem{theorem}{Theorem}
\newtheorem{lemma}[theorem]{Lemma}
\newtheorem{definition}[theorem]{Definition}
\newtheorem{proposition}[theorem]{Proposition}
\newtheorem{corollary}[theorem]{Corollary}
\newtheorem{remark}[theorem]{Remark}
\numberwithin{theorem}{section}
\numberwithin{equation}{section}

\newcommand{\mint}{- \mskip-19,5mu \int}

\def\N{\mathbb{N}}
\def\R{\mathbb{R}}

\def\bg{{\boldsymbol\gamma}}

\renewcommand{\d}{\mathrm{d}}
\newcommand{\dx}{\mathrm{d}x}
\newcommand{\dy}{\mathrm{d}y}

\newcommand{\dt}{\mathrm{d}t}
\newcommand{\ds}{\mathrm{d}s}
\newcommand{\dtau}{\mathrm{d}\tau}

\renewcommand{\epsilon}{\varepsilon}

\newcommand{\essup}{\operatornamewithlimits{ess\,sup}}

\DeclareMathOperator{\Div}{div}

\DeclareMathOperator{\dist}{dist}
\DeclareMathOperator{\loc}{loc}
\newcommand{\osc}{\operatornamewithlimits{osc}}

\renewcommand{\epsilon}{\varepsilon}
\newcommand{\eps}{\varepsilon}
\renewcommand{\rho}{\varrho}

\def\eqn#1$$#2$${\begin{equation}\label#1#2\end{equation}}

\newcommand{\pl}{\partial}

\newcommand{\al}{\alpha}

\newcommand{\gm}{\gamma}

\newcommand{\lm}{\lambda}

\newcommand{\varep}{\varepsilon}

\newcommand{\sig}{\sigma}

\newcommand{\z}{\zeta}

\let\TeXchi\chi
\newbox\chibox
\setbox0 \hbox{\mathsurround0pt $\TeXchi$}
\setbox\chibox \hbox{\raise\dp0 \box 0 }
\def\chi{\copy\chibox}

\def\Xint#1{\mathchoice
    {\XXint\displaystyle\textstyle{#1}}%
    {\XXint\textstyle\scriptstyle{#1}}%
    {\XXint\scriptstyle\scriptscriptstyle{#1}}%
    {\XXint\scriptscriptstyle\scriptscriptstyle{#1}}%
    \!\int}
\def\XXint#1#2#3{\setbox0=\hbox{$#1{#2#3}{\int}$}
    \vcenter{\hbox{$#2#3$}}\kern-0.5\wd0}
\def\bint{\Xint-}
\def\dashint{\Xint{\raise4pt\hbox to7pt{\hrulefill}}}

\def\Xiint#1{\mathchoice
    {\XXiint\displaystyle\textstyle{#1}}%
    {\XXiint\textstyle\scriptstyle{#1}}%
    {\XXiint\scriptstyle\scriptscriptstyle{#1}}%
    {\XXiint\scriptscriptstyle\scriptscriptstyle{#1}}%
    \!\iint}
\def\XXiint#1#2#3{\setbox0=\hbox{$#1{#2#3}{\iint}$}
    \vcenter{\hbox{$#2#3$}}\kern-0.5\wd0}
\def\biint{\Xiint{-\!-}}

\subjclass[2020]{35K65, 35K67, 35B45, 35B65, 35K92, 76S05}

\begin{document}

\title[Schauder estimates for parabolic $p$-Laplace systems]
{Schauder estimates for parabolic $p$-Laplace systems}
\date{\today}

\author[V. B\"ogelein]{Verena B\"{o}gelein}
\address{Verena B\"ogelein,
Fachbereich Mathematik, Universit\"at Salzburg,
Hellbrunner Str. 34, 5020 Salzburg, Austria}
\email{verena.boegelein@plus.ac.at}

\author[F. Duzaar]{Frank Duzaar}
\address{Frank Duzaar,
Fachbereich Mathematik, Universit\"at Salzburg,
Hellbrunner Str. 34, 5020 Salzburg, Austria}
\email{frankjohannes.duzaar@plus.ac.at}

\author[U. Gianazza]{Ugo Gianazza}
\address{Ugo Gianazza,
Dipartimento di Matematica ``F. Casorati",
Universit\`a di Pavia,
via Ferrata 5, 27100 Pavia, Italy}
\email{ugogia04@unipv.it}

\author[N. Liao]{Naian Liao}
\address{Naian Liao,
Fachbereich Mathematik, Universit\"at Salzburg,
Hellbrunner Str. 34, 5020 Salzburg, Austria}
\email{naian.liao@plus.ac.at}

\author[C. Scheven]{Christoph Scheven}
\address{Christoph Scheven,
Fakult\"at f\"ur Mathematik, 
Universit\"at Duisburg-Essen, Thea-Leymann-Str. 9, 45127 Essen, Germany}
\email{christoph.scheven@uni-due.de}

\begin{abstract}
We establish the local H\"older regularity of the spatial gradient of bounded weak solutions $u\colon E_T\to\R^k$ to the non-linear system of parabolic type
\begin{equation*}
  \partial_tu-\Div\Big( a(x,t)\big(\mu^2+|Du|^2\big)^\frac{p-2}2Du\Big)=0
  \qquad\mbox{in $E_T$},
\end{equation*}
where $p>1$, $\mu\in[0,1]$, and the coefficient $a\in L^\infty(E_T)$ is bounded below by a positive constant and is H\"older continuous in the space variable $x$.
As an application, we prove H\"older estimates for the gradient of weak solutions to a doubly non-linear parabolic equation
in the super-critical fast diffusion regime. 
\end{abstract}

\maketitle

\setcounter{tocdepth}{1}
 \tableofcontents

\section{Introduction}

Let $E$ be a bounded open set in $\R^N$ and $E_T:=E\times(0,T]$ for some $T>0$. Consider weak solutions $u\colon E_T\to\R^k$ to the following non-linear system of parabolic type
\begin{equation}\label{p-laplace-intro}
  \partial_tu-\Div\Big( a(x,t)\big(\mu^2+|Du|^2\big)^\frac{p-2}2Du\Big)=0
  \qquad\mbox{in $E_T$},
\end{equation}
for parameters $p>1$ and $\mu\in[0,1]$,  where the coefficient $a\in L^\infty(E_T)$ satisfies
\begin{equation}\label{prop-a-intro}
\left\{
  \begin{array}{c}
    C_o\le a(x,t)\le C_1,\\[6pt]
    |a(x,t)-a(y,t)|\le C_1|x-y|^\alpha,
  \end{array}
  \right.
\end{equation}
for a.e.~$x,y\in E$ and $t\in(0,T]$, with a H\"older exponent
$\alpha\in(0,1)$ and structural constants $0<C_o\le C_1$. 
For {\it bounded} weak solutions to this system, we establish local H\"older regularity of their spatial gradient.
The precise result, including the corresponding quantitative estimates of Schauder-type, is stated as follows. Notation can be found in Section~\ref{S:notation}.

\begin{theorem}[Schauder estimate for parabolic $p$-Laplace systems]\label{theorem:schauder}
Let $p>1$ and $\mu\in[0,1]$. Then there exist a constant
$C>0$ and a H\"older exponent $\alpha_o\in (0,1)$ both depending on $N,p,C_o,C_1,\alpha$ and $k$ such that, whenever $u$ is a bounded weak solution to the parabolic system 
\eqref{p-laplace-intro} with \eqref{prop-a-intro}, we have 
\begin{equation*}
    Du\in C^{\alpha_o,\alpha_o/2}_{\mathrm{loc}}\big(E_T,\R^{kN}\big).
\end{equation*}
Moreover, for any compact subset $\mathsf{K}\subset E_T$ 
with $\varrho:= \tfrac14\min\{1,
\mathrm{dist}_{\mathrm{par}}(\mathsf {K},\partial_\mathrm{par} E_T)\}$, we have
the quantitative local $L^\infty$-gradient estimate
\begin{equation}
    \label{gradient-sup-bound-final}
    \sup_{\mathsf{K}}|Du|
    \le
    C \bigg[\frac{\osc_{E_T}u}{\rho}
    +
    \Big(\frac{\osc_{E_T}u}{\rho}\Big)^{\frac{2}{p}}
    +\mu
    \bigg]
    =:
    C\lambda,
\end{equation}
and the quantitative local gradient H\"older estimate
\begin{equation}\label{gradient-holder-bound-final}
    |Du(x_1,t_1) - Du(x_2,t_2)|
    \le
    C\lambda
    \Bigg[\frac{|x_1-x_2|+\sqrt{\lambda^{p-2}|t_1-t_2|}}{\min\{1,\lambda^{\frac{p-2}{2}}\}\rho}\Bigg]^{\alpha_o}
\end{equation} 
for any $(x_1,t_1),(x_2,t_2)\in\mathsf{K}$.
\end{theorem}

In their groundbreaking works \cite{DiBenedetto-Friedman, DiBenedetto-Friedman2,
DiBenedetto-Friedman3} DiBenedetto \& Friedman established H\"older regularity of the gradient  of solutions to the model $p$-Laplace system; they covered the super-critical range $p>p_\ast:=\frac{2N}{N+2}$. For an extension to the sub-critical range $1<p\le p_*$ see  \cite{Choe:1991}, although for computational simplicity only the scalar case ($k=1$) was considered. These works dealt with model cases without coefficients, 
whereas systems of parabolic $p$-Laplace-type with H\"older continuous coefficients were treated in
\cite{Chen-DB-89, Kuusi-Mingione, Misawa-Schauder}, focusing on $p>p_*$. Furthermore, the case of variable exponents $p(x,t)> p_\ast$ was included in \cite{Boegelein-Duzaar:p(z)}.

The main new feature of our result is
that we cover systems with H\"older continuous coefficients for arbitrary exponents $p>1$. Moreover, Schauder-type estimates are given in terms of the oscillation of the solution; we will show an interesting application of such novel explicit estimates in Section~\ref{sec:DNL}. Below the critical exponent $p_\ast$, weak solutions need not be locally bounded; see the counterexamples in \cite{BDGLS-23}. Therefore, Schauder-type estimates cannot
be expected to hold for arbitrary weak solutions
to~\eqref{p-laplace-intro} if
$1<p\le p_*$. However, we are able to establish a unified theory for {\it bounded}
solutions. In turn, local boundedness of solutions can be established under the assumption $u\in L^r_{\loc}(E_T,\R^k)$ for some $r>\frac{N(2-p)}{p}$. Therefore, the local boundedness assumption could in fact be replaced by this integrability assumption.

Part of this manuscript has been extracted and extended from the arXiv preprint \cite{BDGLS-23} which does not deal with systems.

\subsection{Method of proof}
The proof of the theorem is divided into two steps. As a first step, in Section \ref{S:grad-bound}, we consider systems with differentiable coefficients. In order to avoid further technical complication that arises from the degeneracy of the system, in such a case we restrict ourselves to the case $\mu>0$. However, it is crucial that all constants are independent of $\mu$. The second step in Section~\ref{sec:Schauder-for-Lipschitz} deals with systems with H\"older continuous coefficients. The method consists in  considering a comparison problem with frozen coefficients: Relevant comparison estimates are prepared in Section \ref{S:4}.

The most important results in Section \ref{S:grad-bound} are the quantitative gradient bounds of solutions. To get them, we adapt a general scheme in \cite[Chapter VIII, Section 4]{DB}, that is, we will first reach $|Du|\in L^{\infty}_{\loc}(E_T)$ in a qualitative way; then, we will turn such qualitative information into precise quantitative estimates. 
Special care is needed for the sub-critical range $1<p\le\frac{2N}{N+2}$ as solutions might be unbounded.  
To get proper gradient bounds, we need to work with {\it bounded} solutions, adapt a trick from \cite[Chapter VIII, Lemma 4.1]{DB} (see also \cite{DiBenedetto-Friedman3, Choe:1991}), and raise the integrability of $|Du|$ by a clever manipulation of the energy estimate for second order derivatives in~\eqref{energy-est}. 
This is a unified treatment for all exponents $1<p\le 2$ among {\it bounded} solutions and yields an estimate of the $L^m$-norm of $|Du|$ for every $m>1$ in terms of its $L^p$-norm and the oscillation of $u$, cf. Lemma~\ref{lem:Lq-est}. With this result at hand, we run Moser's and De Giorgi's iterations to obtain the gradient bound in terms of the $L^p$-norm of $|Du|$ and the oscillation of $u$. 

Once the boundedness of the gradient is established
in the case of differentiable coefficients, then the H\"older continuity of the gradient follows for all exponents $p>1$ from the theory developed in the seminal works by DiBenedetto \& Friedman \cite{DiBenedetto-Friedman, DiBenedetto-Friedman2,
DiBenedetto-Friedman3}, see also \cite{Chen,Choe:1991} 
and the monograph \cite[Chapter~IX]{DB}.

The aim of Section~\ref{sec:Schauder-for-Lipschitz} is to generalize these regularity results to systems with merely H\"older continuous coefficients. The strategy for this is to compare the given solution with a solution to a related problem with frozen coefficients. 
This is also the underlying idea in the earlier work \cite{Boegelein-Duzaar:p(z)}, which covers variable exponents $p(x,t)>\frac{2N}{N+2}$. However, in the present work we develop  
a slightly different approach than that in
\cite{Boegelein-Duzaar:p(z)}, which allows us to include the sub-critical 
case, and meanwhile, contains some simplifications also in the super-critical case. First we prove \textit{ a priori} estimates under the additional assumption that the gradient is bounded. 
The virtual advantage of this approach is that we are able to work with cylinders $Q_{\rho}^{(\lambda)}$ of a fixed geometry given 
by the parameter $\lambda$, which can be chosen in terms of $\|Du\|_{L^\infty}$. 
This represents a significant simplification compared to \cite{Boegelein-Duzaar:p(z)}, in which the authors were forced to use a sequence of cylinders $Q_{\rho_i}^{(\lambda_i)}$ with varying parameters $\lambda_i$ that take into account a possible blow-up of the gradient. 
In order to work with cylinders of a fixed geometry, it is crucial that the article \cite{BDLS-Tolksdorf} provides an improved version of a Campanato-type estimate for differentiable coefficients that holds on cylinders $Q_r^{(\lambda)}$ with a fixed geometry given by $\lambda>0$ and arbitrary $r\in(0,\rho]$, cf. Proposition~\ref{prop:apriori}. In Section~\ref{sec:Schauder-for-Lipschitz}, we prove that this Campanato-type estimate 
can be extended to systems with merely H\"older continuous coefficients, assuming that the gradient is bounded. This yields the desired \textit{ a priori} estimates. Finally, using these \textit{a priori} estimates we prove Theorem~\ref{theorem:schauder} by an approximation argument.  

We note that in the super-critical case $p>\frac{2N}{N+2}$, the method described above can also be used to derive gradient bounds that are independent of the oscillation of the solution and depend only on integrals of $|Du|^p$ instead. The necessary changes are indicated in Remark~\ref{rem:super-critical}. Consequently, in the super-critical range, our approach is not limited to the treatment of bounded solutions.

{\bf Acknowledgments.} 
This research was funded in whole or in part by the Austrian Science Fund (FWF) [10.55776/P31956] and [10.55776/P36272]. For open access purposes, the author has applied a CC BY public copyright license to any author accepted manuscript version arising from this submission. Ugo Gianazza was supported in part by Grant  2017TEXA3H\_002 ``Gradient flows, Optimal Transport and Metric Measure Structures”.

\section{Preliminaries and notation}
\subsection{Notation}\label{S:notation}
Throughout the manuscript $\mathbb{N}_0$ denotes $\mathbb{N}\cup\{0\}$, while $\R^k$ and $\R^N$ are Euclidean spaces. Define $E_T:= E\times (0,T]$ with $E$  a bounded open set of $\R^N$ and $T>0$.
Given a \emph{scalar} function $v\colon E_T\to\R$, by $\nabla v$ we always mean its spatial gradient, whereas for a \emph{vector valued} function $u\colon E_T\to\R^k$, its spatial gradient will be denoted by $D u$.

Moreover, we have
\[
a\vee b:=\max\{a,b\}.
\]
For a point $z_o\in \R^N\times \R$, $N\in \N$,  we always write $z_o=(x_o,t_o)$. 
By $B_\rho(x_o)$
we denote the ball in $\R^N$ with center $x_o\in\R^N$ and radius $\rho>0$. When $x_o=0$ we simply write $B_\rho$, omitting the reference to the center.

We use the symbol 
$$
    Q_{R,S}(z_o):= B_R (x_o)\times (t_o-S,t_o]
$$
to denote a \textbf{general backward parabolic cylinder} with the indicated parameters. Occasionally, we also use $\sig Q_{R,S}(z_o)\equiv Q_{\sig R,\sig S}(z_o)$ for $\sig>0$.
For $\lambda>0$ we define a second type of \textbf{backward intrinsic parabolic cylinders}
\begin{equation}\label{def:cyl-lambda}
    Q_\rho^{(\lambda)} (z_o)
    :=
    B_\rho (x_o)\times (t_o-\lambda^{2-p}\rho^2, t_o] .
\end{equation}
If $\lm=1$, then we omit it. When $z_o=(0,0)$ or the context is clear, we also omit $z_o$.

The {\bf parabolic boundary} of $E_T$ is given by the union of its lateral and initial boundary
$$
    \pl_\mathrm{par} E_T
    :=
    \big(\overline E\times\{0\}\big) \cup \big(\partial E\times(0,T]\big).
$$
For two points $z_1=(x_1,t_1), z_2=(x_2,t_2)\in\R^{N+1}$ their {\bf parabolic distance} is defined as 
$$
    \mathrm{d}_\mathrm{par}(z_1,z_2)
    :=
    |x_1-x_2| + \sqrt{|t_1-t_2|}.
$$
The associated distance of a subset $Q\subset E_T$ to the parabolic boundary $\partial_\mathrm{par} E_T$ is
$$
    \dist_\mathrm{par}(Q,\partial_\mathrm{par} E_T)
    :=
    \inf_{\substack{z_1\in Q\\ z_2\in\partial_\mathrm{par} E_T}} \mathrm{d}_\mathrm{par}(z_1,z_2).
$$
For $\lambda>0$ we also define the following {\bf intrinsic} version of the {\bf parabolic distance} 
\begin{equation*}
  \mathrm{d}_\mathrm{par}^{(\lambda)}(z_1,z_2):=|x_1-x_2|+\sqrt{\lambda^{p-2}|t_1-t_2|}.
\end{equation*}
Although $\mathrm{d}_\mathrm{par}^{(\lambda)}$
also depends on $p$, since $p$ is always fixed we suppress this dependence.

Given a $p$-summable function $v\colon A\to\R^k$, with a measurable set $A$ and $k\in\mathbb{N}$, whenever we want to explicitly point out the set where the function takes values, we write $v\in L^p(A,\R^k)$. In the same way, for example, $v\in C^\infty_0(E,[0,1])$ denotes a function $v\in C^\infty_0(E)$ defined on an open set $E$ which takes values in $[0,1]$.

For a function $v\in L^1(E_T,\R^k)$,  and a measurable subset $A\subset E$ with positive $\mathcal L^N$-measure we define the slice-wise mean $(v)_A\colon(0,T)\to\R^k$ of $v$ on $A$ by 
\begin{equation*}
    (v)_A(t)
    :=
    \bint_A v(\cdot,t)\,\dx,
    \quad\mbox{for a.e. $t\in(0,T)$.}
\end{equation*}
Note that the slice-wise mean is defined for any $t\in(0,T)$ if $v\in C(0,T;L^1(E))$. 
If the set $A$ is a ball $B_\rho(x_o)$, then we write $(v)_{x_o,\rho}(t)$ for short. Similarly, for a measurable set $Q\subset E_T$ of positive $\mathcal L^{N+1}$-measure we define the mean value $(v)_Q$ of $v$ on $Q$ by 
$$
    (v)_Q
    :=
    \biint_Q v \,\dx\dt.
$$
If the set $Q$ is a parabolic cylinder $Q_\rho^{(\lambda)}(z_o)$ of the type \eqref{def:cyl-lambda}, we write $(v)_{z_o,\rho}^{(\lambda)}$ for short. 
When it is clear from the context, which vertex $x_o$ or $(x_o,t_o)$ is meant, we will omit the vertex from the above symbols for simplicity.

\subsection{Parabolic function spaces}
Let $X$ be a Banach space with its natural norm $\|\cdot\|_X$, and $T>0$. 
For $1\le p<\infty$ we denote by $L^p(0,T;X)$ the (parabolic) Bochner space of all (strongly) measurable functions
$v\colon [0,T]\to X$ such that the Bochner norm is finite, i.e.
\begin{equation*}
    \| v\|_{L^p(0,T;X)}:=\bigg[\int_0^T\|v(t)\|_X^p\,\dt\bigg]^\frac1p <\infty.
\end{equation*}
Similarly, for $p=\infty$ the Bochner space $L^\infty(0,T;X)$ consists of all (strongly) measurable functions $v\colon [0,T]\to X$, such that $t\to\|v(t)\|_X$ is essentially bounded, which means that
\begin{equation*}
    \| v\|_{L^\infty (0,T;X)}:=\essup_{t\in [0,T]}\|v(t)\|_X<\infty.
\end{equation*}
These spaces are Banach spaces.  We could have defined these spaces on $(0,T)$ or $(0,T]$, which makes no difference at all, because in the end we get the same spaces. Finally, the Bochner space $C([0,T];X)$ denotes the space of all continuous functions $v\colon[0,T ]\to X$. 
The Bochner norm on this space is given by
\begin{equation*}
    \| v\|_{C([0,T];X)}:=\max_{t\in [0,T]}\|v(t)\|_X.
\end{equation*}
Next, the space $C((0,T);X)$ consists  of all continuous functions $v\colon (0,T)\to X$. 
Alternatively, $C([0,T];X)$  could be defined as the space of all continuous functions 
$v\colon(0,T)\to X$ that have a continuous extension (still denoted by $v$) to $[0,T]$. Analogous definitions hold for $C((0,T];X)$ and $C([0,T);X)$. 

Given a bounded, connected, open set $E\subset\R^N$ and the corresponding cylinder $E_T$, for $\alpha\in(0,1)$ the parabolic H\"older space $C^{\alpha,\alpha/2}(\bar E_T,\R^k)$ consists of all functions $u\colon E_T\to\R^k$, which admit continuous extension to $\bar E_T$, and such that  
\begin{equation*}
    \sup_{t\in [0,T]}\langle u(\cdot ,t)\rangle^{(\alpha)}_{x,E_T}+
    \sup_{x\in\bar E}\langle u(x,\cdot)\rangle^{(\alpha/2)}_{t,E_T}<\infty.
\end{equation*}
Here, we used the short-hand abbreviation
\begin{align*}
    \langle u(\cdot ,t)\rangle^{(\alpha)}_{x,E_T}
    &=
    \sup_{x,x'\in\bar E}\frac{|u(x,t)-u(x',t)|}{|x-x'|^\alpha},\\
    \langle u(x,\cdot)\rangle^{(\alpha)}_{t,E_T}
    &=
    \sup_{t,t'\in [0,T]}\frac{|u(x,t)-u(x,t')|}{|t-t'|^\alpha}.
\end{align*}
The space $C^{\alpha,\alpha/2}(\bar E_T,\R^k)$ is a Banach space. Local variants are defined as usual by letting $C^{\alpha,\alpha/2}_{\mathrm{loc}}(E_T,\R^k)$ the space of functions $v\in C^{\alpha,\alpha/2}(\bar Q_T,\R^k)$ for any compactly contained sub-cylinder $\bar Q_T\Subset E_T$. The range $\R^k$ that appears in various function spaces is omitted if $k=1$ or the context is clear.

\subsection{Auxiliary lemmas}
\begin{lemma}\textup{\cite[Lemma 2.3]{BDLS-boundary}}\label{lem:A}
Let $A>1$, $\kappa>1$, $C>0$ and $i\in\N$. Then, we have 
\begin{equation*}
	\prod_{j=1}^i
	A^{\frac{\kappa^{i-j+1}}{C(\kappa^i-1)}}
	=
	A^\frac{\kappa}{C(\kappa-1)}\quad\mbox{and}\quad
	\prod_{j=1}^i
	A^{\frac{j\kappa^{i-j+1}}{C(\kappa^i-1)}}
	\le
	A^{\frac{\kappa^2}{C(\kappa-1)^2}} .
\end{equation*}	
\end{lemma}
\begin{lemma}\textup{\cite[Lemma 6.1, p.191]{Giusti}}\label{lem:tech}
Let $0<\vartheta<1$, $A,B, C\ge 0$ and $\alpha,\beta \ge 0 $. Then there exists a universal constant  $c = c(\alpha,\beta, \vartheta)$
such that, whenever $\phi\colon [R_o,R_1]\times [S_1,S_2]\to \R$  is non-negative bounded function satisfying
\begin{equation*}
	\phi(r_1,s_1)
	\le
	\vartheta \phi(r_2,s_2) 
    + 
    \frac{A}{(s_2-s_1)^\alpha}
    + 
    \frac{B}{(r_2-r_1)^\beta}
    + C
\end{equation*}
for all $R_o\le r_1 <r_2\le R_1$, $S_1\le s_1 <s_2\le S_2$,
we have
\begin{equation*}
	\phi(\varrho,\sigma)
	\le
	c\,  
    \bigg[\frac{A}{(s-\sigma )^\alpha} +
    \frac{B}{(r-\varrho)^\beta} + C\bigg]
\end{equation*}
for any $R_o\le \varrho<r\le R_1$ and $S_1\le \sigma <s\le S_2$.
\end{lemma}

For the next lemma we refer to \cite{Acerbi-Fusco} when $1<p<2$ and \cite{GiaquintaModica:1986-a} when $p\ge 2$.
\begin{lemma}\label{AFGM}
Let $\mu\in (0,1]$ and $p>1$. For any  $\xi,\tilde\xi\in\R^{Nk}$ we have 
\begin{align*}
	\Big|\big( \mu^2 +|\xi|^2\big)^\frac{p-2}{2}\xi 
    -\big( \mu^2 +|\tilde\xi|^2\big)^\frac{p-2}{2}\tilde \xi\Big|
	\le
	C \big(\mu^2 + |\xi|^2 + |\tilde\xi|^2\big)^{\frac{p-2}{2}}
	|\xi-\tilde\xi|
\end{align*}
and 
\begin{align*}
	\Big(\big( \mu^2 +|\xi|^2\big)^\frac{p-2}{2} \xi
    - \big( \mu^2 +|\tilde \xi|^2\big)^\frac{p-2}{2}\tilde\xi \Big)\cdot 
	\big(\xi-\tilde\xi\big)
	&\ge
	\tfrac{1}{C}
	\big(\mu^2 + |\xi|^2 + |\tilde\xi|^2\big)^{\frac{p-2}{2}}
	|\xi-\tilde\xi|^2,
\end{align*}
with a positive constant $C=C (p)$. 
\end{lemma}

Weak solutions to parabolic equations or systems in general do not possess a time derivative in the Sobolev sense. On the other hand, it is desirable to use weak solutions in testing functions. In order to overcome this problem, proper mollifications in the time variable are often employed. A simple version, termed the Steklov average, is sufficient for many occasions. 
Indeed, given a function $v \in L^1(E_T)$ and $0<h<T$, we define its \textit{forward Steklov-average} $[v]_h$ by 
\begin{equation*}
	[v]_h(x,t) 
	:=
	\left\{
	\begin{array}{cl}
		\displaystyle{\frac{1}{h} \int_t^{t+h} v(x,s) \,\ds ,}
		& t\in (0,T-h) , \\[9pt]
		0 ,
		& t\in (T-h,T) 
\,.
	\end{array}
	\right.
\end{equation*}
Steklov averages are differentiable with respect to $t$. More precisely, we have
\begin{equation*}
    \partial_t [v]_h(x,t)=\tfrac1h \big(v(x,t+h)-v(x,t)\big)
\end{equation*}
for almost every $(x,t) \in E\times (0,T-h)$.
The relevant properties of Steklov averages are summarized in the following lemma. 
\begin{lemma}\label{lm:Stek}
    For any $r\ge 1$ we have
    \begin{itemize}
        \item[(i)] If $v\in L^r(E_T)$, then $[v]_h\in L^r(E_{T})$.  Moreover,  
        $
          \| [v]_h\|_{L^r(E_{T})}
            \le 
            \| v\|_{L^r(E_T)}
        $
        and
           $[v]_h\to v$ in $L^r(E_{T})$ and almost everywhere on $E_T$
           as $h\downarrow 0$.
        \item[(ii)] If $Dv\in L^r(E_T)$, then $D[v]_h=[Dv]_h\to Du$ in $L^r(E_T)$ and almost everywhere on $E_T$ as $h\downarrow 0$. 
        \item[(iii)] If  $v\in C\big(0,T;L^r(E)\big)$, then $[v]_h(\cdot ,t)\to v(\cdot, t)$ in $L^r(E)$
        as $h\downarrow 0$ for every $t\in (0,T-\eps)$ for any $\eps\in (0,T)$.
    \end{itemize}
\end{lemma}
\subsection{Existence of weak solutions}\label{S:existence}
Let us consider 
the following Cauchy-Dirichlet Problem:
\begin{equation}\label{Eq:CDP}
\left\{
\begin{array}{cl}
    \partial_tu - \Div\big(a(x,t)\big(\mu^2+|Du|^2\big)^{\frac{p-2}{2}}Du\big) = 0  & \quad \mbox{in  $E_T$,}\\[6pt]
    u =g &\quad \mbox{on $\partial E\times(0,T]$,}\\[6pt]
    u(\cdot,0)= u_o &\quad\mbox{in $E$,}
\end{array}
\right.
\end{equation}
where $\mu\in[0,1]$, $(x,t)\mapsto a(x,t)$ is a measurable function satisfying $0<C_o\le a\le C_1$, $p>1$, $u_o\in L^{2}(E,\R^k)$ and $g\in  L^p(0,T;W^{1,p}(E,\R^k))$ with $\partial_t g\in L^{p'}(0,T;W^{-1,p'}(E,\R^k))$.
We first define what we mean by a weak solution to the Cauchy-Dirichlet Problem \eqref{Eq:CDP}.

\begin{definition}\label{def:weak}
A function $u$
is termed a {\bf weak solution to the Cauchy-Dirichlet Problem \eqref{Eq:CDP}}, if
\begin{equation*}  
	u\in C\big( [0,T]; L^{2}(E,\R^k)\big) \cap g + L^p \big(0,T; W_0^{1,p} (E,\R^k)\big) 
\end{equation*}
and if the integral identity
\begin{equation}\label{Eq:weak-form-1}
\iint_{E_T} \Big[ (u_o - u)\cdot \pl_t\z+ a(x,t)\big(\mu^2+|Du|^2\big)^{\frac{p-2}{2}}Du\cdot D\z\Big]\dx\dt=0
\end{equation}
is satisfied for any test function
$
\z\in  W^{1,2}(0,T; L^{2}(E,\R^k))\cap L^p (0,T; W^{1,p}_{0}(E,\R^k))
$
such that $\z(\cdot, T)=0$. 
\end{definition}

Note that in the setting of Definition~\ref{def:weak}, the integral identity \eqref{Eq:weak-form-1} implies that $u(\cdot,0)=u_o$ in the sense of $L^2(E)$; see \cite[Proposition~4.9]{BDGLS-23} for $q=1$.

The issue of the existence of weak solutions is extensively studied in the literature, often under quite general assumptions on the structure of the diffusion part; in our case the existence
result can be deduced, for instance, from \cite[Chapter III, Proposition 4.1 and Example 4.A]{Showalter}. We will need the existence result in Sections~\ref{S:4} and \ref{sec:Schauder-for-Lipschitz}.

A related concept that dispenses with any initial/boundary data is as follows.
\begin{definition}\label{def:weak-loc}
A function $u$
is termed a {\bf weak solution} to \eqref{p-laplace-intro}, if
\begin{equation*}  
	u\in C\big( [0,T]; L^{2}(E,\R^k)\big) \cap  L^p \big(0,T; W^{1,p} (E,\R^k)\big) 
\end{equation*}
and if the integral identity
\begin{equation*}
\iint_{E_T} \Big[  u \cdot\pl_t\z - a(x,t)\big(\mu^2+|Du|^2\big)^{\frac{p-2}{2}}Du\cdot D\z\Big]\dx\dt=0
\end{equation*}
is satisfied for any test function
$
\z\in  W^{1,2}(0,T; L^{2}(E,\R^k))\cap L^p (0,T; W^{1,p}_{0}(E,\R^k))
$
such that $\z(\cdot, 0)=\z(\cdot, T)=0$. 
\end{definition}
 A notion of local solution is commonly used in the literature, cf. \cite[Chapter VIII]{DB}. We stress that such notion makes no essential difference from Definition \ref{def:weak-loc} modulo a
localization.
\section{Gradient bound for \texorpdfstring{$p$-}{p-}Laplace systems with differentiable coefficients}\label{S:grad-bound}
In this section, we provide a sup-bound for the spatial gradient of solutions to parabolic $p$-Laplace-type systems with differentiable coefficients. More precisely, let us consider weak solutions to 
\begin{align}\label{eq:diff-systems}
    \partial_t u-\Div\Big( b(x,t)\big(\mu^2+|Du|^2\big)^\frac{p-2}2Du\Big)=0
    \quad\mbox{in $E_T$,}
\end{align}
where $p>1$ and $\mu\in (0,1]$. For the coefficients $b$ we assume that the map $E\ni x\mapsto b(x,t)$ is differentiable for almost every $t\in (0,T)$, that $(0,T)\ni t\mapsto b(x,t)$ is measurable for every $x\in E$, and finally that
\begin{equation}\label{ass:b}
    \left\{
    \begin{array}{c}
        C_o\le b(x,t)\le C_1,\\[5pt]
        | \nabla b(x,t)|\le C_2,
    \end{array}
    \right.
\end{equation}  
for every $x\in E$ and almost every $t\in (0,T)$, 
with constants $0<C_o\le C_1$ and $C_2\ge 0$. 

\subsection{Energy estimates for second order derivatives}
The starting point for gradient boundedness  is an energy inequality involving second order weak derivatives of solutions. The proof can be found in~\cite[Proposition~3.2]{BDLS-Tolksdorf}. 

\begin{proposition}\label{prop:estD2u}
Let $\mu\in (0,1]$ and $p>1$, and assume that assumptions~\eqref{ass:b} are in force. 
Then there exists a constant $C=C(p,C_o,C_1,C_2)$ such that whenever $u$ is a weak solution to the parabolic system \eqref{eq:diff-systems} in the sense of Definition \ref{def:weak-loc}, for any non-negative, non-decreasing, bounded, Lipschitz function $\Phi \colon \R_{\ge 0} \to  \R_{\ge 0} $ and every parabolic cylinder $Q_{R,S}(z_o)=B_{R}(x_o)\times(t_o-S,t_o]\subset E_T$ with $R,S>0$ and every non-negative cut-off function $\zeta\in C^\infty(Q_{R,S}(z_o),\R_{\ge0})$ vanishing on $\partial_{\mathrm{par}}Q_{R,S}(z_o)$ we have 
\begin{align}\label{est:estD2u}
\nonumber
    \sup_{\tau\in(t_o-S,t_o]}&\int_{B_R(x_o)\times\{\tau\}}v\, \zeta^2\,\dx\\
    &\phantom{\le\,}+
    \iint_{Q_{R,S}(z_o)}
    \big(\mu^2+|Du|^2\big)^\frac{p-2}2|D^2u|^2 \Phi(|Du|^2)\zeta^2
    \,\dx\dt 
    \nonumber\\\nonumber
    &\le
    C\iint_{Q_{R,S}(z_o)\cap \{\Phi(|Du|^2)>0\}}
   \big(\mu^2+|Du|^2\big)^\frac{p-2}2\\
   &\qquad\qquad\qquad\qquad\cdot
    \frac{\Phi^2(|Du|^2)|Du|^2}
    {\Phi(|Du|^2) + \Phi'(|Du|^2)|Du|^2}\big(\zeta^2+|\nabla\zeta|^2\big)
    \,\dx\dt \\
    &\phantom{\le\,} +
    C\iint_{Q_{R,S}(z_o)}
    \big(\mu^2+|Du|^2\big)^\frac{p-2}2 \Phi'(|Du|^2)
    |Du|^4 \zeta^2\,\dx\dt\nonumber\\
    &\phantom{\le\,} +
    C \iint_{Q_{R,S}(z_o)}v \big|\partial_t\zeta^2\big|\,\dx\dt,\nonumber
\end{align}
where
\begin{equation}\label{def-v}
    v:=\int_0^{|Du|^2}\Phi(s)\, \ds.
\end{equation}
\end{proposition}

The next lemma essentially follows from Proposition~\ref{prop:estD2u} by choosing $\Phi(\tau)=(\delta^2+\tau)^\alpha$. However, since $\Phi$ is not a bounded Lipschitz function we have to choose a truncated version instead. 

\begin{lemma}\label{energy-est}
Let $p>1$ and $\mu\in (0,1]$, and let $u$
be a weak solution to the parabolic system \eqref{eq:diff-systems} in the sense of Definition \ref{def:weak-loc}. Moreover, assume that the hypotheses \eqref{ass:b} are in force and that for some $\alpha\ge 0$ the solution satisfies on some cylinder $ Q_{R,S}(z_o)\Subset E_T$ with $R\in(0,1]$ and $S>0$ the integrability condition 
$$
    |Du|\in L^{p+2\alpha}\big( Q_{R,S}(z_o)\big)\cap  L^{2+2\alpha}\big( Q_{R,S}(z_o)\big).
$$
Then, for every $0< r<R$, $0< s<S$, and every $\delta\in[0,1]$ we have
\begin{align*}
    \tfrac1{1+\alpha}\sup_{t\in (t_o-s,t_o]}&
     \int_{B_r(x_o)\times\{t\}} \big(\delta^2+|Du|^2\big)^{1+\alpha}\, \dx\\
     &\phantom{\le\,}+
     \iint_{Q_{r,s}(z_o)}\big(\mu^2 +|Du|^2\big)^{\frac{p-2}{2}} \big(\delta^2 +|Du|^2\big)^\alpha
     |D^2u|^2\dx\dt\\
     &\le \frac{C(1+\alpha)}{(R-r)^2}
     \iint_{Q_{R,S}(z_o)}
     \big(\mu^2 +|Du|^2\big)^{\frac{p-2}{2}} \big(\delta^2 +|Du|^2\big)^{\alpha+1}
     \,\dx\dt\\
     &\phantom{\le\,}+\frac{C}{S-s}
     \iint_{Q_{R,S}(z_o)}\big(\delta^2+|Du|^2\big)^{1+\alpha}\,\dx\dt,
\end{align*}
for a constant $C=C(p, C_o,C_1,C_2)$.
\end{lemma}

\begin{proof}
We first consider the case $\delta>0$. In Proposition \ref{prop:estD2u} we choose 
\begin{equation*}
    \Phi_m(\tau):=\big(\delta^2+\min\{\tau ,m\}\big)^\alpha\quad\mbox{for $\tau\in\R_{\ge 0}$.}
\end{equation*}
Then, $\Phi_m$ is positive, non-decreasing, bounded and Lipschitz continuous for every $m\ge0$. 
Moreover, $\Phi_m(\tau)\uparrow \Phi(\tau):=(\delta^2+\tau)^\alpha$ in the limit $m\to\infty$. Therefore, Proposition~\ref{prop:estD2u} is applicable with $\Phi_m$ instead of $\Phi$, and with a proper choice of $\z$ the energy estimate implies
\begin{align*}
    &\sup_{t\in(t_o-s,t_o]}\int_{B_r(x_o)\times\{t\}}v_m \,\dx +
    \iint_{Q_{r,s}(z_o)}
    \big(\mu^2+|Du|^2\big)^\frac{p-2}2|D^2u|^2 \Phi_m(|Du|^2)
    \,\dx\dt 
    \nonumber\\
    &\quad\le
    \frac{C}{(R-r)^2}\iint_{Q_{R,S}(z_o)}
   \big(\mu^2+|Du|^2\big)^\frac{p-2}2
    \frac{\Phi_m^2(|Du|^2)|Du|^2}
    {\Phi_m(|Du|^2) + \Phi_m'(|Du|^2)|Du|^2}
    \,\dx\dt \nonumber\\
    &\quad\phantom{\le\,} +
    C\iint_{Q_{R,S}(z_o)}
    \big(\mu^2+|Du|^2\big)^\frac{p-2}2 \Phi_m'(|Du|^2)
    |Du|^4 \,\dx\dt \nonumber\\
    &\quad\phantom{\le\,}+
    \frac{C}{S-s} \iint_{Q_{R,S}(z_o)}v_m \,\dx\dt,
\end{align*}
where
\begin{align*}
        v_m:=\int_0^{|Du|^2}\Phi_m(\tau)\,\dtau.
\end{align*}
We note that $v_m\uparrow v$ a.e. in the limit $m\to\infty$, where $v$ is defined analogously with $\Phi$ instead of $\Phi_m$. 
Using either Fatou's lemma or monotone convergence theorem, it follows that the integrals on the left-hand side in the limit $m\to\infty$ converge to the corresponding integrals with $\Phi$ instead of $\Phi_m$. For the last integral on the right-hand side this holds true again by monotone convergence theorem. For the second integral on the right, we have
\begin{align*}
    \iint_{Q_{R,S}(z_o)}&
    \big(\mu^2+|Du|^2\big)^\frac{p-2}2 \Phi_m'(|Du|^2)
    |Du|^4 \,\dx\dt\\
    &=
    \alpha
    \iint_{Q_{R,S}(z_o)\cap\{|Du|^2\le m\}}
    \underbrace{
    \big(\mu^2+|Du|^2\big)^\frac{p-2}2 \big(\delta^2+|Du|^2\big)^{\alpha-1}
    |Du|^4}_{\le (1+|Du|^2)^{\frac{p}2+\alpha}} \,\dx\dt.
\end{align*}
Since $ |Du|\in L^{p+2\alpha}\big( Q_{R,S}(z_o)\big)$, we conclude with the dominated convergence theorem that also these integrals converge in the limit $m\to\infty$. Therefore, it remains to treat the first integral on the right-hand side. Here we have
\begin{align*}
    \big(\mu^2+|Du|^2\big)^\frac{p-2}2&
    \frac{\Phi_m^2(|Du|^2)|Du|^2}
    {\Phi_m(|Du|^2) + \Phi_m'(|Du|^2)|Du|^2}\\
    &\le
    \big(\mu^2+|Du|^2\big)^\frac{p-2}2\big(\delta^2+|Du|^2\big)^\alpha |Du|^2\\
    &\le 
    \big(1+|Du|^2\big)^{\frac{p}2+\alpha}\in L^1\big( Q_{R,S}(z_o)).
\end{align*}
Therefore, we can pass to the limit $m\to\infty$ in these integrals as well. Finally, the function $v$ takes the form
\begin{equation*}
    v=\tfrac1{1+\alpha} \big(\delta^2+|Du|^2\big)^{1+\alpha}
    - 
    \tfrac1{1+\alpha}\delta^{2(1+\alpha)}.
\end{equation*}
We move the integral over $B_R(x_o)$ of the negative part to the right-hand side. Furthermore we estimate this integral simply by
\begin{align*}
     \tfrac1{1+\alpha}\delta^{2(1+\alpha)}|B_R|
     &=
     \tfrac1{(1+\alpha)S}\delta^{2(1+\alpha)}|Q_{R,S}|\le \frac1{S-s}\iint_{Q_{R,S}(z_o)}
     \big(\delta^2+|Du|^2\big)^{1+\alpha}\,\dx\dt.
\end{align*}
This completes the proof in the case $\delta>0$. If $\delta=0$ we proceed as before with some $\tilde\delta>0$ instead of $\delta$ and then let $\tilde\delta\downarrow 0$. 
\end{proof}

\subsection{\texorpdfstring{$L^m$-}{Lm-}gradient estimates in the sub-quadratic case}
Our goal is to implement a Moser iteration procedure to prove the local boundedness of the gradient for any $p>1$. However, in the so-called sub-critical case $p\le\frac{2N}{N+2}$, this is only possible for solutions that satisfy the additional regularity property $|Du|\in L^m_{\mathrm{loc}}(E_T)$ for a sufficiently large power $m>p$. The latter can be shown for {\it bounded} solutions using ideas from \cite[Chapter VIII, Lemma 4.1]{DB}, see also \cite{DiBenedetto-Friedman3, Choe:1991}. In what follows, we present the proof for all $1<p\le2$ in a unified fashion.

\begin{lemma}\label{lem:Lq-est}
Let $\mu\in(0,1]$ and $1<p\le 2$. Then, for any locally bounded weak solution to the parabolic system \eqref{eq:diff-systems} in the sense of Definition \ref{def:weak-loc}, we have $|Du|\in L^m_{\mathrm{loc}}(E_T)$ for every $m>1$.
Moreover, for every $m> 3$, 
we have for every pair of concentric cylinders $Q_{r,s}(z_o)\subset Q_{R,S}(z_o)\Subset E_T$ with $0<r<R\le1$ and $0<s<S$
the quantitative estimate
\begin{align}\label{Lq-est}\nonumber
     \iint_{Q_{r,s}(z_o)}&
    |Du|^m \,\dx\dt\\
    &\le
    \max\bigg\{
    C\epsilon^\frac{(p-2)(m-p)}{p}\bigg[
    \frac{m^\frac32\boldsymbol\omega}{R-r} 
    +
    \Big(
    \frac{m\boldsymbol\omega}{\sqrt{S-s}}
    \Big)^\frac{2}{p}\bigg]^{m-p}\iint_{Q_{R,S}(z_o)}|Du|^p\,\dx\dt, \\\nonumber
    &\phantom{\le
    \max\bigg\{\,\,}
    \epsilon^m\mu^m|Q_{R,S}|
    \bigg\}
\end{align}
for every $\epsilon\in(0,1]$, with a constant $C=C(N,p,C_o,C_1, C_2)$. In  \eqref{Lq-est} we denoted 
$$
    \boldsymbol\omega 
    :=
    \sup_{t\in (t_o-S, t_o]}\sup_{x\in B_R(x_o)}\big| u(x,t)-(u)_{x_o,R}(t)\big|,
$$
and 
$$
    (u)_{x_o,R}(t)
    =\bint_{B_R(x_o)}u(x,t)\,\dx.
$$
\end{lemma}

\begin{proof} 
We omit the base point $z_o$ from the notation for simplicity.
To start with, we show 
$$
    |Du|\in L^2_{\mathrm{loc}}(E_T).
$$
This can be achieved by adapting
the difference quotient technique as illustrated in the appendix of~\cite{Choe:1991}. We briefly sketch the argument by giving the formal computations. We consider two concentric cylinders $Q_{r_1,s_1}\subset Q_{r_2,s_2}$ with $r\le r_1<r_2\le R$ and $s\le s_1<s_2\le S$, and  define $\varrho:=\frac12 (r_1+r_2)$. Moreover, we let $\zeta\in C^1_0\big(B_{\varrho}, [0,1]\big) $ be a cut-off function in space such that $\zeta =1$ on $B_{r_1}$ and $|\nabla\zeta|\le \frac{2}{\varrho-r_1}=\frac{4}{r_2-r_1}$.
First, we note that integration by parts yields 
\begin{align}\label{Lp->L2}\nonumber
    \boldsymbol\phi_{r_1,s_1}
    &:=
    \iint_{Q_{r_1,s_1}}|Du|^2\,\dx\dt
    \le
    \iint_{Q_{\rho, s_1}}|Du|^2\zeta\,\dx\dt\\\nonumber
    &=
     \iint_{Q_{\rho, s_1}}D[u-(u)_R(t)]\cdot Du\,\zeta\,\dx\dt
    \\\nonumber
    &=
    -\iint_{Q_{\rho, s_1}}
    [u-(u)_R(t)]\cdot \Delta u\,\zeta\,\dx\dt
    -
    \iint_{Q_{\rho, s_1}}[u-(u)_R(t)]\cdot Du\nabla\zeta\,\dx\dt\\
    &\le  
    \sqrt{N}\boldsymbol \omega\,
    \iint_{Q_{\rho, s_1}}|D^2u|\,\dx\dt
    +
    \frac{4\boldsymbol \omega}{r_2-r_1}
    \iint_{Q_{\rho, s_1}}|Du|\,\dx\dt.
\end{align}

For the estimate of the first integral on the right-hand side, we apply H\"older's inequality with  exponents $2$, $\frac{2p}{2-p}$, and $\frac{p}{p-1}$
to obtain
\begin{align*}
    \iint_{Q_{\rho, s_1}}|D^2u|\,\dx\dt
    &=
    \iint_{Q_{\rho, s_1}}
    \big(\mu^2+|Du|^2\big)^{\frac{p-2}4}|D^2u|
    \big(\mu^2+|Du|^2\big)^{\frac{2-p}4}
    \,\dx\dt\\
    &\le
    \mathbf I_{\rho,s_1}^\frac12
    \bigg[
    \iint_{Q_{\rho, s_1}}
    \big(\mu^2+|Du|^2\big)^{\frac{p}2}
    \,\dx\dt\bigg]^{\frac{2-p}{2p}}|Q_{\rho, s_1}|^\frac{p-1}{p}\\
    &\le
    \mathbf I_{\rho,s_1}^\frac12
    \bigg[\underbrace{
    \iint_{Q_{R,S}}
    \big(\mu^2+|Du|^2\big)^{\frac{p}2}
    \,\dx\dt}_{=:\mathsf E}\bigg]^{\frac{2-p}{2p}}|Q_{R, S}|^\frac{p-1}{p},
\end{align*}
where
\begin{equation*}
    \mathbf I_{\rho,s_1}:=
    \iint_{Q_{\rho, s_1}}
    \big(\mu^2+|Du|^2\big)^{\frac{p-2}2}|D^2u|^2 \,\dx\dt.
\end{equation*}
To estimate $\mathbf I_{\rho,s_1}$ we apply the energy estimate from Lemma~\ref{energy-est}  with $\alpha=0$, $\delta=0$ and with $(r,s)$, $(R,S)$ replaced by $(\rho,s_1)$, $(r_2,s_2)$. This yields
\begin{align*}
    \mathbf I_{\rho,s_1}
    &\le
    \frac{C}{(r_2-r_1)^2}
    \underbrace{
    \iint_{Q_{r_2, s_2}}
    \big(\mu^2+|Du|^2\big)^{\frac{p-2}2}|Du|^2
    \,\dx\dt}_{\le\,\mathsf E}
    +
    \frac{C\, \boldsymbol\phi_{r_2,s_2}}{s_2-s_1}\\
    &\le
    \frac{C\,\mathsf E}{(r_2-r_1)^2}
    +
    \frac{C\, \boldsymbol\phi_{r_2,s_2}}{s_2-s_1}.
\end{align*}
We use this inequality to bound the right-hand side in the estimate of $|D^2u|$ and obtain
\begin{align*}
    \iint_{Q_{\rho, s_1}}|D^2u|\,\dx\dt
    &\le
    C\bigg[
    \frac{\mathsf E}{(r_2-r_1)^2}
    +
    \frac{\boldsymbol\phi_{r_2,s_2}}{s_2-s_1}
    \bigg]^\frac12\mathsf E^\frac{2-p}{2p}
    |Q_{R, S}|^{1-\frac{1}{p}}\\ 
    &\le
    C
    \bigg[
    \frac{\mathsf E^\frac1p}{r_2-r_1}
    +
    \frac{\boldsymbol\phi_{r_2,s_s}^\frac12
    \mathsf E^\frac{2-p}{2p}}{(s_2-s_1)^\frac12}
    \bigg]|Q_{R,S}|^{1-\frac{1}{p}}.
\end{align*}
To estimate the first term on the right-hand side of \eqref{Lp->L2}, we multiply the above inequality by $C\boldsymbol \omega$ and apply Young's inequality to the individual terms. In this way we get
\begin{align*}
   C \boldsymbol \omega
     \iint_{Q_{\rho, s_1}}|D^2u|\,\dx\dt
     &\le
     \tfrac12 \boldsymbol\phi_{r_2,s_2}
     +
     C
     \bigg[
     \frac{\boldsymbol \omega\mathsf E^\frac{1}{p}}{r_2-r_1}
     |Q_{R,S}|^{1-\frac{1}{p}} +
     \frac{\boldsymbol \omega^2\mathsf E^\frac{2-p}{p}}{s_2-s_1}|Q_{R,S}|^{2(1-\frac{1}{p})}
     \bigg]\\
     &\le 
     \tfrac12 \boldsymbol\phi_{r_2,s_2}
     +
     \frac{C \boldsymbol \omega|Q_{R,S}|}{r_2-r_1}
     \bigg[\biint_{Q_{R,S}}\big(
     \mu^2 +|Du|^2\big)^\frac{p}{2} \,\dx\dt\bigg]^\frac{1}{p} \\
     &\qquad\quad\quad\,\,\,\, +
     \frac{C\boldsymbol \omega^2|Q_{R,S}|}{s_2-s_1}\bigg[\biint_{Q_{R,S}}\big(
     \mu^2 +|Du|^2\big)^\frac{p}{2} \,\dx\dt\bigg]^\frac{2-p}{p}
     .
\end{align*}
Together with the estimate
\begin{equation*}
     \frac{4\boldsymbol \omega}{r_2-r_1}
    \iint_{Q_{\rho, s_1}}|Du|\,\dx\dt
    \le
    \frac{4\boldsymbol \omega|Q_{R,S}|}{r_2-r_1}
    \bigg[\biint_{Q_{R,S}}\big(
     \mu^2 +|Du|^2\big)^\frac{p}{2} \,\dx\dt\bigg]^\frac{1}{p},
\end{equation*}
this allows us to bound the right-hand side of~\eqref{Lp->L2}, so that
\begin{align*}
    \boldsymbol\phi_{r_1,s_1}
    &\le
     \tfrac12 \boldsymbol\phi_{r_2,s_2}
     +
     C\frac{\boldsymbol \omega^2|Q_{R,S}|}{s_2-s_1}
    \bigg[\biint_{Q_{R,S}}\big(
     \mu^2 +|Du|^2\big)^\frac{p}{2}\dx\dt\bigg]^\frac{2-p}{p} \\ 
     &\qquad\qquad\,\,\; +
     C\frac{\boldsymbol \omega  |Q_{R,S}|}{r_2-r_1}
     \bigg[\biint_{Q_{R,S}}\big(
     \mu^2 +|Du|^2\big)^\frac{p}{2}\dx\dt\bigg]^\frac{1}{p} .
\end{align*}
In view of the iteration Lemma \ref{lem:tech}, we can re-absorb the first term on the
right-hand side and obtain
\begin{align*}
    \iint_{Q_{r,s}}|Du|^2 \,\dx\dt
    &\le
     \frac{C \boldsymbol \omega^2|Q_{R,S}|}{S-s}\bigg[\biint_{Q_{R,S}}\big(
     \mu^2 +|Du|^2\big)^\frac{p}{2} \,\dx\dt\bigg]^\frac{2-p}{p}\\
     &\phantom{\le\,}
     +
     \frac{C\boldsymbol \omega|Q_{R,S}|}{R-r}
     \bigg[\biint_{Q_{R,S}}\big(
     \mu^2 +|Du|^2\big)^\frac{p}{2} \,\dx\dt\bigg]^\frac{1}{p}
     ,
\end{align*}
where the right-hand side is finite by assumption.  Therefore, we have shown (modulo a suitable difference quotient argument) that $|Du|\in L^2_{\mathrm{loc}}(E_T)$.

Now, let $m\ge p+1>2$. We start by assuming that
\begin{equation*}
    |Du|\in L^{m-p+1}_{\mathrm{loc}}(E_T),
\end{equation*}
which holds true for $m=p+1$ since $|Du|\in 
L^{2}_{\mathrm{loc}}(E_T)$. 

As in the first part of the proof, consider concentric cylinders $Q_{r,s}\subset Q_{R,S}$ with $0<r<R$ and $0<s<S$, and let  $\varrho:=\frac12 (R+r)$ and $\theta :=\frac12 (S+s)$. Moreover, let $\zeta\in C^1_0(B_{\varrho}, [0,1]) $ be a cut-off function in space such that $\zeta =1$ on $B_r$ and $|\nabla\zeta|\le \frac{2}{\varrho-r}=\frac{4}{R-r}$. 
Similarly to the first part of the proof, we calculate by integrating by parts (modulo the difference quotient technique)
\begin{align*}
    \iint_{Q_{\varrho,\theta}}&
    |Du|^m\zeta \,\dx\dt\\
    &= 
    \iint_{Q_{\varrho,\theta}} D_\gamma \big[u-u_R(t)\big]\cdot D_\gamma u
    |Du|^{m-2}\zeta \,\dx\dt \\
    &= 
    - \iint_{Q_{\varrho,\theta}}
    \underbrace{\big[u-u_R(t)\big]}_{|\cdot|\,\le\, \boldsymbol\omega}\cdot D_\gamma
    \big[  D_\gamma u |Du|^{m-2}\zeta\big] \,\dx\dt\\
    &\le 
    \boldsymbol\omega
    \iint_{Q_{\varrho,\theta}}
    \big|D_\gamma\big[D_\gamma u|Du|^{m-2}\big]\big| \,\dx\dt
    +
    \frac{4\boldsymbol\omega}{R-r}
    \iint_{Q_{\varrho,\theta}}
    |Du|^{m-1} \,\dx\dt.
\end{align*}
To estimate the first integral on the right-hand side, we observe that 
\begin{align*}
    \big| D_\gamma\big[D_\gamma u |Du|^{m-2}\big]\big|
    &=\Big|D_\gamma D_\gamma u |Du|^{m-2}
    +
    D_\gamma u D_\gamma |Du|^{m-2}\Big|\\
    &\le 
    |Du|^{m-2}\bigg| 
    D_\gamma D_\gamma u
    +(m-2)\frac{D_\gamma u D_\gamma D_\beta u\cdot D_\beta u}{|Du|^2}\bigg|\\
    &\le  
    |Du|^{m-2}
    \Big[\sqrt{N}|D^2u| + (m-2)|D^2u|\Big]\\
    &\le 
    \sqrt{N} (m-1)|Du|^{m-2} |D^2u|.
\end{align*}
Using H\"older's inequality we have
\begin{align}\label{est-I1-I2}\nonumber
    \iint_{Q_{\varrho,\theta}}&
    |Du|^{m-2} |D^2u|\,\dx\dt\\\nonumber
    &=
    \iint_{Q_{\varrho,\theta}}
    \big(\mu^2+|Du|^2\big)^\frac{p-2}4
    |Du|^\frac{m-p-1}{2}|D^2u|
    \big(\mu^2+|Du|^2\big)^\frac{2-p}4|Du|^\frac{m+p-3}2\,\dx\dt\\
    &\le
    \mathbf{I}_1^\frac12 \,
    \mathbf{I}_2^\frac12,
\end{align}
where we abbreviated
\begin{align*}
    \mathbf{I}_1
    &:=
    \iint_{Q_{\varrho,\theta}}
    \big(\mu^2+|Du|^2\big)^\frac{p-2}2
    |Du|^{m-p-1}|D^2u|^2 \,\dx\dt,
\end{align*}
and 
\begin{align*}
    \mathbf{I}_2
    &:=
    \iint_{Q_{\varrho,\theta}}
    \big(\mu^2+|Du|^2\big)^\frac{2-p}2|Du|^{m+p-3} \,\dx\dt.
\end{align*}
Note that, $m+p-3\ge 2(p-1)>0$. Therefore, we have $\mathbf{I}_2<\infty$, due to the assumption $|Du|\in L^{m-p+1}_{\mathrm{loc}}(E_T)$ and since $m-1\le m-p+1$. The first integral $\mathbf I_1$ can be controlled with the help of the energy inequality from Lemma~\ref{energy-est}. In fact, taking $\alpha =\tfrac12 (m-p-1)$ we have $p+2\alpha =m-1$, and $2+2\alpha =m-p+1$, so that the required integrability is fulfilled. Applying the energy inequality on $Q_{\rho,\theta}$, $Q_{R,S}$  and with $\delta=0$ we have
\begin{align}\label{est-I1}\nonumber
    \mathbf I_1
    &\le
    \frac{Cm}{(R-r)^2}
     \iint_{Q_{R,S}}\big(\mu^2 +|Du|^2\big)^\frac{p-2}2
    |Du|^{m-p+1}
    \,\dx\dt\\\nonumber
    &\phantom{\le\,}
    +
    \frac{C}{S-s}
     \iint_{Q_{R,S}}
    |Du|^{m-p+1}\,\dx\dt\\
    &\le
    \frac{Cm}{(R-r)^2}
    \iint_{Q_{R,S}}
    |Du|^{m-1}
    \,\dx\dt
    +
    \frac{C}{S-s}
     \iint_{Q_{R,S}}
    |Du|^{m-p+1}\,\dx\dt.
\end{align}
For the second integral, we estimate
\begin{align*}
    \mathbf I_2
    &\le \iint_{Q_{R,S}}
    \big(\mu^2+|Du|^2\big)^\frac{2-p}2|Du|^{m+p-3} \,\dx\dt\\
    &\le
    \mu^{2-p}
    \iint_{Q_{R,S}}|Du|^{m+p-3} \,\dx\dt
    +
    \iint_{Q_{R,S}}|Du|^{m-1} \,\dx\dt.
\end{align*}
Inserting this above and noting that $\zeta=1$ on $Q_{r,s}$ we arrive at
\begin{align}\label{start:quant}\nonumber
    \iint_{Q_{r,s}}
    &|Du|^m \,\dx\dt\\ \nonumber
    &\le 
    Cm
    \boldsymbol\omega
    \iint_{Q_{\varrho,\theta}}
    |Du|^{m-2}|D^2u| \,\dx\dt+
    \frac{4\boldsymbol\omega}{R-r}
    \iint_{Q_{\varrho,\theta}}
    |Du|^{m-1} \,\dx\dt\\ \nonumber
    &\le
    Cm
    \boldsymbol\omega
    \Bigg[
    \frac{m}{(R-r)^2}
    \iint_{Q_{R,S}}
    |Du|^{m-1}
    \,\dx\dt+
    \frac{1}{S-s}
     \iint_{Q_{R,S}}
    |Du|^{m-p+1}\,\dx\dt
    \Bigg]^\frac12\\ 
    &\qquad\qquad\qquad \cdot\Bigg[\mu^{2-p}
    \iint_{Q_{R,S}}|Du|^{m+p-3} \,\dx\dt
    +
    \iint_{Q_{R,S}}|Du|^{m-1} \,\dx\dt\Bigg]^\frac12\\ \nonumber
    &\phantom{\le\,}
    +
    \frac{4\boldsymbol\omega}{R-r}
    \iint_{Q_{R,S}}
    |Du|^{m-1} \,\dx\dt.
\end{align}
To obtain the last line we used the estimates \eqref{est-I1-I2} and \eqref{est-I1}. Observe that the right-hand side is finite by the assumption $|Du|\in L^{m-p+1}_{\mathrm{loc}}(E_T) \cap L^{m-1}_{\mathrm{loc}}(E_T)$. By means of an induction argument it is now easy to show that $|Du|\in L^m_{\rm loc} (E_T)$ for every $m>1$. Indeed, starting with 
$m=p+1$, we have  $m-1=p$ and $m-p+1=2$. Therefore, the assumption  $|Du|\in L^{p}(Q_{R,S}) \cap L^{2}(Q_{R,S})$ is fulfilled as already mentioned. This implies that $|Du|\in L^{p+1}(Q_{r,s})$, and by a covering argument we have $|Du|\in L^{p+1}_{\rm loc}(E_T)$. We now define $m_\ell := p+1+\ell (p-1)$. Then, $m_{\ell+1}-p+1 =m_{\ell}$ and $m_{\ell +1}-1=m_{\ell} +p-2\le m_{\ell} $. Therefore, the requirement $|Du|\in L^{m_{\ell+1}-p+1}(Q_{R,S}) \cap L^{m_{\ell+1}-1}(Q_{R,S})$ is fulfilled.  This allows us to conclude that $|Du|\in L^{m_{\ell+1}}_{\rm loc}(E_T)$. Since $m_\ell\uparrow\infty$, the assertion, i.e.~that $|Du|\in L^m_{\rm loc}(E_T)$ for any $m>1$, is proven, and it remains to establish the quantitative $L^m$-bound.

To obtain the quantitative $L^m$-estimate \eqref{Lq-est}, we further estimate various integrals on the right-hand side of \eqref{start:quant} by interpolating between $L^m$ and $L^p$ norms. Indeed, by splitting the exponent $m-1$ in the form 
\begin{equation*}
    m-1
    =
    \frac{m(m-p-1)}{m-p}+\frac{p}{m-p},
\end{equation*}
we apply H\"older's inequality with exponents $\frac{m-p}{m-p-1}$ and $m-p$ to obtain
\begin{align*}
    \iint_{Q_{R,S}}|Du|^{m-1}\,\dx\dt
    &=
    \iint_{Q_{R,S}}
    |Du|^\frac{m(m-p-1)}{m-p} |Du|^\frac{p}{m-p} \,\dx\dt \nonumber\\
    &\le
    \bigg[
    \iint_{Q_{R,S}}
    |Du|^{m} \,\dx\dt
    \bigg]^\frac{m-p-1}{m-p}
    \bigg[
    \iint_{Q_{R,S}}
    |Du|^{p} \,\dx\dt
    \bigg]^\frac{1}{m-p} \nonumber\\
    &=\mathfrak D_m^\frac{m-p-1}{m-p}\mathfrak D_p^\frac{1}{m-p},
\end{align*}
with the obvious meaning of $\mathfrak D_m$ and $\mathfrak D_p$.
Similarly, splitting
\begin{equation*}
    m-p+1
    =
    \frac{m(m-2p+1)}{m-p}+\frac{p(p-1)}{m-p},
\end{equation*}
we find using H\"older's inequality with exponents $\frac{m-p}{m-2p+1}$ and $\frac{m-p}{p-1}$ that
\begin{align*}
    \iint_{Q_{R,S}}|Du|^{m-p+1}\,\dx\dt
    &=
    \iint_{Q_{R,S}}
    |Du|^\frac{m(m-2p+1)}{m-p} |Du|^\frac{p(p-1)}{m-p} \,\dx\dt\\
    &\le
   \bigg[
    \iint_{Q_{R,S}}
    |Du|^{m} \,\dx\dt
    \bigg]^\frac{m-2p+1}{m-p}
    \bigg[
    \iint_{Q_{R,S}}
    |Du|^{p} \,\dx\dt
    \bigg]^\frac{p-1}{m-p}\\
    &=
    \mathfrak D_m^\frac{m-2p+1}{m-p}
    \mathfrak D_p^\frac{p-1}{m-p}.
\end{align*}
Moreover, since
\begin{equation*}
    \frac{(m+p-3)(m-p)-p}{m(m-p)}
    +
    \frac{1}{m-p}
    +
    \frac{2-p}{m}=1
\end{equation*}
we have 
\begin{align*}
    \iint_{Q_{R,S}}&
    |Du|^{m+p-3}\,\dx\dt\\
    &=
    \iint_{Q_{R,S}}
    |Du|^\frac{(m+p-3)(m-p)-p}{m-p}
    |Du|^\frac{p}{m-p}\,\dx\dt\\
    &\le
    \bigg[\iint_{Q_{R,S}}
    |Du|^m \,\dx\dt\bigg]^ \frac{(m+p-3)(m-p)-p}{m(m-p)}
    \bigg[\iint_{Q_{R,S}}
    |Du|^p\,\dx\dt\bigg]^\frac{1}{m-p}
    |Q_{R,S}|^\frac{2-p}{m} \\
    &=
    \mathfrak D_m^\frac{(m+p-3)(m-p)-p}{m(m-p)}
    \mathfrak D_p^\frac{1}{m-p}
    |Q_{R,S}|^\frac{2-p}{m}.
\end{align*}
This is the step in the proof where we need the assumption $m>3$. 
Without any loss of generality, we can assume that 
\begin{equation}\label{est:mu-Du^m}
    \epsilon^m\mu^m<\frac{1}{|Q_{R,S}|}\iint_{Q_{r,s}}|Du|^m\,\dx\dt
\end{equation}
holds true, because otherwise \eqref{Lq-est} trivially holds.
This allows us to estimate
\begin{align*}
    \mu^{2-p} \iint_{Q_{R,S}}
    |Du|^{m+p-3}\,\dx\dt 
    &\le
    \epsilon^{p-2}
    \mathfrak D_m^{\frac{(m+p-3)(m-p)-p}{m(m-p)}+\frac{2-p}{m}}
    \mathfrak D_p^\frac{1}{m-p}\\
    &=
    \epsilon^{p-2}
    \mathfrak D_m
    ^{\frac{m-p-1}{m-p}}
    \mathfrak D_p^\frac{1}{m-p}.
\end{align*}
Hence, substituting these estimates in \eqref{start:quant} and applying Young's inequality, we get 
\begin{align*}
    \iint_{Q_{r,s}}&
    |Du|^m \,\dx\dt\\
    &\le
    Cm \epsilon^\frac{p-2}2
    \boldsymbol\omega
    \Bigg[
    \frac{m\mathfrak D_m^\frac{m-p-1}{m-p}\mathfrak D_p^\frac{1}{m-p}}{(R-r)^2} 
    +
    \frac{\mathfrak D_m^\frac{m-2p+1}{m-p}
    \mathfrak D_p^\frac{p-1}{m-p}}{S-s}
    \Bigg]^\frac12\cdot
    \Big[
    \mathfrak D_m^\frac{m-p-1}{m-p}\mathfrak D_p^\frac{1}{m-p}
    \Big]^\frac12\\
    &\phantom{\le\,}
    +
    \frac{4\boldsymbol\omega}{R-r}
    \mathfrak D_m^\frac{m-p-1}{m-p}\mathfrak D_p^\frac{1}{m-p}\\
    &\le
    \frac{Cm^\frac32\epsilon^\frac{p-2}2\boldsymbol\omega }{R-r}\mathfrak D_m^\frac{m-p-1}{m-p}\mathfrak D_p^\frac{1}{m-p}
    +
    \frac{Cm\epsilon^\frac{p-2}2\boldsymbol\omega}{\sqrt{S-s}}
    \mathfrak D_m^\frac{2m-3p}{2(m-p)}\mathfrak D_p^\frac{p}{2(m-p)}\\
    &\le
    \tfrac12 \mathfrak D_m
    +C\bigg[
    \frac{m^\frac32\epsilon^\frac{p-2}2\boldsymbol\omega}{R-r}  
    +
    \Big(
    \frac{m\epsilon^\frac{p-2}2\boldsymbol\omega}{\sqrt{S-s}}
    \Big)^\frac{2}{p}\bigg]^{m-p}\mathfrak D_p\\
    &\le
    \tfrac12 \iint_{Q_{R,S}}|Du|^m\,\dx\dt\\
     &\phantom{\le\,}
    +C\epsilon^\frac{(p-2)(m-p)}{p}\bigg[
    \frac{m^\frac32\boldsymbol\omega}{R-r} 
    +
    \Big(
    \frac{m\boldsymbol\omega}{\sqrt{S-s}}
    \Big)^\frac{2}{p}\bigg]^{m-p}\iint_{Q_{R,S}}|Du|^p\,\dx\dt.
\end{align*}
We apply this inequality on the cylinders $Q_{\sigma,\tau}$, $Q_{\rho,\theta}$
with $r\le \sigma <\rho\le R$ and $s\le\tau<\theta\le S$. To this end, we must ensure that \eqref{est:mu-Du^m} is also fulfilled on these cylinders, i.e.~that
\begin{equation*}
    \epsilon^m\mu^m|Q_{\rho,\theta}| 
    < 
    \iint_{Q_{\sigma,\tau}}|Du|^m\,\dx\dt
\end{equation*}
holds true. However, this is implied by
\eqref{est:mu-Du^m}. 
Indeed, we have
\begin{align*}
    \epsilon^m\mu^m|Q_{\rho,\theta}| 
    &\le
    \epsilon^m
    \mu^m |Q_{R,S}| 
    < \iint_{Q_{r,s}}|Du|^m\,\dx\dt
    \le
    \iint_{Q_{\sigma,\tau}}|Du|^m\,\dx\dt.
\end{align*}
Therefore, we get
\begin{align*}
    \iint_{Q_{\sigma,\tau}}
    |Du|^m \,\dx\dt
    &\le
    \tfrac12 \iint_{Q_{\rho,\theta}}|Du|^m\,\dx\dt\\
    &\phantom{\le}+C\epsilon^\frac{(p-2)(m-p)}{p}\bigg[
    \frac{m^\frac32\boldsymbol\omega}{\rho-\sigma} 
    +
    \Big(
    \frac{m\boldsymbol\omega}{\sqrt{\theta-\tau}}
    \Big)^\frac{2}{p}\bigg]^{m-p}\iint_{Q_{\rho,\theta}}|Du|^p\,\dx\dt.
\end{align*}
The iteration Lemma \ref{lem:tech} implies
\begin{align*}
    \iint_{Q_{r,s}}
    |Du|^m \,\dx\dt
    &\le
    C\epsilon^\frac{(p-2)(m-p)}{p}\bigg[
    \frac{m^\frac32\boldsymbol\omega}{R-r} 
    +
    \Big(
    \frac{m\boldsymbol\omega}{\sqrt{S-s}}
    \Big)^\frac{2}{p}\bigg]^{m-p}\iint_{Q_{R,S}}|Du|^p\,\dx\dt.
\end{align*}
Together with \eqref{est:mu-Du^m}  this results in
\eqref{Lq-est} and concludes the proof.
\end{proof}

\subsection{Qualitative gradient bound in the sub-quadratic case}
The next result upgrades the previous higher integrability of the gradient and shows that locally bounded weak solutions have bounded gradient. Moser's iteration scheme is employed here.
\begin{proposition}\label{lem:L-infty-est-qual-p<2}
Let $\mu\in(0,1]$ and $1<p\le 2$. Then, for any locally bounded weak solution $u$ to the parabolic system \eqref{eq:diff-systems} in the sense of Definition \ref{def:weak-loc} with assumption \eqref{ass:b}, we have $$
    |Du|\in L^\infty_{\mathrm{loc}}(E_T).
$$

\end{proposition}

\begin{proof}
The starting point is the energy estimate from Lemma \ref{energy-est} applied with $\delta=\mu$ on standard cylinders $Q_R(z_o):=Q_{R,R^2}(z_o)$, i.e.~for $0<r<R$ such that $Q_R(z_o)\Subset E_T$. For simplicity we take $z_o=(0,0)$; we have
\begin{align}\label{energy-intrinsic}\nonumber
    \tfrac1{1+\alpha}\sup_{t\in (-r^2,0]}&
     \int_{B_r\times\{t\}} \big(\mu^2+|Du|^2\big)^{1+\alpha}\, \dx +
     \iint_{Q_{r}}\big(\mu^2 +|Du|^2\big)^{\frac{p-2}{2}+\alpha}
     |D^2u|^2\dx\dt\nonumber\\
     &\le \frac{C(1+\alpha)}{(R-r)^2}
     \iint_{Q_{R}}
     \big(\mu^2 +|Du|^2\big)^{\frac{p}{2}+\alpha}
     \,\dx\dt\nonumber\\
     &\phantom{\le\,}+\frac{C}{(R-r)^2}
     \iint_{Q_{R}}\big(\mu^2+|Du|^2\big)^{1+\alpha}\,\dx\dt.
\end{align}
Since $1<p\le 2$, the dominant term on the right-hand side is the integral with exponent $1+\alpha$, whereas the other term is controlled via
\begin{align*}
    \big(\mu^2 +|Du|^2\big)^{\frac{p}2+\alpha}
    &
    \le \big(\mu^2 +|Du|^2\big)^{1+\alpha} +1.
\end{align*}
Hence, by Kato's inequality, for any $\alpha\ge 0$ we 
have 
\begin{align*}\nonumber
    \tfrac1{1+\alpha}
    \sup_{t\in (-r^2,0]}
     \int_{B_r\times\{t\}}& \big(\mu^2+|Du|^2\big)^{1+\alpha}\, \dx+
     \iint_{Q_{r}}\big(\mu^2 +|Du|^2\big)^{\frac{p-2}{2}+\alpha}
     \big|\nabla |Du|\big|^2\dx\dt\\
     &\le \frac{C(1+\alpha)}{(R-r)^2}
     \iint_{Q_{R }}
     \Big[ \big(\mu^2+|Du|^2\big)^{1+\alpha}+1\Big]\,\dx\dt,
\end{align*}
where $C=C(p, C_o,C_1,C_2)$. Introduce
$$
    w:=  \big(\mu^2+|Du|^2\big)^\frac{p+2\alpha}{4}
$$
and rewrite the preceding estimate as
\begin{align}\label{est:energy-w}\nonumber
     \tfrac1{1+\alpha}
     \sup_{t\in (-r^2,0]}&
     \int_{B_r\times\{t\}} w^\frac{4+4\alpha}{p+2\alpha}\, \dx
     +
     \tfrac{1}{(p+2\alpha)^2}
     \iint_{Q_{r}}|\nabla w|^2
    \,\dx\dt\\
    &\le 
    \frac{C (1+\alpha)}{(R-r)^2}
     \iint_{Q_{R }}
     \big[w^\frac{4+4\alpha}{p+2\alpha} +1\big]\,\dx\dt.
\end{align}
The above inequality  serves as the starting point of the Moser iteration scheme. For $i\in\N_0$, we introduce the following abbreviations
\begin{equation}\label{Eq:k_n-p<2}
\left\{
	\begin{array}{c}
	\displaystyle \rho_i=\tfrac12 R+\frac{R}{2^{i+1}},\quad 
	\tilde{\rho}_i=\frac{\rho_i+\rho_{i+1}}{2},\\[6pt]
	\displaystyle B_{i}=B_{\rho_i},\quad
	Q_{i}=Q_{\varrho_i}
    =B_i\times (-\varrho_i^2,0],\\[6pt]
	\widetilde{B}_{i}=B_{\tilde{\rho}_i},\quad 
	\widetilde{Q}_i=Q_{\tilde\varrho_i}=\widetilde{B}_i\times(-\tilde\varrho_i^2,0],\\[6pt]
  	\kappa=1+\frac{2}N,\quad
  	\alpha_o=\tfrac12p_o-1,\quad 
	\alpha_{i+1}=\kappa\alpha_i+\frac2N-\frac{2-p}2,\\[6pt]
   	\displaystyle p_i=2+2\alpha_{i},\quad 
	q_i=\frac{4+4\alpha_i}{p+2\alpha_i},
	\end{array}
\right.
\end{equation}
where the exponent $p_o$ is chosen such that
$p_o>\max\big\{\frac{ N(2-p)}2, 2\big\}$.
For the sequences $\alpha_i$ and $p_i$ it can be directly deduced from the recursion formulas that
$$
    \alpha_i=\tfrac{\boldsymbol{\nu}}{4}\kappa^i-1+\tfrac{N(2-p)}4\quad\mbox{with}\quad \boldsymbol\nu:= 
    N(p-2)+2p_o>0
$$
and
\begin{equation}\label{p_i+1}
    p_{i+1}=(p+2\alpha_i)\tfrac{N+q_i}{N}
\end{equation}
hold true for any $i\in\N_0$. 
Note also that
\begin{equation*}
    p_i=\tfrac{\boldsymbol \nu}2\kappa^i+\tfrac{N(2-p)}{2}.
\end{equation*}
To proceed further  
let $\zeta_i\in C^\infty(\widetilde Q_i,[0,1])$ be a standard cut-off function that
vanishes on $\partial_{\mathrm{par}}\widetilde{Q}_i$ and equals one in $Q_{i+1}$,
such that $|\nabla\zeta_i |\le 2^{i+4}/R$. 
We apply the Sobolev imbedding \cite[Chapter I, Proposition~3.1]{DB} with $p=2$ and $m=q_i$, 
to obtain that
\begin{align}\label{Sobolev:p<2-unified}\nonumber
    \iint_{Q_{i+1}}&\big(\mu^2+|Du|^2\big)^\frac{p_{i+1}}{2}\,\dx\dt
    \le
    \iint_{\widetilde{Q}_i}(w\zeta_i)^{2\frac{N+q_i}N}\,\dx\dt\\\nonumber
    &\le C\iint_{\widetilde{Q}_i}|\nabla (w\zeta_i)|^{2}\,\dx\dt
    \bigg[\sup_{\tau\in(-\tilde\varrho_i^2,0]}
    \int_{\widetilde{B}_i\times\{\tau\}}(w\zeta_i)^{q_i}\,\dx\bigg]^{\frac{2}N
    }\\\nonumber
    &\le 
    C  \iint_{\widetilde{Q}_i}
    \big[ \zeta_i^2|\nabla w|^2 +w^2|\nabla \zeta_i|^2
    \big]\dx\dt
    \bigg[\sup_{\tau\in(-\tilde\varrho_i^2,0]}
    \int_{\widetilde{B}_i\times\{\tau\}}w^{q_i}\,\dx\bigg]^{\frac{2}N}
    \\
    &\le 
    C  \iint_{\widetilde{Q}_i}
    \bigg[ |\nabla w|^2 +\frac{2^{2i+8}w^2}{R^2}
    \bigg]\dx\dt
    \bigg[\sup_{\tau\in(-\tilde\varrho_i^2,0]}
    \int_{\widetilde{B}_i\times\{\tau\}}w^{q_i}\dx\bigg]^{\frac{2}N},
\end{align}
for a constant $C=C(N)$. In turn we used the properties of the cut-off function $\zeta_i$. The fact that the constant actually depends only on $N$ and $p$ can be seen as follows. Applying the Sobolev embedding from the first to second line yields a constant depending on $N$ and $q_i$. 
Since $q_i=2+\frac{2(2-p)}{p+2\alpha_i}\in [2,4]$ and because the constant $C_{\rm Sob}(N,q_i)$ from the Sobolev embedding depends continuously on $q_i$, the dependence on $q_i$ can be eliminated. 
To the right-hand side integrals of \eqref{Sobolev:p<2-unified} we apply the energy estimate~\eqref{est:energy-w} on $\widetilde{Q}_i\subset Q_i$ instead of $Q_r\subset Q_R$, and  with $\alpha_i$ in place of $\alpha$. This yields 
\begin{align*}
    \iint_{\widetilde{Q}_i}
    |\nabla w|^2 \dx\dt
    &\le C (p+2\alpha_i)^2(1+\alpha_i)
    \frac{2^{2i+6}}{R^2}
    \iint_{Q_i}\big[ w^{q_i}+1\big]\,\dx\dt\\
     &\le 
    \frac{C (1+\alpha_i)^32^{2i}}{R^2}
    \iint_{Q_i}\big[ w^{q_i}+1\big]\,\dx\dt
\end{align*}
and 
\begin{align*}
    \sup_{\tau\in(-\tilde\varrho_i^2,0]}
    \int_{\widetilde{B}_i\times\{\tau\}}w^{q_i}\,\dx
    &\le 
        \frac{C(1+\alpha_i)^2 2^{2i}}{R^2}
    \iint_{Q_i}\big[ w^{q_i}+1\big]\,\dx\dt.
\end{align*}
The term   $w^2$ in \eqref{Sobolev:p<2-unified} is estimated by means of 
$w^2\le  w^{q_i}+1$.
Using also $1+\alpha_i\le C(N,p_o)\kappa^i$,
we obtain
\begin{align*}
    \iint_{Q_{i+1}}
    \big(\mu^2+|Du|^2\big)^\frac{p_{i+1}}{2}\,\dx\dt
    &\le
    C (1+\alpha_i)^{3+\frac4N}
    \bigg[ \frac{2^{2i}}{R^2}
    \iint_{Q_i}\big[ w^{q_i}+1\big]\,\dx\dt\bigg]^\kappa\\
    &\le 
    C (4\kappa^3)^{i\kappa} \bigg[ \frac{1}{R^2}
    \iint_{Q_i}\big[ w^{q_i}+1\big]\,\dx\dt\bigg]^\kappa\\
    &= 
    C \boldsymbol b^{i\kappa} \bigg[ \frac{1}{R^2}
    \iint_{Q_i}\Big[ \big(\mu^2+|Du|^2\big)^\frac{p_i}2+1\Big]\,\dx\dt\bigg]^\kappa
\end{align*}
with $\boldsymbol b=\boldsymbol b(N)=4\kappa^3$ and $C=C(N,p,p_o, C_o,C_1,C_2)$. Taking into account  $|Q_i|\le 2^{N+2}|Q_{i+1}|$ and $|Q_i|^\frac2N\le \big(2^NR^{N+2}\big)^\frac2N=4R^{2\kappa}$ we can pass to mean values on both sides to deduce that
\begin{align*}
     \biint_{Q_{i+1}}
     \big(\mu^2+|Du|^2\big)^{\frac{p_{i+1}}{2}}\,\dx\dt
    &\le
    C \boldsymbol b^{i\kappa} \bigg[ 
    \biint_{Q_i}\Big[ \big(\mu^2+|Du|^2\big)^{\frac{p_i}{2}}+1\Big]\,\dx\dt\bigg]^\kappa
\end{align*}
for a constant $C=C(N,p,p_o, C_o,C_1,C_2)$. Adding $1$ to the left-hand side we have
\begin{align*}   
    \biint_{Q_{i+1}}\Big[\big(\mu^2+|Du|^2\big)^{\frac{p_{i+1}}{2}}+ 1\Big]\,\dx\dt
    &\le
    C \boldsymbol b^{i\kappa} \bigg[ 
    \biint_{Q_i}\Big[ \big(\mu^2+|Du|^2\big)^{\frac{p_i}{2}}+1\Big]\,\dx\dt\bigg]^\kappa.
\end{align*}
Using the abbreviation
\begin{equation*}
    \boldsymbol
    Y_i:=\left[\biint_{Q_i}
    \Big[ 
    \big(\mu^2+|Du|^2\big)^{\frac{p_i}{2}}+1
    \Big]\,\dx\dt
    \right]^{\frac1{p_i}},
\end{equation*}
the last inequality  can be put into the recursive inequality 
\begin{equation}\label{iteration-1}
    \boldsymbol Y_{i+1}^{p_{i+1}}\le \big( C \boldsymbol b^{i}
    \boldsymbol Y_i^{p_i}\big)^{\kappa}
    \qquad\forall\, i\in\N_0.
\end{equation}
Iterating \eqref{iteration-1}, we get 
\begin{align}\label{iteration-2-1}
    \boldsymbol Y_{i}^{p_i}
    &\le
    \prod_{j=1}^{i} C^{\kappa^{i-j+1}}
    \prod_{j=1}^{i} \boldsymbol b^{j\kappa^{i-j+1}}
    \boldsymbol Y_o^{p_o\kappa^{i}}
\end{align}
for every $i\in\N$. The right-hand side of \eqref{iteration-2-1} can be controlled by means of Lemma~\ref{lem:A}. In fact, we first take both sides to the power $\frac{1}{p_i}$. Then, we re-write the products in the form
\begin{align*}
     \prod_{j=1}^{i} C^\frac{\kappa^{i-j+1}}{p_i}&
      \prod_{j=1}^{i} \boldsymbol b^\frac{j\kappa^{i-j+1}}{p_i}
    \boldsymbol Y_o^\frac{p_o\kappa^{i}}{p_i}\\
     &=
      \Bigg[
      \prod_{j=1}^{i} C^\frac{\kappa^{i-j+1}}{\frac1{2}\boldsymbol{\nu}(\kappa^i-1)}
      \Bigg]^\frac{
      \frac12\boldsymbol{\nu}
      (\kappa^i-1)}{p_i}
    \Bigg[
    \prod_{j=1}^{i} \boldsymbol b^\frac{j\kappa^{i-j+1}}{\frac1{2}\boldsymbol{\nu}(\kappa^i-1)}
    \Bigg]^\frac{\frac12\boldsymbol{\nu}
    (\kappa^i-1)}{p_i}
     \big[\boldsymbol Y_o\big]^{p_o\frac{\kappa^i}{p_i}}.
\end{align*}
The exponents outside the brackets converge as $i\to\infty$, i.e.~we
have
\begin{align*}
    \lim_{i\to\infty}\frac{\frac12\boldsymbol{\nu}
    (\kappa^i-1)}{p_i}
    =
    \lim_{i\to\infty}\frac{\frac12\boldsymbol{\nu}
    (\kappa^i-1)}{\frac12\boldsymbol{\nu}\kappa^i+\tfrac{N(2-p)}{2}}
    =1
\end{align*}
as well as
\begin{align*}
    \lim_{i\to\infty}\frac{\kappa^i}{p_i}=\frac2{\boldsymbol{\nu}}.
\end{align*}
Consequently, by virtue of Lemma~\ref{lem:A} we get
\begin{align*}
    \limsup_{i\to\infty}\boldsymbol Y_{i}
    &\le
    \limsup_{i\to\infty}
    \prod_{j=1}^{i} 
    C^{\frac{\kappa^{i-j+1}}{\frac12\boldsymbol{\nu}(\kappa^i-1)}}
      \prod_{j=1}^{i} \boldsymbol b^{\frac{j\kappa^{i-j+1}}{\frac12\boldsymbol{\nu}(\kappa^i-1)}}
      \boldsymbol Y_o^{\frac{2p_o}{\boldsymbol{\nu}}}\\
    &\le
      C^{\frac{2\kappa}{\boldsymbol{\nu}(\kappa-1)}}
      \boldsymbol b^{\frac{2\kappa^2}{\boldsymbol{\nu}(\kappa-1)^2}}
      \boldsymbol Y_o^{\frac{2p_o}{\boldsymbol{\nu}}}\\
    &=
      C^{\frac{N+2}{\boldsymbol{\nu}}}\boldsymbol b^{\frac{(N+2)^2}{2\boldsymbol{\nu}}} 
      \boldsymbol Y_o^{\frac{2p_o}{\boldsymbol{\nu}}},
\end{align*}
whence we deduce
\begin{equation*}
    \sup_{Q_{\frac12 R}}|Du|
    \le
    C
    \bigg[
    \biint_{Q_{R}}\big[
    |Du|^{p_o}+1
    \big]\,\dx\dt\Bigg]^\frac2{\boldsymbol{\nu}}
\end{equation*}
with a constant $C=C(N,p,p_o, C_o,C_1,C_2)$.  By Lemma \ref{lem:Lq-est} we have $|Du|\in L^{p_o}_{\rm loc}(E_T)$. With a standard covering argument we conclude the proof.
\end{proof}

\subsection{Quantitative gradient bound
in the sub-quadratic case}
Now we are in a position to refine the previous result and derive a quantitative gradient bound. The argument exploits a De Giorgi-type iteration.
\begin{proposition}\label{lem:L-infty-quant-p<2}
Let $\mu\in(0,1]$, $1<p< 2$, and $r>1$ that satisfies
\begin{equation*}
    \boldsymbol\nu_r :=N(p-2)+2r>0.
\end{equation*}
Then there exists a  constant $C=C(N,p, r,C_o,C_1,C_2)$ such that  for any locally bounded weak solution $u$ to the parabolic system \eqref{eq:diff-systems} in the sense of Definition \ref{def:weak-loc} with assumption \eqref{ass:b} we have
\begin{align*}
    &\sup_{Q_{\sigma R,\sigma S}(z_o)} |Du|\\
    &\qquad 
    \le
    \max\Bigg\{
     \frac{C}{\eps^\frac{2(2-p)(N+2)}{\boldsymbol \nu_{r}}}
    \frac{\big(R^2/S\big)^\frac{N}{\boldsymbol\nu_{r}}}{(1-\sigma)^\frac{2(N+2)}{\boldsymbol \nu_{r}}} \bigg[\biint_{Q_{R,S}(z_o)}|Du|^{r}\, \dx\dt\bigg]^\frac{2}{\boldsymbol\nu_{r}}
    , \eps\Big(\frac{S}{R^2}
    \Big)^\frac{1}{2-p},\eps\mu
    \Bigg\}
\end{align*}
for every $Q_{R,S}(z_o)\Subset E_T$ with $R\in(0,1]$, $S>0$, every $\sigma \in (0,1)$, and every $\eps\in (0,1] $.
\end{proposition}
\begin{proof}
The starting point of the proof of the quantitative gradient bound is the energy estimate \eqref{est:estD2u} from Proposition~\ref{prop:estD2u} in which we discard the non-negative term $\Phi'(|Du|^2)|Du|^2$ in the denominator of the first integral on the right-hand side. Let us apply this inequality with 
\begin{equation*}
    \Phi (\tau):= 
    \tau^\frac{2-p}2\big(\sqrt{\tau}-k\big)_+^{r},\quad \mbox{and $r>1$, $k>0$.}
\end{equation*}
The fact that $\Phi$ is neither bounded nor Lipschitz is immaterial, since by Proposition \ref{lem:L-infty-est-qual-p<2} we have $|Du|\in L^\infty_{\rm loc} (E_T)$ and we can replace $\Phi$ by $\Phi_m(\tau):=\Phi(\min\{\tau,m\})$ with
$m>\|Du\|_{L^\infty (Q_{R,S})}^2$.
 
Next, we aim to estimate $v$ defined in \eqref{def-v} with this choice of $\Phi$. Notice that we only need to consider the case $|Du|>k$ as $v=0$ in the opposite case.  Substituting $\sqrt{\tau}=s$ and noting that $p<2$ we estimate from below by
\begin{align*}
    v&
    =
    \int_0^{|Du|^2}\Phi(\tau)\,\d\tau
    =
    \int_{k^2}^{|Du|^2}
    \tau^\frac{2-p}2\big(\sqrt{\tau}-k\big)_+^{r}
    \,\dtau\\
    &=
    2\int_{k}^{|Du|} s^{2-p}(s-k)^{r}s\,\ds
    \ge 
    2k^{2-p}\int_{k}^{|Du|}(s-k)^{r+1}\,\ds\\
    &=
    \frac{2k^{2-p}}{r+2}\big(|Du|-k\big)^{r+2}.
\end{align*}
Similarly, we estimate from above by
\begin{align*}
    v
    &=
    2\int_{k}^{|Du|} s^{2-p}(s-k)^{r}s\,\ds
    \le
    2 \int_{k}^{|Du|} s^{r+3-p}\,\ds\\
    &\le
    \frac{2}{r+4-p}|Du|^{r+4-p}\le  \frac{2}{r+2}|Du|^{r+4-p}.
\end{align*}
Altogether, this results in
\begin{equation*}
    \frac{2k^{2-p}}{r+2}\big(|Du|-k\big)_+^{r+2}\le v\le \frac{2}{r+2}|Du|^{r+4-p}\mathbf 1_{\{|Du|>k\}}.
\end{equation*}

The second integral on the left-hand
side of \eqref{est:estD2u}, i.e.~the one containing $|D^2u|^2$, can be estimated from below, if we  require
$k\ge \frac12 \eps \mu$ with $\eps\in (0,1]$.
In this case we have  $|Du|>k\ge\frac12 \eps \mu$ on the set $\{|Du|>k\}$, which implies 
\begin{align*}
    &\big(\mu^2 +|Du|^2\big)^\frac{p-2}{2}
     |D^2u|^2 
     \Phi\big( |Du|^2\big)\\
     &\qquad=\mathbf 1_{\{|Du|>k\}}
     \big(\underbrace{\mu^2 +|Du|^2}_{\le 5\eps^{-2}|Du|^2}\big)^\frac{p-2}{2}
     |Du|^{2-p}\big(|Du|-k\big)^{r}
     |D^2u|^2 \\
     &\qquad 
     \ge 5^\frac{p-2}2\eps^{2-p}
     \mathbf 1_{\{|Du|>k\}} \big(|Du|-k\big)^{r} \big|\nabla |Du|\big|^2\\
     &\qquad
     \ge \frac{4\,\eps^{2-p}}{C(p)(r+2)^2}
     \Big|\nabla\big(|Du|-k\big)_+^\frac{r+2}2\Big|^2.
\end{align*}
In turn we  used Kato's inequality $|\nabla |Du||\le |D^2u|$. 

Finally, we consider the integrand of the first and second integrals appearing on the right-hand side of \eqref{est:estD2u}. Again we only have to focus on the set $\{|Du|>k\}$. Taking again into account that $p<2$, we obtain
\begin{align*}
    \big(\mu^2 &+|Du|^2\big)^\frac{p-2}{2}
    |Du|^2\Phi\big( |Du|^2\big)\\
    &
    =
    \mathbf 1_{\{|Du|>k\}}
    \underbrace{\big(\mu^2 +|Du|^2\big)^\frac{p-2}{2}}_{\le |Du|^{p-2}} |Du|^2
    |Du|^{2-p}\big(|Du|-k\big)^{r}\\
    & \le
    \mathbf 1_{\{|Du|>k\}}|Du|^{r+2}.
\end{align*}
Similarly, we have
\begin{align*}
    \big(\mu^2 &+|Du|^2\big)^\frac{p-2}{2}\Phi'(|Du|^2)|Du|^4\\
    &\le
    \mathbf 1_{\{|Du|>k\}}
    |Du|^{2}
    \big(|Du|-k\big)^{r-1}
    \Big[\big(1-\tfrac{p}2+\tfrac{r}2\big)|Du|
    - \big(1-\tfrac{p}2\big)k\Big]
    \\
    &
    \le\tfrac12 (r+1)\mathbf 1_{\{|Du|>k\}}|Du|^{r+2}.
\end{align*}
Inserting these inequalities into the energy inequality \eqref{est:estD2u} (taking $z_o=(0,0)$ for simplicity) results in
\begin{align}\label{energy:quant}
    k^{2-p}\sup_{\tau\in (-S,0]}&
     \int_{B_R\times\{\tau\}} \big[\big(|Du|-k\big)_+^\frac{r+2}2\zeta\big]^2\, \dx \nonumber\\
     &\phantom{\le\,}
     +
     \eps^{2-p}
    \iint_{Q_{R,S}}
    \big|\nabla\big[\big(|Du|-k\big)_+^\frac{r+2}{2}\zeta\big]\big|^2\,\dx\dt \nonumber\\
     &\le
     C\iint_{Q_{R,S}}
     \mathbf 1_{\{|Du|>k\}}|Du|^{r+2}\big(|\nabla\zeta|^2+\zeta^2\big)\,\dx\dt \nonumber\\
     &\phantom{\le\,}+
     C\iint_{Q_{R,S}}|Du|^{r+4-p}
     \mathbf 1_{\{ |Du|>k\}}\big|\partial_t\zeta^2\big|\,\dx\dt,
\end{align}
with a constant $C=C(p,r,C_o,C_1,C_2)$.
Note that the preceding energy inequality holds true for any general cylinder $Q_{R,S}$, for any $k\ge\frac12\eps\mu $, and any cut-off function $\zeta$ vanishing on the parabolic boundary $\partial_{\mathrm{par}}Q_{R,S}$. 

Now, for a fixed cylinder $Q_o:=Q_{R,S}$ we define shrinking families $Q_n:=
Q_{\rho_n,\theta_n}$ and $\widetilde Q_n:= Q_{\tilde\rho_n,\tilde\theta_n}$ of cylinders with the same base point, where
$\rho_n$, $\tilde\rho_n$, $\theta_n$, $\tilde \theta_n$ are defined according to
\begin{equation*}
    \rho_n:= \sigma R +\frac{1-\sigma}{2^n}R,
    \quad\
    \theta_n:= \sigma S+\frac{1-\sigma}{2^n}S
\end{equation*}
and 
\begin{equation*}
    \tilde\rho_n:=\tfrac12\big( \rho_n+\rho_{n+1}\big)=\sigma R+\frac{3(1-\sigma)}{2^{n+2}}R,
    \quad
    \tilde\theta_n:=\tfrac12\big( \theta_n+\theta_{n+1}\big)
    =
    \sigma S +\frac{3(1-\sigma)}{2^{n+2}}S.
\end{equation*}
These cylinders are arranged in such a way that $Q_{n+1}\subset \widetilde 
Q_n\subset Q_n$ for any $n\in\N_0$. The level $k$ in the energy estimate is fixed by an increasing sequence of levels 
\begin{equation*}
     k_n:= k-\frac{ k}{2^{n}},
\end{equation*}
where $ k\ge \epsilon\mu$ is a quantity to be determined later in the course of the proof. Note that $k_{n+1}>k_o=\frac12  k\ge \tfrac12\eps\mu$. 
Since the above energy inequality shall be applied on the cylinders $Q_n$, we need to specify the cut-off function $\zeta$ in such a way that it vanishes on the parabolic boundary of $Q_n$ on the one hand, and is identical to 1 on the intermediate cylinder $\widetilde Q_n$ on the other hand. In addition, we require that
$$
    |\nabla\zeta_n|\le \frac{2^{n+3}}{(1-\sigma)R}
    \quad\mbox{and}\quad
    |\partial_t\zeta_n|\le\frac{2^{n+3}}{(1-\sigma)S}.
$$
With respect to these choices, the energy estimate \eqref{energy:quant} for $Q_n$ with $k_{n+1}$ reads as follows
\begin{align}\label{est:en-k_(n+1)}\nonumber
     k^{2-p}\sup_{\tau\in (-\theta_n,0]}&
     \int_{B_n\times\{\tau\}}\big(|Du|-k_{n+1}\big)_+^{r+2}\zeta_n^2\, \dx
     \\\nonumber
     &\phantom{\le\,}
     +
     \eps^{2-p}
     \iint_{Q_n}
     \big|\nabla
     \big[\big(|Du|-k_{n+1}\big)_+^\frac{r+2}{2}\zeta_n\big]\big|^2\,\dx\dt\\\nonumber
     &\le
     \frac{C 2^{2n}}{(1-\sigma)^2R^2}
     \iint_{Q_n}
     \mathbf 1_{\{|Du|>k_{n+1}\}}|Du|^{r+2}\,\dx\dt\\
     &\phantom{\le\,}+
     \frac{C 2^{n}}{(1-\sigma)S}
     \iint_{Q_n}\mathbf 1_{\{|Du|>k_{n+1}\}}
     |Du|^{r+4-p}\,\dx\dt.
\end{align}
To shorten the notation introduce
\begin{equation*}
    \boldsymbol Y_n:=\iint_{Q_n}\big(|Du|-k_n\big)_+^{r+2}\,\dx\dt.
\end{equation*}
The integral $\boldsymbol Y_n$ can be easily estimated from below by reducing the domain of integration. Indeed, we have
\begin{align*}
    \iint_{Q_n} \big( |Du|-k_n\big)_+^{r+2}\,\dx\dt
    &\ge
    \iint_{\widetilde Q_n\cap\{|Du|>k_{n+1}\} } \big( |Du|-k_n\big)_+^{r+2}\,\dx\dt\\
    &\ge
    \big( k_{n+1}-k_n\big)^{r+2}
    \Big| \widetilde Q_n \cap \big\{ |Du|>k_{n+1}\big\}
   \Big|\\
   &=\frac{ k^{r+2}}{2^{(n+1)(r+2)}}\Big| \widetilde Q_n \cap \big\{ |Du|>k_{n+1}\big\}
   \Big|.
\end{align*}
This measure estimate  can be used to estimate $\boldsymbol Y_{n+1}$ from above. For this purpose we first increase in $\boldsymbol Y_{n+1}$ the domain of integration from $Q_{n+1}$ to $\widetilde Q_n$ and then apply H\"older's inequality with exponents $\frac{N+2}N$ and $\frac{N+2}2$. Subsequently, we use the fact that $\zeta_n=1$ on $\widetilde Q_n$ and the  measure estimate. In this way we get
\begin{align*}
    \boldsymbol Y_{n+1}
    &\le 
    \iint_{\widetilde Q_n} \Big[\big(
    |Du|-k_{n+1}\big)_+^\frac{r+2}{2} \zeta_n\Big]^2\, \dx\\
    &\le
    \bigg[ \iint_{\widetilde Q_n} \Big[\big(|Du|-k_{n+1}\big)_+^\frac{r+2}{2}\zeta_n\Big]^{2\frac{N+2}{N}}\dx\dt\bigg]^\frac{N}{N+2}
    \Big|\widetilde Q_n \cap \big\{ |Du|>k_{n+1}\big\}
   \Big|^\frac{2}{N+2}\\
   &\le
    \bigg[ \iint_{ Q_n} 
    \Big[\big(|Du|-k_{n+1}\big)_+^\frac{r+2}{2}\zeta_n
    \Big]^{2\frac{N+2}{N}}\dx\dt\bigg]^\frac{N}{N+2}
    \bigg[
    \frac{2^{(r+2)(n+1)}}{ k^{r+2}}\boldsymbol Y_n
    \bigg]^\frac{2}{N+2}.
\end{align*}
The first integral on the right-hand side can be controlled using Sobolev's inequality \cite[Chapter I, Proposition~3.1]{DB} with $m=p=2$ and $q=2\frac{N+2}{N}$. The application yields
\begin{align*}
    \iint_{ Q_n} &\Big[\big(|Du|-k_{n+1}\big)_+^\frac{r+2}{2} \zeta_n\Big]^{2\frac{N+2}{N}}\dx\dt
    \le  C\mathbf I\cdot \mathbf{II}^\frac{2}{N},
\end{align*}
where we abbreviated
\begin{align*}
    \mathbf I
    &:=
    \iint_{ Q_n}
    \Big|\nabla\big[ \big(|Du|-k_{n+1}\big)_+^\frac{r+2}{2}
    \zeta_n\big]\Big|^2\dx\dt
    \le
    \eps^{p-2}\big[
      \mbox{right-hand side of \eqref{est:en-k_(n+1)}}
      \big]
\end{align*}
and
\begin{align*}
    \mathbf{II}
    &:=
     \sup_{\tau\in(-\theta_n,0]}\int_{B_n\times\{\tau\}}
    \big[\big(|Du|-k_{n+1}\big)_+^\frac{r+2}{2} \zeta_n\big]^{2}\dx
    \le k^{p-2}\big[
      \mbox{right-hand side of \eqref{est:en-k_(n+1)}}
      \big].
\end{align*}
Inserting this above gives
\begin{align*}
      \iint_{ Q_n} \!\Big[\big(|Du|-k_{n+1}\big)_+^\frac{r+2}{2} \zeta_n\Big]^{2\frac{N+2}{N}}\!\dx\dt
      &\le
      C\eps^{p-2}  k^{(p-2)\frac2N}\big[
      \mbox{right-hand side of \eqref{est:en-k_(n+1)}}
      \big]^\frac{N+2}{N}.
\end{align*}
Therefore, we obtain
\begin{align*}
    \boldsymbol Y_{n+1}
   &\le
   C
    \Big[ \eps^{p-2} k^{(p-2)\frac2N}\big[
      \mbox{right-hand side of \eqref{est:en-k_(n+1)}}
      \big]^\frac{N+2}{N}\Big]^\frac{N}{N+2}
    \bigg[
    \frac{2^{(r+2)(n+1)}}{ k^{r+2}}\boldsymbol Y_n
    \bigg]^\frac{2}{N+2}\\
    &=
     C 
     \eps^\frac{N(p-2)}{N+2}
       k^{-\frac{2(r+4-p)}{N+2}} 2^\frac{2(r+2)(n+1)}{N+2}\boldsymbol Y_n^\frac{2}{N+2}
    \big[
      \mbox{right-hand side of \eqref{est:en-k_(n+1)}}
      \big]\\
    &\le 
    \frac{C 2^{n(r+4)}}{(1-\sigma)^2}
     \eps^\frac{N(p-2)}{N+2}
      k^{-\frac{2(r+4-p)}{N+2}}\boldsymbol Y_n^\frac{2}{N+2} \boldsymbol{E}_n,
\end{align*}
where we abbreviated
\begin{align*}
    \boldsymbol{E}_n
    &:=
    \frac{1}{R^2}
     \iint_{Q_n}
     \mathbf 1_{\{|Du|>k_{n+1}\}}|Du|^{r+2}\,\dx\dt
     +
     \frac{1}{S}
     \iint_{Q_n}\mathbf 1_{\{ |Du|>k_{n+1}\}}|Du|^{r+4-p}\,\dx\dt.
\end{align*}
Since $|Du|$ is locally bounded in $E_T$, we can further estimate
\begin{align*}
    \boldsymbol{E}_n
    &\le
    \bigg[\frac{1}{R^2}
    +
    \frac1S
    \Big(\sup_{Q_n} |Du|\Big)^{2-p}
    \bigg]
     \iint_{Q_n}\mathbf 1_{\{ |Du|>k_{n+1}\}}|Du|^{r+2}\,\dx\dt,
\end{align*}
so that only an integral with $|Du|^{r+2}$ as integrand remains on the right-hand side. This integral can be controlled by $\boldsymbol Y_n$. In fact, 
we shrink the domain of integration in $\boldsymbol Y_{n}$ from  $\{|Du|>k_n\}$ to 
$\{ |Du|>k_{n+1}\}$ and use 
$$
    k_n=k_{n+1}\frac{2^{n+1}-2}{2^{n+1}-1},
$$
to obtain
\begin{align*}
    \boldsymbol Y_{n}
    &\ge
    \iint_{Q_n}\big(|Du|-k_n\big)_+^{r+2}\mathbf 1_{\{ |Du|>k_{n+1}\}}\,\dx\dt\\
    &=
    \iint_{Q_n}\bigg[|Du|-k_{n+1}\frac{2^{n+1}-2}{2^{n+1}-1}\bigg]^{r+2}
    \mathbf 1_{\{ |Du|>k_{n+1}\}}\,\dx\dt\\
    &\ge
    \iint_{Q_n}|Du|^{r+2}\bigg[1-\frac{2^{n+1}-2}{2^{n+1}-1}\bigg]^{r+2}
    \mathbf 1_{\{ |Du|>k_{n+1}\}}\,\dx\dt\\
    &\ge
    \frac{1}{2^{(n+1)(r+2)}}\iint_{Q_n}|Du|^{r+2}
    \mathbf 1_{\{ |Du|>k_{n+1}\}}\,\dx\dt.
\end{align*}
Inserting this above we get
\begin{align*}
    \boldsymbol Y_{n+1}
     &\le
     \frac{C 2^{2n(r+4)}}{(1-\sigma)^2}
     \eps^\frac{N(p-2)}{N+2}
      k^{-\frac{2(r+4-p)}{N+2}}
    \bigg[
         \frac{1}{R^2}+\frac{1}{S} \Big(\sup_{Q_n} |Du|\Big)^{2-p} \bigg]
     \boldsymbol Y_n^{1+\frac{2}{N+2}}.
\end{align*}
Now, if for some $n\in\N_0$ we have
\begin{equation*}
    \Big(\sup_{Q_n} |Du|\Big)^{2-p}
    \le
    \epsilon^{2-p}\frac{S}{R^2} 
\end{equation*}
there is nothing more to prove, since 
\begin{equation}\label{est:Du-first-case}
    \sup_{Q_{\sigma R,\sigma S}} |Du|
    = 
    \sup_{Q_\infty} |Du|
    \le
    \sup_{Q_n} |Du|
    \le
    \eps \Big(\frac{S}{R^2}
    \Big)^\frac1{2-p}.
\end{equation}
Otherwise, for any $n\in \N_0$ we have 
\begin{equation*}
    \Big(\sup_{Q_n} |Du|\Big)^{2-p}
    \ge
    \epsilon^{2-p}\frac{S}{R^2}
    \quad\Longleftrightarrow\quad
    \frac{1}{R^2}
    \le
    \frac{1}{S}\Big(\tfrac1{\eps}\sup_{Q_n} |Du|\Big)^{2-p},
\end{equation*}
and from the above iterative inequality we obtain
\begin{align*}
    \boldsymbol Y_{n+1}
    &\le 
     \frac{C \boldsymbol b^{n}}{(1-\sigma)^2S}
     \eps^\frac{N(p-2)}{N+2}
      k^{-\frac{2(r+4-p)}{N+2}}
     \Big(\tfrac1{\eps}\sup_{Q_o} |Du|\Big)^{2-p}
     \boldsymbol Y_n^{1+\frac{2}{N+2}},
\end{align*}
where we abbreviated $\boldsymbol b= 4^{2(r+4)}$. At this point we apply the lemma on geometric convergence \cite[Chapter I, Lemma 4.1]{DB} with $C$ and $\alpha$ replaced by
$$
    \frac{C \boldsymbol }{(1-\sigma)^2S} \eps^\frac{N(p-2)}{N+2}
       k^{-\frac{2(r+4-p)}{N+2}}
     \Big(\tfrac1{\eps}\sup_{Q_o} |Du|\Big)^{2-p}
    \quad
    \mbox{and}
    \quad
    \tfrac2{N+2}.
$$
Then, if 
\begin{align*}
    \boldsymbol Y_o
    &=
    \iint_{Q_o}|Du|^{r+2}\,\dx\dt\\
    &\le
    \Bigg[
    \frac{C}{(1-\sigma)^2S}
    \eps^\frac{N(p-2)}{N+2}
     k^{-\frac{2(r+4-p)}{N+2}}
    \Big(
\tfrac1{\eps}\sup_{Q_o} |Du|\Big)^{2-p}
    \Bigg]^{-\frac{N+2}{2}}
    \boldsymbol b^{-\frac{(N+2)^2}{4}}\\
    &=
    C^{-1}\big((1-\sigma)^2S\big)^\frac{N+2}{2} 
    \eps^{-\frac{N(p-2)}2}\Big(\tfrac1{\eps}\sup_{Q_o} |Du|\Big)^{-\frac{(2-p)(N+2)}{2}}
     k^{r+4-p},
\end{align*}
we have $\boldsymbol Y_n\to 0$ as $n\to\infty$, which yields 
\[
\sup_{Q_{\sigma R,\sigma S}}|Du|\le k.
\]
The condition for $k$ encoded in this geometric convergence is 
\begin{equation*}
      k\ge \frac{C\eps^\frac{N(p-2)}{2(r+4-p)}}{[(1-\sigma)^2S]^\frac{N+2}{2(r+4-p)}}
    \Big( \tfrac1{\eps}\sup_{Q_o} |Du|\Big)^{\frac{(2-p)(N+2)}{2(r+4-p)}}
    \bigg[\iint_{Q_o}|Du|^{r+2}\,\dx\dt
    \bigg]^\frac{1}{r+4-p}.
\end{equation*}
Let us label the quantity on the right-hand side as $\boldsymbol{K}$.
Taking also into account the condition $  k\ge\epsilon\mu$ that we set up earlier, 
we end up with
\[
\sup_{Q_{\sigma R,\sigma S}}|Du|\le\max\{\boldsymbol{K}, \varep\mu\}.
\] 
At this point, 
we simplify the exponent of $\varep$ in the definition of $\boldsymbol{K}$ by estimating
\begin{align*}
    \frac{N(2-p)}{2(r+4-p)}+\frac{(2-p)(N+2)}{2(r+4-p)}
    &=
    \frac{(2-p)(2N+2)}{2(r+4-p)}
    \le
    \frac{2(2-p)(N+2)}{2(r+4-p)}
\end{align*}
and obtain
\begin{align*}
    &\sup_{Q_{\sigma R,\sigma S}}|Du|\\
    &\quad
    \le
    \bigg[  \frac{C\eps^{2(p-2)}}{(1-\sigma)^2S}\bigg]^\frac{N+2}{2(r+4-p)}
    \bigg[\iint_{Q_{R,S}}|Du|^{r+2}\dx\dt
    \bigg]^\frac{1}{r+4-p}
    \Big(\sup_{Q_{R,S}} |Du|\Big)^{\frac{(2-p)(N+2)}{2(r+4-p)}}\vee
    \epsilon\mu.
\end{align*}
Let us reduce the integral exponent from $r+2$ to $r$ while increasing the exponent of the sup-norm. Consequently, joining this with the alternative estimate \eqref{est:Du-first-case} we get
\begin{align*}
    &\sup_{Q_{\sigma R,\sigma S}}|Du|\\
    &\quad
    \le\max\Bigg\{
    \bigg[\frac{C\eps^{2(p-2)}}{(1-\sigma)^2S}\bigg]^\frac{N+2}{2(r+4-p)}
    \bigg[\iint_{Q_{R,S}}|Du|^{r}\dx\dt
    \bigg]^\frac{1}{r+4-p}
    \Big(\sup_{Q_{R,S}} |Du|\Big)^{\frac{(2-p)(N+2)+4}{2(r+4-p)}},\\
    &\qquad\qquad\quad\,\,\,
    \eps \Big(\frac{S}{R^2}\Big)^\frac1{2-p}, \eps\mu\Bigg\}.
\end{align*}

To derive the desired $L^\infty$-gradient bound we rest upon the above estimate and set up an interpolation argument. 
To this end, define a sequence of nested cylinders $\widehat Q_n=\widehat Q_{\rho_n,\theta_n}$ with $\rho_o=\sigma R$, $\theta_o=\sigma S$, and 
with
\begin{equation*}
    \rho_n:=\sigma R +(1-\sigma)R \sum_{j=1}^n2^{-j}
    \quad\mbox{and}\quad
    \theta_n:=\sigma S +(1-\sigma)S \sum_{j=1}^n2^{-j}
\end{equation*}
for $n\in\N_0$.
Observe that $\widehat Q_o=Q_{\sigma R,\sigma S}$ and $\widehat Q_\infty = Q_{R,S}$. Moreover, we define
$$
    \boldsymbol M_n
    :=
    \sup_{\widehat Q_n} |Du|.
$$
Now, we apply the above gradient bound on  two consecutive cylinders $\widehat Q_{n+1}\supset \widehat Q_n$ instead of
$Q_{R,S}\supset Q_{\sigma R, \sigma S}$. This means that $\tau\in (0,1)$ must be determined in such a way that $\tau \widehat Q_{n+1}=
\widehat Q_{n}$,
  i.e.~$\tau\theta_{n+1}=\theta_n$. This is the same as
\begin{equation*}
    (1-\tau)^2\theta_{n+1} 
    =
    \frac{(\theta_{n+1}-\theta_{n})^2}
    {\theta_{n+1}}
    = 
    \frac{(1-\sigma)^2S
    2^{-2(n+1)}}{1-(1-\sigma)2^{-(n+1)}},
\end{equation*}
which implies
\begin{equation*}
    \frac{1}{(1-\tau)^2\theta_{n+1}}
    =
    2^{2(n+1)}
    \frac{1-(1-\sigma)2^{-(n+1)}}{(1-\sigma)^2S}
    \le 
    \frac{2^{2(n+1)}}{(1-\sigma)^2S}.
\end{equation*}
Moreover, we have
\begin{equation*}
    \frac{\theta_{n+1}}{\rho_{n+1}^2}
    =
    \frac{S}{R^2}
    \Big[
    \sigma +(1-\sigma) \sum_{j=1}^n2^{-j}
    \Big]^{-1}
    \le
    \frac{2S}{R^2}.
\end{equation*}
This yields recursive inequalities of the type
\begin{align*}
    \boldsymbol M_n
    &
    \le
    \max\Bigg\{ 
    \bigg[  \frac{C2^{2(n+1)}\eps^{2(p-2)}}{(1-\sigma)^2S}\bigg]^\frac{N+2}{2(r+4-p)}
    \bigg[\iint_{Q_{R,S}}|Du|^{r}\dx\dt
    \bigg]^\frac{1}{r+4-p}
    \boldsymbol M_{n+1}^{\frac{(2-p)(N+2)+4}{2(r+4-p)}},\\
    &\qquad\qquad\,
    2^\frac{1}{2-p}\epsilon \Big(\frac{S}{R^2}\Big)^\frac{1}{2-p},\eps\mu\Bigg\}.
\end{align*}
Note that the factor $2^\frac{1}{2-p}$ in the second term can be avoided if we replace $\epsilon$ by $2^{-\frac{1}{2-p}}\epsilon$.

If for some $n\in\N_0$ either the second or the third term in the maximum dominates the first one, there is nothing to prove, because we trivially have
\begin{align*}
    \sup_{Q_{\sigma R,\sigma S}} |Du| 
    =
   \boldsymbol  M_o
    \le 
   \boldsymbol  M_{n}
    \le 
    \max\bigg\{ \epsilon
    \Big(\frac{S}{R^2}\Big)^{\frac{1}{2-p}}, \eps\mu\bigg\}.
\end{align*}
Otherwise, we have
\begin{align*}
    \boldsymbol M_n
    &
    \le 
    \underbrace{
    \Bigg[  \frac{C\eps^{2(p-2)}
    }{(1-\sigma)^2S} 
    \bigg[
    \iint_{Q_{R,S}}|Du|^{r}\dx\dt
    \bigg]^\frac2{N+2}
    \Bigg]^\frac{N+2}{2(r+4-p)}}_{=:\kappa}
    \big(\underbrace{2^\frac{N+2}{r+4-p }}_{=:\boldsymbol b}\big)^n
    \boldsymbol M_{n+1}^{\frac{(2-p)(N+2)+4}{2(r+4-p)}}
\end{align*} 
for any $n\in\N_0$. Let
\begin{align*}
    \boldsymbol \nu_r:= 2r+N(p-2)>0
    \quad
    \mbox{and}
    \quad
    \alpha:=\frac{\boldsymbol\nu_{r}}{2(r+4-p)}\in(0,1).
\end{align*}
Note that  the exponent of $\boldsymbol M_{n+1}$ is
\begin{align*}
    \frac{(2-p)(N+2)+4}{2(r+4-p)} 
    &
    =1-\frac{\boldsymbol\nu_{r}}{2(r+4-p)}
    =1-\alpha.
\end{align*}
With these abbreviations, the iterative inequalities have the form
$$
    \boldsymbol M_n
    \le 
    \kappa \boldsymbol b^n
    \boldsymbol M_{n+1}^{1-\alpha}.
$$
The Interpolation Lemma \cite[Chapter I, Lemma 4.3]{DB} yields
\begin{align*}
    \boldsymbol M_o
    &=
    \sup_{Q_{\sigma R,\sigma S}}|Du|\\
    &\le
    \Bigg[
    2 \bigg\{  \frac{C\eps^{2(p-2)}
    }{(1-\sigma)^2S} 
    \bigg[
    \iint_{Q_{R,S}}|Du|^{r}\dx\dt
    \bigg]^\frac2{N+2}\bigg\}^\frac{N+2}{2(r+4-p)} \boldsymbol b^{\frac{1}{\alpha}-1}
    \Bigg]^\frac{2(r+4-p)}{\boldsymbol\nu_{r}}\\
    &=
    \frac{C}{\eps^\frac{2(2-p)(N+2)}{\boldsymbol \nu_{r}}}\frac{1}{[(1-\sigma)^2S]^{\frac{N+2}{\boldsymbol\nu_{r}}}}
    \bigg[\iint_{Q_{R,S}}|Du|^{r}\, \dx\dt\bigg]^\frac{2}{\boldsymbol\nu_{r}}\\
    &=
    \frac{C}{\eps^\frac{2(2-p)(N+2)}{\boldsymbol \nu_{r}}}
    \frac{\big(R^2/S\big)^\frac{N}{\boldsymbol\nu_{r}}}{(1-\sigma)^\frac{2(N+2)}{\boldsymbol \nu_{r}}} \bigg[\biint_{Q_{R,S}}|Du|^{r}\, \dx\dt\bigg]^\frac{2}{\boldsymbol\nu_{r}}.
\end{align*}
Joining this with the first alternative, we obtain the claimed quantitative $L^\infty$-gradient bound.
\end{proof}

\begin{remark}\label{Rmk:p>p_*}\upshape
The preceding gradient bound corresponds to \cite[Chapter VIII, Theorem 5.2]{DB}.  To allow $r=p$, we need $\boldsymbol\nu_p=N(p-2)+2p>0$ which is equivalent to $p>\frac{2N}{N+2}$. In this case the gradient bound simplifies  to
\begin{align*}
    &\sup_{Q_{\sigma R,\sigma S}(z_o)} |Du|\\
    &\qquad\le
    \max\Bigg\{
    \Bigg[\frac{C}{\epsilon^{\frac{2(2-p)(N+2)}{2}}}
     \frac{\big(R^2/S\big)^\frac{N}{2}}{(1-\sigma)^{N+2}}
    \biint_{Q_{R,S}(z_o)} |Du|^p\,\dx\dt
    \Bigg]^\frac{d}{p}, \eps\Big(\frac{S}{R^2}
    \Big)^\frac{1}{2-p}, \eps\mu
    \Bigg\},
\end{align*}
where
$$
    d:=\frac{2p}{(N+2)p-2N}
$$
stands for the {\em scaling deficit}.
Note that $d$ blows up in the limit $p\downarrow \frac{2N}{N+2}$.
\end{remark}

\begin{remark}\label{Rmk:p=2}
\upshape A careful examination of the proof shows that when $p=2$, we can take $\eps=1$ and, quite obviously, $\mu=0$. Moreover, due to the natural homogeneity of the equation, we can also assume $S=R^2$. Hence, the scaling deficit $d\equiv1$, and the previous estimate reduces to 
\begin{equation*}
    \sup_{Q_{\sigma R,\sigma R^2}(z_o)} |Du|
    \le
    \Bigg[\frac{C}{(1-\sigma)^{N+2}}
    \biint_{Q_{R,R^2}(z_o)} |Du|^2\,\dx\dt
    \Bigg]^\frac{1}{2}.
\end{equation*}
\end{remark}

When $p\in(1,\frac{2N}{N+2}]$, a gradient bound cannot be estimated via the $L^p$-norm of $Du$ in general like in Remark~\ref{Rmk:p>p_*}. However, it is possible if we consider bounded solutions.
\begin{proposition}\label{lem:L-infty-est-quant-p<2}
Let $\mu\in(0,1]$ and $1<p< 2$. 
Then, for any locally bounded weak solution $u$ to the parabolic system \eqref{eq:diff-systems} in the sense of Definition \ref{def:weak-loc} with assumption \eqref{ass:b}, we have $|Du|\in L^\infty_{\mathrm{loc}}(E_T)$.
Moreover, there exist positive constants $\theta =\theta (N,p)$ and $C=C(N,p, C_o,C_1,C_2)$ such that for every cylinder $Q_{R,S}(z_o)\Subset E_T$ with $0<R\le1$ and $S>0$ we have
\begin{align}\label{est:L^infty-L^p}\nonumber
    \sup_{\frac12 Q_{ R, S}(z_o)} |Du|
    & 
    \le
    \max\Bigg\{
     \frac{C}{\eps^\theta}
    \Big(\frac{R^2}{S}\Big)^\frac{N}{4p} 
    \boldsymbol{\mathfrak T}^\frac{N(2-p)+2p}{4p}
    \bigg[\biint_{Q_{2R,2S}}
    |Du|^{p}\, \dx\dt\bigg]^\frac{1}{2p}
    ,\\
    &\phantom{ 
    \le
    \max\Bigg\{\,}
    \eps\Big(\frac{S}{R^2}
    \Big)^\frac{1}{2-p},  \epsilon\Big(\frac{R^2}{S}\Big)^\frac{N}{4p} \mu^\frac{N(2-p)+4p}{4p},\epsilon\mu
    \Bigg\},
\end{align}
for every $\epsilon\in(0,1]$, where
\begin{equation*}
        \boldsymbol{\mathfrak T}
    :=
    \frac{\boldsymbol\omega}{R} 
    +
    \Big(
    \frac{\boldsymbol\omega}{\sqrt{S}}
    \Big)^\frac{2}{p},
\end{equation*}
and
\begin{equation*}
    \boldsymbol\omega :=\sup_{t\in (t_o-2S,t_o]}
    \sup_{x\in B_{2R}(x_o)}\big|u(x,t)-(u)_{x_o,2R}(t)\big|.
\end{equation*}
\end{proposition}

\begin{proof}
In the gradient bound from Proposition \ref{lem:L-infty-quant-p<2}, we rename $r$ by $m$ and fix $\sigma =\frac12$. Moreover, we take $m:=\frac{N(2-p)}{2}+2p$ so that $\boldsymbol \nu_m=4p>0$.  
The abbreviation $m$ is retained at various points in the following calculation to keep some exponents short.
For the particular choice of $m$ we have $m>\max\big\{\frac{N(2-p)}{2},p+1\big\}$.  
Moreover, since $N\ge2$, we also have $m\ge 2+p>3$.
With these choices we get from Proposition \ref{lem:L-infty-quant-p<2} that
\begin{align}\label{est:before-L^m}\nonumber
    \sup_{\frac12 Q_{ R, S}}& |Du|\\
    & 
    \le
    \max\Bigg\{
     \frac{C_\infty}{\eps^{\frac{(2-p)(N+2)}{2p}}}
    \Big(\frac{R^2}{S}\Big)^\frac{N}{4p} \bigg[\biint_{Q_{R,S}}|Du|^{m}\, \dx\dt\bigg]^\frac{1}{2p}
    , \eps\Big(\frac{S}{R^2}
    \Big)^\frac{1}{2-p},\epsilon\mu
    \Bigg\}.
\end{align}
The  integral in the first entry of the maximum can be estimated by means of the $L^m$-estimate \eqref{Lq-est} from Lemma \ref{lem:Lq-est}. In fact, we have
\begin{align*}
    \bigg[\biint_{Q_{R,S}}
    |Du|^m \,\dx\dt\bigg]^\frac1{2p}
    \le
    C_m
    \max\Bigg\{
    \frac{\boldsymbol{\mathfrak T}^\frac{N(2-p)+2p}{4p}}{\delta^\frac{(2-p)(m-p)}{2p^2}}
    \bigg[
    \biint_{Q_{2R,2S}}
    |Du|^p\,\dx\dt \bigg]^\frac1{2p}
    ,
    (\delta\mu)^\frac{m}{2p}
    \Bigg\}
\end{align*}
for every $\delta\in(0,1]$. 
Now, we distinguish between two cases, depending on whether the maximum is  assumed by the former or latter entry. If the second entry is larger,
we choose $\delta$ in terms of $\eps$ such that 
\begin{align*}
    \frac{C_\infty C_m}{\eps^{\frac{(2-p)(N+2)}{2p}}}\,
    \delta^\frac{m}{2p}
    =
    \epsilon\quad\Longleftrightarrow\quad
    \delta=
    \frac{\eps^{\frac{N(2-p)+4}{m}}}{[C_\infty C_m]^\frac{2p}{m}}.
\end{align*}
This specifies the value of $\delta$ 
and leads to the estimate
\begin{align*}
    \frac{C_\infty}{\eps^\frac{(2-p)(N+2)}{2p}}&
    \Big(\frac{R^2}{S}\Big)^\frac{N}{4p} \bigg[\biint_{Q_{R,S}}|Du|^{m}\, \dx\dt\bigg]^\frac{1}{2p}\\
    &\le
    \frac{C_\infty C_m}{\eps^\frac{(2-p)(N+2)}{2p}}
    \Big(\frac{R^2}{S}\Big)^\frac{N}{4p}
    \delta^\frac{m}{2p}\mu^\frac{m}{2p} 
    \le 
    \epsilon\Big(\frac{R^2}{S}\Big)^\frac{N}{4p} \mu^\frac{N(2-p)+4p}{4p}.
\end{align*}
If the first entry is larger,
for the same value of $\delta$ we have
\begin{align*}
    \frac{C_\infty}{\eps^\frac{(2-p)(N+2)}{2p}}&
    \Big(\frac{R^2}{S}\Big)^\frac{N}{4p} \bigg[\biint_{Q_{R,S}}|Du|^{m}\, \dx\dt\bigg]^\frac{1}{2p}\\
    &\le
    \frac{C_\infty C_m}{\eps^\frac{(2-p)(N+2)}{2p}
    \delta^\frac{(2-p)(m-p)}{2p^2}
    }
    \Big(\frac{R^2}{S}\Big)^\frac{N}{4p} 
    \boldsymbol{\mathfrak T}^\frac{N(2-p)+2p}{4p}
    \bigg[\biint_{Q_{2R,2S}}
    |Du|^{p}\, \dx\dt\bigg]^\frac{1}{2p} \\
    &= 
    \frac{[C_\infty C_m]^{1+ \frac{(2-p)(m-p)}{mp}}}{\eps^\theta}
    \Big(\frac{R^2}{S}\Big)^\frac{N}{4p} 
    \boldsymbol{\mathfrak T}^\frac{N(2-p)+2p}{4p}
    \bigg[\biint_{Q_{2R,2S}}
    |Du|^{p}\, \dx\dt\bigg]^\frac{1}{2p}\\
    &\equiv
    \frac{C}{\eps^\theta}
    \Big(\frac{R^2}{S}\Big)^\frac{N}{4p} 
    \boldsymbol{\mathfrak T}^\frac{N(2-p)+2p}{4p}
    \bigg[\biint_{Q_{2R,2S}}
    |Du|^{p}\, \dx\dt\bigg]^\frac{1}{2p}.
\end{align*}
Here, $\theta
    :=
    \frac{(2-p)}{2p^2}[p(N+2)+\frac{(m-p)(N(2-p)+4)}{m}]$ effectively depends on $N$ and $p$.
Combining the two cases with 
\eqref{est:before-L^m} we get \eqref{est:L^infty-L^p}. This finishes the proof.
\end{proof}

Taking a more special geometry of cylinders, the previous result yields a gradient bound that will be useful in later applications.
\begin{corollary}\label{cor:L-infty-est-p<2}
Let $\mu\in(0,1]$, $1<p\le 2$, and $\lm\ge\mu$. Then, for any locally bounded weak solution $u$ to the parabolic system \eqref{eq:diff-systems} in the sense of Definition \ref{def:weak-loc} with assumption \eqref{ass:b}, we have $|Du|\in L^\infty_{\mathrm{loc}}(E_T)$.
Moreover, there exist positive constants $\theta =\theta (N,p)$ and $C=C(N,p, C_o,C_1,C_2)$ such that for every cylinder $Q_{2R}^{(\lm)}(z_o)\Subset E_T$ with $0<R\le1$ we have
\begin{align*}
    \sup_{ Q_{\frac12R}^{(\lm)}(z_o)}& |Du|\\
    & 
    \le
    \max\Bigg\{
     \frac{C\sqrt{\lm }}{\eps^\theta}
    \bigg[ \frac{\boldsymbol\omega}{\lm R} 
    +
    \Big(
    \frac{\boldsymbol\omega}{\lm R}
    \Big)^\frac{2}{p}\bigg]^\frac{N(2-p)+2p}{4p}
    \bigg[\biint_{Q_{2R}^{(\lm)}(z_o)}
    |Du|^{p}\, \dx\dt\bigg]^\frac{1}{2p}
    ,
    \eps\lm
    ,
    \epsilon\mu
    \Bigg\}.
\end{align*}
where 
\begin{equation*}
    \boldsymbol\omega :=\sup_{t\in (t_o-\lambda^{2-p}(2R)^2,t_o]}\,
    \sup_{x\in B_{2R}(x_o)}\big|u(x,t)-(u)_{x_o,2R}(t)\big|.
\end{equation*}
\end{corollary}

\begin{proof}
Recalling the definition of $\boldsymbol{\mathfrak T}$ from Proposition~\ref{lem:L-infty-est-quant-p<2} we have
\begin{align*}
    \Big(\frac{R^2}{S}\Big)^\frac{N}{4p} 
    \boldsymbol{\mathfrak T}^\frac{N(2-p)+2p}{4p}
    &=
    \Big(\frac{R^2}{S}\Big)^\frac{N}{4p} 
    \bigg[ \frac{\boldsymbol\omega}{R} 
    +
    \Big(
    \frac{\boldsymbol\omega}{\sqrt{S}}
    \Big)^\frac{2}{p}\bigg]^\frac{N(2-p)+2p}{4p}
    \\
    &=
    \lm^{\frac{N}{4p}(p-2)}\lm^\frac{N(2-p)+2p}{4p}
    \bigg[ \frac{\boldsymbol\omega}{\lm R} 
    +
    \Big(
    \frac{\boldsymbol\omega}{\lm R}
    \Big)^\frac{2}{p}\bigg]^\frac{N(2-p)+2p}{4p}\\
    &= 
    \sqrt{\lm}
    \bigg[ \frac{\boldsymbol\omega}{\lm R} 
    +
    \Big(
    \frac{\boldsymbol\omega}{\lm R}
    \Big)^\frac{2}{p}\bigg]^\frac{N(2-p)+2p}{4p}.
\end{align*}
Moreover, we have $(S/R^2)^\frac1{2-p}=\lm$, and 
\begin{align*}
    \Big(\frac{R^2}{S}\Big)^\frac{N}{4p} \mu^\frac{m}{2p}
    &=
    \lm^{\frac{N}{4p}(p-2)}\mu^{\frac{N}{4p}(2-p)+1}
    =
    \Big(\frac{\mu}{\lm}\Big)^{\frac{N}{4p}(2-p)}\mu\le \mu.
\end{align*}
Inserting this into \eqref{est:L^infty-L^p}, we obtain the claim.
\end{proof}

\subsection{Qualitative gradient bound in the super-quadratic case}
In the super-quadratic case $p\ge 2$ the assumption of local boundedness is superfluous for weak solutions, since they are automatically locally bounded. This has consequences for the $L^\infty$-gradient bound.  We start with the qualitative bound, which is stated as follows. 
\begin{proposition}\label{lem:L-infty-est-p>2-intrinsic}
Let $\mu\in(0,1]$ and $p\ge 2$. Then, for any  weak solution $u$ to the parabolic system \eqref{eq:diff-systems} in the sense of Definition \ref{def:weak-loc} with assumption \eqref{ass:b}, we have
\begin{equation*}
    |Du|\in L^\infty_{\mathrm{loc}}(E_T).
\end{equation*}

\end{proposition}
\begin{proof}
As in the case $1<p\le 2$ the starting point is the energy estimate from Lemma~\ref{energy-est} in the form \eqref{energy-intrinsic}. Since now $p\ge 2$, the term with exponent $\frac{p}2+\alpha$ is dominant on the right-hand side, whereas the other term is controlled by the elementary estimate
\begin{align*}
    \big(\mu^2 +|Du|^2\big)^{1+\alpha}
    &\le
    \big(\mu^2 +|Du|^2\big)^{\frac{p}2+\alpha} +1.
\end{align*}
Using this inequality in \eqref{energy-intrinsic} we obtain
\begin{align*}
    \tfrac1{1+\alpha}\sup_{t\in (-r^2,0]}&
     \int_{B_r\times\{t\}} \big(\mu^2+|Du|^2\big)^{1+\alpha}\, \dx +
     \iint_{Q_{r}}\big(\mu^2 +|Du|^2\big)^{\frac{p-2}{2}+\alpha}
     |D^2u|^2\dx\dt\\
     &\qquad\qquad\qquad\qquad\le
     \frac{C(1+\alpha)}{(R-r)^2}
     \iint_{Q_{R}}
     \Big[ \big(\mu^2+|Du|^2\big)^{\frac{p}2+\alpha}+1\Big]\,\dx\dt
\end{align*}
for any $0<r<R$ such that $Q_R\Subset E_T$. 
As in the sub-quadratic case introduce 
$$
    w=\big(\mu^2+|Du|^2\big)^{\frac{p+2\alpha}4} 
$$
and rewrite the above estimate as
\begin{align}\label{energy-est-super}\nonumber
     \tfrac1{1+\alpha}\sup_{t\in (-r^2,0]}&
     \int_{B_r\times\{t\}} w^\frac{4+4\alpha}{p+2\alpha}\, \dx+
     \tfrac{1}{(p+2\alpha)^2}
     \iint_{Q_{r}}|\nabla w|^2
    \,\dx\dt\\
    &\le 
     \frac{C(1+\alpha)}
     {(R-r)^2}
     \iint_{Q_{R}}
     \big[w^2 +1\big]\,\dx\dt.
\end{align}
The symbolism $\varrho_i$, $\tilde\varrho_i$, $B_i$, $Q_i$, $\widetilde B_i$, $\widetilde Q_i$, and $\kappa$ from \eqref{Eq:k_n-p<2} is retained. However, the value of $\alpha_o$, the recursion formula for $\alpha_i$, the definition of $p_i$ and $q_i$, must be adapted to the case $p\ge 2$.  To this end, we define
\begin{equation*}
    \left\{
    \begin{array}{c}
    \alpha_o=0,\quad
    \alpha_{i+1}
    =\kappa\alpha_i+\frac2N,\\[6pt]
    p_i=p+2\alpha_i,\quad\displaystyle q_i=\frac{4+4\alpha_i}{p+2\alpha_i},
    \end{array}
    \right.
\end{equation*}
for $i\in\N_0$.
The definition of $\alpha_i$ immediately yields $\alpha_i=\kappa^i-1$. Moreover, we have \eqref{p_i+1} at hand, namely, $p_{i+1}=(p+2\alpha_i)\tfrac{N+q_i}{N}$. To continue, we choose the standard cut-off  function $\zeta\equiv \zeta_i\in C^\infty(\widetilde Q_i,[0,1])$ such that $\zeta_i$ 
vanishes on the parabolic boundary $\partial_{\mathrm{par}}\widetilde{Q}_i$, $\zeta_i$ is identically equal to one on $Q_{i+1}$, and the gradient satisfies $|\nabla\zeta_i |\le 2^{i+4}/R$. 
An application of the Sobolev embedding \cite[Chapter I, Proposition~3.1]{DB} with $p=2$ and $m=q_i$ gives that
\begin{align}\label{Sobolev:p>2-intrinsic}\nonumber
    \iint_{Q_{i+1}}&\big(\mu^2+|Du|^2\big)^\frac{p_{i+1}}{2}\,\dx\dt
    \le
    \iint_{\widetilde{Q}_i}(w\zeta_i)^{2\frac{N+q_i}N}\,\dx\dt\\\nonumber
    &\le C_{\rm Sob}(N,q_i)\iint_{\widetilde{Q}_i}|\nabla (w\zeta_i)|^{2}\,\dx\dt
    \bigg[\sup_{\tau\in(-\tilde\varrho_i^2,0]}
    \int_{\widetilde{B}_i\times\{\tau\}}(w\zeta_i)^{q_i}\,\dx\bigg]^{\frac{2}N
    }\\\nonumber
    &\le 
    C  \iint_{\widetilde{Q}_i}
    \Big[ \zeta_i^2|\nabla w|^2 +w^2|\nabla \zeta_i|^2
    \Big]\dx\dt
    \bigg[\sup_{\tau\in(-\tilde\varrho_i^2,0]}
    \int_{\widetilde{B}_i\times\{\tau\}}w^{q_i}\,\dx\bigg]^{\frac{2}N}
    \\
    &\le 
    C  \iint_{\widetilde{Q}_i}
    \bigg[ |\nabla w|^2 +\frac{2^{2i+8}w^2}{R^2}
    \bigg]\dx\dt    
    \bigg[\sup_{\tau\in(-\tilde\varrho_i^2,0]}
    \int_{\widetilde{B}_i\times\{\tau\}}w^{q_i}\dx\bigg]^{\frac{2}N}
\end{align}
for a constant $C=C(N,p)$.
In the application of the Sobolev inequality, as indicated, there is a dependence of the constant on $q_i$. However, $q_i$ only takes values in the interval $[ \frac4p,2]$. Therefore,  the continuous dependence of the Sobolev constant on $q_i$ eventually yields a dependence on $p$. To estimate the right-hand side integrals in the last displayed inequality, we use the energy estimate~\eqref{energy-est-super} applied on $\widetilde Q_i\subset Q_i$. In fact, we have
\begin{align*}
     \iint_{\widetilde Q_i}|\nabla w|^2
    \,\dx\dt
     &\le 
    C
     \frac{(p+2\alpha_i)^32^{2i+6}}{R^2}
     \iint_{Q_i}
     \big[w^2 +1\big]\,\dx\dt
\end{align*}
and
\begin{align*}
     \sup_{t\in (-\tilde\varrho_i^2,0]}
     \int_{\widetilde B_i\times\{t\}} w^{q_i}\, \dx
    &\le C
     \frac{(p+2\alpha_i)^22^{2i+6}}{R^2}
     \iint_{Q_i}
     \big[w^2 +1\big]\,\dx\dt.
\end{align*}
Inserting this into \eqref{Sobolev:p>2-intrinsic} yields
\begin{align*}
    \iint_{Q_{i+1}}&\big(\mu^2+|Du|^2\big)^\frac{p_{i+1}}{2}\,\dx\dt\\
    &\le
    C(p+2\alpha_i)^{3+\frac{4}{N}}
    \Bigg[
    \frac{2^{2i}}{R^2}
    \iint_{Q_i}
     \Big[\big(\mu^2+|Du|^2\big)^\frac{p_i}{2} +1\Big]\,\dx\dt
    \Bigg]^\kappa. 
\end{align*}
With the help of
\begin{align*}
    (p+2\alpha_i)^{3+\frac4N}2^{2i\kappa}
    &
    \le
    \big[p(1+\alpha_i)\big]^{1+2\kappa}2^{2i\kappa}
    =\big[ p\kappa^i\big] ^{1+2\kappa}2^{2i\kappa}
    \le
    p^{1+2\kappa} (4\kappa^3)^{i\kappa},
\end{align*}
this can be simplified, i.e.~we have
\begin{align*}
    \iint_{Q_{i+1}}\big(\mu^2+|Du|^2\big)^\frac{p_{i+1}}{2}\,\dx\dt
    \le
    C
    \Bigg[
    \frac{\boldsymbol b^i}{R^2}
    \iint_{Q_i}
     \Big[\big(\mu^2+|Du|^2\big)^\frac{p_i}{2} +1\Big]\,\dx\dt
    \Bigg]^\kappa
\end{align*}
with $\boldsymbol b =\boldsymbol b(N)=4\kappa^3$ and $C=C(N,p,C_o,C_1,C_2)$.  Divide both sides by $|Q_{i+1}|$ to get
\begin{align*}
    \biint_{Q_{i+1}}& \big(\mu^2+|Du|^2\big)^\frac{p_{i+1}}{2}\,\dx\dt\\
    &\le
     C\boldsymbol b^{i\kappa}
     \underbrace{\frac{|Q_i|^\kappa}{R^{2\kappa}|Q_{i+1}|}}_{\le 2^{N+2}}
    \Bigg[
    \biint_{Q_i}
     \Big[\big(\mu^2+|Du|^2\big)^\frac{p_i}{2} +1\Big]\,\dx\dt
    \Bigg]^\kappa\\
    &\le
     C\boldsymbol b^{i\kappa}
    \Bigg[
    \biint_{Q_i}
     \Big[\big(\mu^2+|Du|^2\big)^\frac{p_i}{2} +1\Big]\,\dx\dt
    \Bigg]^\kappa.
\end{align*}
Therefore, adding $1$ to the left-hand side we get
\begin{align*}
    \boldsymbol Y_{i+1}^{p_{i+1}}&
    :=
    \biint_{Q_{i+1}}\Big[ \big(\mu^2+|Du|^2\big)^\frac{p_{i+1}}{2}+1\Big]\,\dx\dt\\
    &\le
     C\boldsymbol b^{i\kappa}
    \Bigg[
    \underbrace{
    \biint_{Q_i}
     \Big[\big(\mu^2+|Du|^2\big)^\frac{p_i}{2} +1\Big]\,\dx\dt}_{=\boldsymbol Y_i^{p_i}}
    \Bigg]^\kappa
    \equiv\big( C\boldsymbol b^i\boldsymbol Y_i^{p_i}\big)^\kappa.
\end{align*}
Iterating this inequality yields 
\begin{align*}
    \boldsymbol Y_{i}^{p_i}
    &\le
    \prod_{j=1}^{i} C^{\kappa^{i-j+1}}
    \prod_{j=1}^{i} \boldsymbol b^{j\kappa^{i-j+1}}
    \boldsymbol Y_o^{p\kappa^{i}},
\end{align*}
so that
\begin{align*}
    \boldsymbol Y_{i}
    &\le
    \Bigg[ \prod_{j=1}^{i} C^\frac{\kappa^{i-j+1}}{2(\kappa^i-1)}\Bigg]^\frac{2(\kappa^i-1)} {p_i}
    \Bigg[\prod_{j=1}^{i} \boldsymbol b^\frac{j\kappa^{i-j+1}}{2(\kappa^i-1)}\Bigg]^\frac{2(\kappa^i-1)} {p_i}
    \boldsymbol Y_o^\frac{p\kappa^{i}}{p_i}.
\end{align*}
Now, Lemma~\ref{lem:A}, and
\begin{equation*}
    \lim_{i\to\infty}\frac{2 (\kappa^i-1)}{p_i}=1
    \quad\mbox{and}\quad
     \lim_{i\to\infty}\frac{p\kappa^i}{p_i}= \frac{p}{2}
\end{equation*}
allow us to pass to the limit $i\to\infty$ in the preceding inequality to conclude that
\begin{align*}
    \limsup_{i\to\infty}\boldsymbol Y_{i}
    &\le
    C^{\frac{N+2}{4}}\boldsymbol b^\frac{(N+2)^2}{8}\boldsymbol Y_o^\frac{p}{2}.
\end{align*}
This immediately implies
\begin{align*}
    \sup_{\frac12 Q_{R}}|Du|
    &\le
    C\Bigg[
     \biint_{Q_{R}}
     \big[|Du|^p+1\big]\,\dx\dt
   \Bigg]^\frac{1}{2}
\end{align*}
for a constant  $C=C(N,p,C_o,C_1,C_2)$. A covering argument then shows $|Du|\in L^{\infty}_{\rm loc}(E_T)$. This proves the claim.
\end{proof}

\subsection{Quantitative gradient bound in the super-quadratic case}
In this section we derive an improved quantitative $L^\infty$-gradient bound for weak solutions to parabolic $p$-Laplacian systems with Lipschitz continuous coefficients.

\begin{proposition}\label{lem:L-infty-est-p>2-intrinsic-improved}
Let $\mu\in(0,1]$ and $p> 2$. 
There exists a constant $C$ depending on $N,p,C_o,C_1,C_2$ such that for any  weak solution $u$ to the parabolic system \eqref{eq:diff-systems} in the sense of Definition \ref{def:weak-loc} with assumption \eqref{ass:b}, any $\sigma\in (0,1)$, any $\epsilon\in(0,1]$, and  any cylinder $Q_{R,S}\equiv Q_{R,S}(z_o)\Subset E_T$ with $R\in(0,1]$ and $S>0$ we have the quantitative $L^\infty$-gradient bound 
\begin{align*}
    \sup_{Q_{\sigma R,\sigma S}}|Du|
    \le
    \max\Bigg\{
   \frac{C}{\eps^{\frac{p(N+2)}{4}}}
    \bigg[
    \frac{ S/R^2 }{(1-\sigma)^{N+2}}\biint_{Q_{R,S}}
    |Du|^p\dx\dt
    \bigg]^\frac{1}2, \epsilon\Big(\frac{R^2}{S}\Big)^\frac1{p-2}, \eps\mu\Bigg\}.
\end{align*} 
\end{proposition}
\noindent If we take $S\equiv \lambda^{2-p}R^2$ we have the following straightforward consequence.
\begin{corollary}\label{cor:L-infty-est-p>2-intrinsic-improved}
   In cylinders of the form 
$
    Q_R^{(\lambda)}(z_o)= B_R(x_o)\times \big( t_o-\lambda^{2-p}R^2, t_o\big]
$
with $R\in(0,1]$, the gradient bound reads as
\begin{align*}
    \sup_{Q_{\sigma R}^{(\lambda)}(z_o)}|Du|
    &
    \le
    \max\Bigg\{
    \frac{C}{\eps^{\frac{p(N+2)}{4}}}\bigg[ 
    \frac{\lambda^{2-p}}{
    (1-\sigma)^{N+2}}
    \biint_{Q_{R}^{(\lambda)}(z_o)}
    |Du|^p\dx\dt
    \bigg]^\frac1{2}, \epsilon\lambda, \eps\mu\Bigg\}.
\end{align*}  
\end{corollary}

\begin{proof}[{\bf Proof of Proposition~\ref{lem:L-infty-est-p>2-intrinsic-improved}}]
Writing the energy inequality from Proposition~\ref{prop:estD2u} with $z_o=(0,0)$, we have
\begin{align*}
    \sup_{\tau\in (-S,0]} & \int_{B_{R}\times\{\tau\}}v\zeta^2\,\dx +
      \iint_{Q_{R,S}}
      \underbrace{\big(\mu^2 +|Du|^2\big)^\frac{p-2}2 |D^2u|^2}_{\ge |Du|^{p-2}|D^2u|^2}\Phi(|Du|^2) 
      \zeta^2 \,\dx\dt\\
        &\qquad\quad\le
      C \iint_{Q_{R,S}}
     \underbrace{\big(\mu^2 +|Du|^2\big)^\frac{p-2}2|Du|^2}_{\le C(p)(|Du|^p+\mu^p)}
     \Phi(|Du|^2)
      \big(\zeta^2+|\nabla\zeta|^2\big) \,\dx\dt\\   
      &\qquad\quad\phantom{\le\,}+
      C\iint_{Q_{R,S}}
      \underbrace{\big(\mu^2 +|Du|^2\big)^\frac{p-2}2|Du|^2}_{\le C(p)(|Du|^p+\mu^p)} \Phi'(|Du|^2)|Du|^2
      \zeta^2\,\dx\dt\\
    &\qquad\quad\phantom{\le\,}+
      C\iint_{Q_{R,S}}v\,\big|\partial_t\zeta^2\big|\,\dx\dt.
\end{align*}
The main difference to the case of constant or only time-dependent coefficients is the integral on the right-hand side containing the derivative $\Phi'$. This term prevents us from choosing $
\Phi(\tau):=\big( \sqrt{\tau}-k\big)_+^{p-2}
$ as in the case of constant coefficients, cf.~\cite[p.~234, Lemma~4.2]{DB}. 
For this reason, we choose 
\begin{equation*}
    \Phi(\tau):=\big( \sqrt{\tau}-k\big)_+^{p}.
\end{equation*}
The fact that $\Phi$ is neither bounded nor Lipschitz is immaterial, since we have shown $|Du|$ is locally bounded.
For $|Du|\ge k$, otherwise there is nothing to compute since in this case $v=0$, we have
\begin{align*}
    v
    &
    =
    \int_{k^2}^{|Du|^2}\big( \sqrt{\tau}-k\big)^{p}\,\d\tau
    =
    2\int_{k}^{|Du|} ( s-k)^{p}s\,\ds\\
    &\ge
    2\int_{k}^{|Du|} ( s-k)^{p+1}\,\ds
    =
    \frac{2}{p+2}\big(|Du|-k\big)^{p+2}.
\end{align*}
Similarly, we obtain the bound from above. In fact, we have
\begin{align*}
    v
    &=
    2\int_{k}^{|Du|} ( s-k)^{p}s\,\ds
    \le
    2\int_{k}^{|Du|} s^{p+1}\,\ds
    \le
    \frac{2}{p+2}|Du|^{p+2}.
\end{align*}
Hence,
\begin{equation*}
     \frac{2}{p+2}\big(|Du|-k\big)_+^{p+2}
     \le
     v
     \le
     \frac{2}{p+2}|Du|^{p+2}
     \mathbf 1_{\{|Du|>k\}}.
\end{equation*}
Next, we estimate the second term on the left-hand side from below. 
Namely,
\begin{align*}
    |Du|^{p-2}
     |D^2u|^2\big(|Du|-k\big)_+^{p}
     &=
     \mathbf 1_{\{|Du|>k\}}
     \underbrace{|Du|^{p-2}}_{\ge k^{p-2}}
     |D^2u|^2\big(|Du|-k\big)^{p}\\
     &\ge
     k^{p-2}
     \big(|Du|-k\big)_+^{p}
     |D^2u|^2.
\end{align*}
Due to Kato's inequality, the right-hand side of the last display can be further estimated from below if we observe that
\begin{align*}
    \Big|\nabla \big(|Du|-k\big)_+^\frac{p+2}{2}\Big|^2
    &\le
    \tfrac{1}4 (p+2)^2 \big(|Du|-k\big)_+^{p}
    \big|\nabla |Du|\big|^2\\
    &\le
    \tfrac{1}4 (p+2)^2  \big(|Du|-k\big)_+^{p}|D^2u|^2.
\end{align*}
The last two estimates combined result in
\begin{align*}
    |Du|^{p-2}&
     |D^2u|^2\big(|Du|-k\big)_+^{p}
     \ge\frac{4k^{p-2}}{(p+2)^2}
    \Big|\nabla \big(|Du|-k\big)_+^\frac{p+2}{2}\Big|^2.
\end{align*}
It remains to consider the first and second terms on the right-hand side of the energy estimate. The first one can be estimated  by
\begin{align*}
    \big(|Du|^p+\mu^p\big)\Phi\big(|Du|^2\big)
    &=
     \big(|Du|^p+\mu^p\big)
     \big( |Du|-k\big)_+^{p}\\
     &\le
     \mathbf 1_{\{|Du|>k\}}\big(|Du|^{2p}+ \mu^p|Du|^{p}\big),
\end{align*}
while the second one can be bounded by
\begin{align*}
    \big(|Du|^p+\mu^p\big)\Phi'
    \big(|Du|^2\big)|Du|^2
    &=
     \tfrac12 p\big(|Du|^p+\mu^p\big)
     \big( |Du|-k\big)_+^{p-1}|Du|\\
     &\le
     \mathbf 1_{\{|Du|>k\}}
     \big(|Du|^{2p}+ \mu^p|Du|^{p}\big).
\end{align*}
Using the estimates of the individual terms in the energy inequality, we eventually obtain
\begin{align*}
    \sup_{\tau\in (-S,0]}&
     \int_{B_R\times\{\tau\}} \big(|Du|-k\big)_+^{p+2} \zeta^2\, \dx
     +
     k^{p-2}
     \iint_{Q_{R,S}}
    \Big|\nabla \big(|Du|-k\big)_+^\frac{p+2}{2}\Big|^2\zeta^2\,\dx\dt\\
     &\le C
     \iint_{Q_{R,S}}
     \mathbf 1_{\{|Du|>k\}}\big(
     |Du|^{2p}+\mu^p|Du|^{p}\big)\big(
     |\nabla\zeta|^2+\zeta^2\big)\,\dx\dt\\
     &\phantom{\le\,}+C
     \iint_{Q_{R,S}}|Du|^{p+2}\mathbf 1_{\{|Du|>k\}}\big|\partial_t\zeta^2\big|\,\dx\dt
\end{align*}
with a constant $C=C(p,C_o,C_1,C_2)$.
Note that the preceding energy inequality holds true for any general cylinder $Q_{R,S}\subset E_T$, for any $k>0$, and any cut-off function $\zeta$ vanishing on the parabolic boundary $\partial_{\mathrm{par}}Q_{R,S}$. Now,  we define shrinking families of cylinders $Q_n:=
Q_{\rho_n,\theta_n}$ and $\widetilde Q_n:= Q_{\tilde\rho_n,\tilde\theta_n}$, where
$\rho_n$, $\tilde\rho_n$, $\theta_n$, $\tilde \theta_n$ are defined according to
\begin{equation*}
    \rho_n:= \sigma R +\frac{1-\sigma}{2^n}R,
    \quad\
    \theta_n:= \sigma S +\frac{1-\sigma}{2^n}S
\end{equation*}
and 
\begin{equation*}
    \tilde\rho_n:=\tfrac12\big( \rho_n+\rho_{n+1}\big)=\sigma R +\frac{3(1-\sigma)}{2^{n+2}}R,
    \quad
    \tilde\theta_n:=\tfrac12\big( \theta_n+\theta_{n+1}\big)
    =
    \sigma S +\frac{3(1-\sigma)}{2^{n+2}}S.
\end{equation*}
These cylinders are arranged in such a way that $Q_{n+1}\subset \widetilde 
Q_n\subset Q_n$ for any $n\in\N_0$. The level $k$ in the energy estimate is fixed by an increasing sequence of levels
\begin{equation*}
    k_n:=k-\frac{k}{2^n},
\end{equation*}
where $k>0$ is a number to be determined in a universal way later in the course of the proof. Since the above energy inequality shall be applied on the cylinders $Q_n$, we specify the cutoff function 
$\zeta \equiv \zeta_n$ in such a way that it vanishes on the parabolic boundary of $Q_n$ on the one hand, and is identical to 1 on the intermediate cylinder $\widetilde Q_n$ on the other hand. Moreover, we require that
$$
    |\nabla\zeta_n|\le \frac{2^{n+3}}{(1-\sigma)R},
    \quad\mbox{and}\quad
    |\partial_t\zeta_n|\le\frac{2^{n+3}}{(1-\sigma)S}.
$$
With respect to these choices, the energy inequality for $Q_n$ reads as follows
\begin{align}
\label{est:en-DG-p>2}\nonumber
        \sup_{\tau\in (-\theta_n,0]}&
     \int_{B_n\times\{\tau\}} \Big[\big(|Du|-k_{n+1}\big)_+^\frac{p+2}2 \zeta_n\Big]^2\, \dx\\\nonumber
     &\phantom{\le\,}
     +
     \underbrace{k_{n+1}^{p-2}}_{\ge (k/2)^{p-2}}
     \iint_{Q_n}
    \Big|\nabla\big[ \big(|Du|-k_{n+1}\big)_+^\frac{p+2}2\zeta_n\big]\Big|^2\dx\dt\\\nonumber
    &\le \frac{C2^{2n}}{(1-\sigma)^2R^2}
     \iint_{Q_n}
     \mathbf 1_{\{|Du|>k_{n+1}\}}\big( |Du|^{2p}+\mu^p|Du|^{p}\big)\,\dx\dt
     \\
     &\phantom{\le\,}+\frac{C2^n}{(1-\sigma)S}
     \iint_{Q_{_n}}\mathbf 1_{\{|Du|>k_{n+1}\}}
     |Du|^{p+2}\,\dx\dt.
\end{align}
To shorten the notation we introduce 
\begin{equation*}
    \boldsymbol Y_n:=\iint_{Q_n} \big( |Du|-k_n\big)_+^{p+2}\,\dx\dt.
\end{equation*}
The integral $\boldsymbol Y_n$ can be easily estimated from below by reducing the domain of integration. Indeed, we have
\begin{align*}
    \boldsymbol Y_n
    &\ge
    \iint_{Q_n\cap\{|Du|>k_{n+1}\} } \big( |Du|-k_n\big)_+^{p+2}\,\dx\dt\\
    &\ge
    \big( k_{n+1}-k_n\big)^{p+2}
    \Big| Q_n \cap \big\{ |Du|>k_{n+1}\big\}
   \Big|\\
   &=\frac{k^{p+2}}{2^{(n+1){(p+2)}}}\Big| Q_n \cap \big\{ |Du|>k_{n+1}\big\}
   \Big|.
\end{align*}
Next, we estimate $\boldsymbol Y_{n+1}$ from above. For this purpose we first increase in $\boldsymbol Y_{n+1}$ the domain of integration from $Q_{n+1}$ to $\widetilde Q_n$ and then apply H\"older's inequality with exponents $\frac{N+2}N$ and $\frac{N+2}2$. Subsequently, we use the fact that $\zeta_n=1$ on $\widetilde Q_n$ and the last displayed  measure estimate. In this way we obtain 
\begin{align*}
    \boldsymbol Y_{n+1}
    &\le 
    \iint_{\widetilde Q_n} \Big[\big(|Du|-k_{n+1}\big)_+^\frac{p+2}2 \zeta_n\Big]^2\, \dx\dt\\
    &\le
    \bigg[ \iint_{\widetilde Q_n} \Big[\big(|Du|-k_{n+1}\big)_+^\frac{p+2}2 \zeta_n\Big]^{2\frac{N+2}
    {N}}\dx\dt\bigg]^\frac{N}{N+2}
    \Big|\widetilde Q_n \cap \big\{ |Du|>k_{n+1}\big\}
   \Big|^\frac{2}{N+2}\\
   &\le
    \bigg[ \iint_{ Q_n} \Big[\big(|Du|-k_{n+1}\big)_+^\frac{p+2}2 \zeta_n\Big]^{2\frac{N+2}{N}}\dx\dt\bigg]^\frac{N}{N+2}
    \bigg[
    \frac{2^{(n+1)(p+2)}}{k^{p+2}}\boldsymbol Y_n
    \bigg]^\frac{2}{N+2}.
\end{align*}
The first integral on the right-hand side can be controlled using Sobolev's inequality \cite[Chapter I, Proposition~3.1]{DB} with $p=2$ and $m=2$. The application yields
\begin{align*}
    \iint_{Q_n} &\Big[\big(|Du|-k_{n+1}\big)_+^\frac{p+2}2 \zeta_n\Big]^{2\frac{N+2}{N}}\dx\dt
    \\
    &\le
    C
    \iint_{ Q_n}
    \Big|\nabla\big[ \big(|Du|-k_{n+1}\big)_+^\frac{p+2}2\zeta_n\big]\Big|^2\dx\dt\\
    &\phantom{\le\,C}\cdot
    \bigg[
    \sup_{\tau\in (-\theta_n,0]}\int_{B_n\times\{\tau\}}
    \Big[\big(|Du|-k_{n+1}\big)_+^\frac{p+2}2 \zeta_n\Big]^{2}\dx
    \bigg]^\frac{2}{N}\\
    &\le
     Ck^{-(p-2)}
     \big[
     \mbox{right-hand side of \eqref{est:en-DG-p>2}}
     \big]^{\frac{N+2}{N}}.
\end{align*}
To obtain the last line, we used \eqref{est:en-DG-p>2}. Inserting this inequality in the estimate of $\boldsymbol Y_{n+1}$ we obtain
\begin{align*}
    \boldsymbol Y_{n+1}
    &\le
    \frac{C 2^{n(p+4)}}{(1-\sigma)^2} k^{-\frac{N(p-2)+2(p+2)}{N+2}}
    \boldsymbol Y_{n}^\frac{2}{N+2}\\
    &\quad\cdot
    \bigg[
    \frac{1}{R^2}
     \iint_{Q_n}
     \mathbf 1_{\{|Du|>k_{n+1}\}}
     \big(|Du|^{2p}+\mu^p|Du|^{p}\big)\,\dx\dt\\
&\phantom{\qquad\qquad\qquad\cdot\bigg[\,}+\frac{1}{S}
     \iint_{Q_{_n}}\mathbf 1_{
     \{|Du|>k_{n+1}\}
     }
     |Du|^{p+2}
     \,\dx\dt
    \bigg].
\end{align*}
To derive a homogeneous estimate we reduce the power of $|Du|$ in the first integral from $2p$ to $p+2$ and estimate a part of
$|Du|^{2p}$ by its supremum. More precisely, we have
\begin{align*}
    \iint_{Q_n}&
     \mathbf 1_{\{|Du|>k_{n+1}\}}
     \big(|Du|^{2p}+\mu^p|Du|^{p}\big)\,\dx\dt \\
     &\le
     \Big( 1+\frac{\mu^p}{k^p}\Big)
     \iint_{Q_n} 
     \mathbf 1_{\{|Du|>k_{n+1}\}}
     |Du|^{2p} \,\dx\dt \\
     &\le
      C\bigg(\sup_{Q_n} |Du|\bigg)^{p-2}\Big( 1+\frac{\mu^p}{k^p}\Big)
     \iint_{Q_n}
     \mathbf 1_{\{|Du|>k_{n+1}\}}|Du|^{p+2}\,\dx\dt.
\end{align*}
To obtain the second-to-last line we used that $|Du|>k_{n+1}\ge \frac12 k$ on the domain of integration. 
Furthermore, reducing the domain of integration from $Q_n$ to the superlevel set $Q_n\cap\{|Du|>k_{n+1}\}$ and using the relation between $k_n$ and $k_{n+1}$, i.e.
$$
    k_n=k_{n+1}\frac{2^{n+1}-2}{2^{n+1}-1},
$$
we obtain
\begin{align*}
    \boldsymbol Y_{n}
    &\ge
    \iint_{Q_n}\big(|Du|-k_n\big)_+^{p+2}\mathbf 1_{\{ |Du|>k_{n+1}\}} \,\dx\dt\\
    &=
    \iint_{Q_n}\bigg[|Du|-k_{n+1}\frac{2^{n+1}-2}{2^{n+1}-1}\bigg]^{p+2}\mathbf 1_{\{ |Du|>k_{n+1}\}} \,\dx\dt\\
    &\ge
    \iint_{Q_n}|Du|^{p+2}\bigg[1-\frac{2^{n+1}-2}{2^{n+1}-1}\bigg]^{p+2}\mathbf 1_{\{ |Du|>k_{n+1}\}} \,\dx\dt\\
    &\ge
    \frac{1}{2^{(n+1)(p+2)}}\iint_{Q_n}|Du|^{p+2}
    \mathbf 1_{\{ |Du|>k_{n+1}\}} \,\dx\dt,
\end{align*}
so that
\begin{align*}
    \iint_{Q_n}&
     \mathbf 1_{\{|Du|>k_{n+1}\}}|Du|^{p+2}\,\dx\dt
     \le
     2^{(n+1)(p+2)}
     \boldsymbol Y_{n}.
\end{align*}
Inserting this inequality above we obtain
\begin{align*}
    \iint_{Q_n}
     \mathbf 1_{\{|Du|>k_{n+1}\}}&
     \big(|Du|^{2p}+\mu^p|Du|^{p}\big)\,\dx\dt\\
     &\le
      C\bigg(\sup_{Q_n} |Du|\bigg)^{p-2}\bigg( 1+\frac{\mu^p}{k^p}\bigg)
    2^{(n+1)(p+2)}\boldsymbol Y_n.
\end{align*}
However, this leads us to the recursive inequality
\begin{align*}
    \boldsymbol Y_{n+1}
    &\le
    \frac{C 2^{2n(p+3)}}{(1-\sigma)^2} k^{-\frac{N(p-2)+2(p+2)}{N+2}}
    \boldsymbol Y_{n}^{1+\frac{2}{N+2}}
    \bigg[
    \frac{1}{R^2}\bigg(\sup_{Q_n} |Du|\bigg)^{p-2}
    \bigg( 1+\frac{\mu^p}{k^p}\bigg)
    +\frac1{S}
    \bigg].
\end{align*}
Here, we require that $k\ge \eps\mu$. Now, if for some $n\in \N_0$
\begin{align*}
    \bigg(\tfrac1{\eps}\sup_{Q_n} |Du|\bigg)^{p-2}
    &\le
    \frac{R^2}{S},
\end{align*}
the claimed $L^\infty$-gradient bound follows since $Q_{\sigma R,\sigma S}\subset Q_n$ which will yield
\begin{align}\label{choice-k:trivial-case}
    \sup_{Q_{\sigma R,\sigma S}}|Du|
    \le
    \sup_{Q_{n}}|Du|
    \le 
    \epsilon \Big( \frac{R^2}{S}\Big)^\frac{1}{p-2}
    .
\end{align}
Otherwise, for any $n\in\N_0$ the recursive inequality simplifies to
\begin{align*}
    \boldsymbol Y_{n+1}
    &\le
    \frac{C \boldsymbol b^{n}}{(1-\sigma)^2R^2\eps^p} k^{-\frac{N(p-2)+2(p+2)}{N+2}}\bigg(\sup_{Q_n} |Du|\bigg)^{p-2}
    \boldsymbol Y_{n}^{1+\frac{2}{N+2}}\\
    &\le
    \frac{C \boldsymbol b^{n}}{(1-\sigma)^2R^2\eps^p} k^{-\frac{N(p-2)+2(p+2)}{N+2}}\bigg(\sup_{Q_{R,S}} |Du|\bigg)^{p-2}
    \boldsymbol Y_{n}^{1+\frac{2}{N+2}}
\end{align*} 
with $\boldsymbol b:=4^{p+3}$. 
At this stage, we
apply the lemma of geometric convergence \cite[Chapter I, Lemma 4.1]{DB} with $C,\alpha$ replaced by 
$$
    \frac{C \boldsymbol }{(1-\sigma)^2R^2\eps^p} k^{-\frac{N(p-2)+2(p+2)}{N+2}}\bigg(\sup_{Q_{R,S}} |Du|\bigg)^{p-2},
    \quad
    \mbox{and}
    \quad
    \tfrac2{N+2},
$$
to conclude that $\boldsymbol Y_n\to 0$ as $n\to\infty$ provided that
\begin{align*}
    \boldsymbol Y_o
    &\le
    \bigg[
    \frac{C  }{(1-\sigma)^2R^2\eps^p} k^{-\frac{N(p-2)+2(p+2)}{N+2}}\bigg(\sup_{Q_{R,S}} |Du|\bigg)^{p-2}
    \bigg]^{-\frac{N+2}{2}}
    \boldsymbol b^{-\frac{(N+2)^2}{4}}.
\end{align*}
We resolve this inequality for $k$, and obtain that $k$ must satisfy
\begin{align*}
    k
    &\ge
    \bigg[\frac{C \boldsymbol b^\frac{N+2}2 }{(1-\sigma)^2R^2\eps^p}\bigg]^\frac{N+2}{N(p-2)+2(p+2)}
    \bigg(\sup_{Q_{
    R,S}} |Du|
    \bigg)^\frac{(N+2)(p-2)}{N(p-2)+2(p+2)}\\
    &\qquad\qquad\qquad\qquad\qquad\qquad\cdot
    \bigg[\iint_{Q_{R,S}}
    |Du|^{p+2}\dx\dt
    \bigg]^\frac{2}{N(p-2)+2(p+2)}.
\end{align*}
Apparently, the integrability exponent in the integral on the right-hand side can be reduced by estimating a part of the integrand by the $L^\infty$-norm of $|Du|$. This leads to a stronger condition on $k$. The use of
\begin{align*}
    \iint_{Q_{R,S}}
    |Du|^{p+2}\dx\dt
    &\le 
    \bigg(\sup_{Q_{R,S}}|Du|\bigg)^{2}
    \iint_{Q_{R,S}}
    |Du|^{p}\dx\dt
\end{align*}
leads to the requirement
\begin{align*}
    k
    &\ge
    \bigg[\frac{C \boldsymbol b^\frac{N+2}2 }{(1-\sigma)^2R^2\eps^p}\bigg]^\frac{N+2}{N(p-2)+2(p+2)}
    \bigg(\sup_{Q_{R,S}} |Du|
    \bigg)^\frac{N(p-2)+2p}{N(p-2)+2(p+2)}\\
    &\qquad\qquad\qquad\qquad\qquad\qquad\cdot
    \bigg[\iint_{Q_{R,S}}
    |Du|^{p}\dx\dt
    \bigg]^\frac{2}{N(p-2)+2(p+2)}.
\end{align*}
Let us name the right-hand side as $\boldsymbol{K}$ and
choose $k$ as $\max\{ \boldsymbol K,\epsilon\mu\}$. Then, 
from $\boldsymbol Y_\infty=0$ we obtain that
\begin{align*}
    \sup_{Q_{\sigma R,\sigma S}} |Du|
    &\le\max\{ \boldsymbol K,\epsilon\mu\}.
\end{align*}
In combination with \eqref{choice-k:trivial-case}, which holds in the other case, this results in
\begin{align*}
    \sup_{Q_{\sigma R,\sigma S}} |Du|
    &\le\max\bigg\{ \boldsymbol K, \epsilon \Big( \frac{R^2}{S}\Big)^\frac{1}{p-2},\eps\mu\bigg\}.
\end{align*}

To set up an interpolation argument we first re-write the last display in a more compact form using the abbreviations
$$
    \boldsymbol Y_o=\iint_{Q_{R,S}}
    |Du|^{p}\dx\dt,
    \quad\mbox{and}\quad
    \boldsymbol{\nu} :=N(p-2)+2p.
$$  
Namely, we have
\begin{align*}
       \sup_{Q_{\sigma R,\sigma S}} |Du|
       \le\max\Bigg\{
            \bigg[\frac{C}{(1-\sigma)^2R^2 \epsilon^{p}}\bigg]^\frac{N+2}{\boldsymbol{\nu}+4}
     \bigg(\sup_{Q_{R,S}} |Du|
    \bigg)^\frac{\boldsymbol{\nu}}{\boldsymbol{\nu}+4} \boldsymbol Y_o^\frac{2}{\boldsymbol{\nu}+4},
    \epsilon \Big( \frac{R^2}{S}\Big)^\frac{1}{p-2},\epsilon\mu
    \Bigg\}.
\end{align*}
Next, we define a sequence of nested cylinders $\widehat Q_n=\widehat Q_{\rho_n,\theta_n}$ with $\rho_o=\sigma R$, with $\theta_o=\sigma S$, and 
for $n\in\N$ with
\begin{equation*}
    \rho_n:=\sigma R +(1-\sigma)R \sum_{j=1}^n2^{-j}
    \quad\mbox{and}\quad
    \theta_n:=\sigma S +(1-\sigma)S \sum_{j=1}^n2^{-j}.
\end{equation*}
Observe that $\widehat Q_o=Q_{\sigma R,\sigma S}$ and $\widehat Q_\infty = Q_{R,S}$. Moreover, we define
$$
    \boldsymbol M_n
    :=
    \sup_{\widehat Q_n} |Du|
$$
and apply the above gradient bound  on  two consecutive cylinders $\widehat Q_{n+1}\supset \widehat Q_n$ instead of
$Q_{R,S}\supset Q_{\sigma R, \sigma S}$.
Moreover, we use the facts $(\rho_{n+1}-\rho_n)^2=(1-\sigma)^2R^22^{-2(n+1)}$
and
\begin{equation*}
    \frac{\rho_{n+1}^2}{\theta_{n+1}}
    =
    \frac{R^2}{S}
    \Big[
    \sigma +(1-\sigma) \sum_{j=1}^n2^{-j}
    \Big]
    \le
    \frac{R^2}{S}.
\end{equation*}
This results in recursive inequalities of the type
\begin{align*}
    \boldsymbol M_n
    &\le
    \max\Bigg\{
            \bigg[\frac{C }{(1-\sigma)^2R^2\epsilon^{p}}\bigg]^\frac{N+2}{\boldsymbol{\nu}+4}
     2^{\frac{2(n+1)(N+2)}{\boldsymbol{\nu}+4}}\boldsymbol M_{n+1}^\frac{\boldsymbol{\nu}}{\boldsymbol{\nu}+4} \boldsymbol Y_o^\frac{2}{\boldsymbol{\nu}+4},
    \epsilon \Big( \frac{R^2}{S}\Big)^\frac{1}{p-2},\eps\mu
    \Bigg\}.
\end{align*}
If for some $n\in\N_0$ either the second or the third entry in the maximum dominates the first one, there is nothing to prove, since we trivially have
\begin{align*}
    \sup_{Q_{\sigma R,\sigma S}} |Du| 
    =
   \boldsymbol  M_o
    \le 
   \boldsymbol  M_{n}
    \le 
    \max\bigg\{ \epsilon
    \Big(\frac{R^2}{S}\Big)^{\frac{1}{p-2}}, \eps\mu\bigg\}.
\end{align*}
Otherwise, the first term in the maximum dominates for all $n\in\N_0$ the second and third entry, so that for any $n\in\N_0$ we have
\begin{align*}
    \boldsymbol M_n
    &\le
    \bigg[\frac{C }{(1-\sigma)^2R^2\eps^{p}}\bigg]^\frac{N+2}{\boldsymbol{\nu}+4}\boldsymbol Y_o^\frac{2}{\boldsymbol{\nu}+4}
     \big(\underbrace{4^{\frac{N+2}{\boldsymbol{\nu}+4}}}_{=:\boldsymbol b}\big)^{n}\boldsymbol M_{n+1}^{1-\frac{4}{\boldsymbol{\nu}+4}} .
\end{align*}
To the preceding recursive inequalities we apply the Interpolation Lemma \cite[Chapter I, Lemma 4.3]{DB} with $\alpha$, $C$ replaced by
\begin{equation*}
    \frac{4}{\boldsymbol{\nu}+4},
    \qquad
    \bigg[\frac{C \epsilon^{-p} }{(1-\sigma)^2R^2}\bigg]^\frac{N+2}{\boldsymbol{\nu}+4}\boldsymbol Y_o^\frac{2}{\boldsymbol{\nu}+4}.
\end{equation*}
The application yields
\begin{align*}
    \boldsymbol M_o
    &\le
    \Bigg[ 2   \bigg[\frac{C}{(1-\sigma)^2R^2 \eps^{p} }\bigg]^\frac{N+2}{\boldsymbol{\nu}+4}\boldsymbol Y_o^\frac{2}{\boldsymbol{\nu}+4}\boldsymbol b^{\frac{\boldsymbol{\nu}}{4}}      \Bigg]^\frac{\boldsymbol{\nu}+4}{4}\\
    &=
    \frac{C}{\eps^{\frac{p(N+2)}{4}}}\frac{1}{[(1-\sigma)R]^\frac{N+2}{2}}
    \bigg[ \iint_{Q_{R,S}}
    |Du|^{p}\dx\dt\bigg]^\frac12\\
    &=
    \frac{C}{\eps^{\frac{p(N+2)}{4}}}
    \frac{ \sqrt{S/R^2} }{(1-\sigma)^\frac{N+2}2}
    \bigg[\biint_{Q_{R,S}}
    |Du|^p\dx\dt
    \bigg]^\frac12.
\end{align*}
Taking into account the bound from the trivial case, we conclude that claimed gradient bound.
\end{proof}

\begin{remark}\upshape
With an interpolation argument as in the previous proof, the integrability exponent $p$ can be further reduced. In fact, for any $q\in(p-2,p)$, we have
\begin{align*}
    &\sup_{Q_{\sigma R,\sigma S}(z_o)}|Du|\\
    &\quad
    \le
    \max\Bigg\{
     \frac{C}{\eps^{\frac{p(N+2)}{2(q+2-p)}}}
    \bigg[
    \frac{ S/R^2 }{(1-\sigma)^{N+2}}
    \biint_{Q_{R,S}(z_o)}
    |Du|^{q}\dx\dt
    \bigg]^\frac{1}{q+2-p}, \epsilon\Big(\frac{R^2}{S}\Big)^\frac1{p-2}, \eps\mu\Bigg\}.
\end{align*}
\end{remark}

\begin{remark}
\upshape Exactly as already pointed out in Remark~\ref{Rmk:p=2},
when $p=2$, we can take $\eps=1$, $\mu=0$, $S=R^2$. Hence, the previous estimate reduces to 
\begin{equation*}
    \sup_{Q_{\sigma R,\sigma R^2}(z_o)} |Du|
    \le
    \Bigg[\frac{C}{(1-\sigma)^{N+2}}
    \biint_{Q_{R,R^2}(z_o)} |Du|^q\,\dx\dt
    \Bigg]^\frac{1}{q},
\end{equation*}
for any $q\in(0,2]$.
\end{remark}

\section{\texorpdfstring{$p$-}{p-}Laplace systems with measurable coefficients}\label{S:4}
This section mainly provides some comparison estimates that are necessary to prove the Schauder-type estimates. In contrast to the previous section, we consider systems of $p$-Laplace-type with coefficients that are merely measurable, and we include the case $\mu=0$. More precisely, we study weak solutions to the parabolic system~\eqref{p-laplace-intro} under the assumptions~\eqref{prop-a-intro}$_1$. The assumption~\eqref{prop-a-intro}$_2$, i.e.~the H\"older continuity with respect to the spatial variable, is only used in Lemma~\ref{lem:comp-2}, when we show a comparison estimate in preparation for the proof of Schauder estimates.

\subsection{Zero order energy estimates}
Weak solutions to the parabolic system \eqref{p-laplace-intro} satisfy a basic energy estimate,  which follows in a straightforward way by testing the system  against $(u-\xi)\z^p$,
 with $\xi\in\R^k$ and a suitable cut-off function $\z$, and modulo a Steklov averaging procedure. 

\begin{proposition}\label{prop:energy-est-zero-order}
Assume that $p>1$, $\mu\in [0,1]$, and that the coefficient $a$ satisfies
\eqref{prop-a-intro}$_1$. Then there exists a constant $C=C(p,C_o,C_1)$ such that
for any weak solution $u$ to the parabolic system \eqref{p-laplace-intro} in the sense of Definition \ref{def:weak-loc}, for any $\xi\in\R^k$, and any pair of cylinders $Q_{r,s}(z_o)
\subset Q_{R,S}(z_o)\subset E_T$ with  $0<s<S$,
$0<r<R$, we have the energy estimate
\begin{align*}
    \sup_{\tau\in (t_o-s,t_o]} \int_{B_r(x_o)\times\{\tau\}}&|u-\xi|^2 \,\dx
    +
    \iint_{Q_{r,s}(z_o)}\big(\mu^2+|Du|^2\big)^\frac{p}2 \,\dx\dt\\
    &\le
    C\iint_{Q_{R,S}(z_o)}\bigg[ \frac{|u-\xi|^p}{(R-r)^p}+\frac{|u-\xi|^2}{S-s}
    +\mu^p\bigg] 
    \,\dx\dt.
  \end{align*}
\end{proposition}
\begin{remark}\label{rem:energy-est-zero-order}\upshape
For cylinders $Q_{r}^{(\lambda)}(z_o)\subset Q_{R}^{(\lambda)}(z_o)\subset E_T$ with $\tfrac12 R\le r<R$, after passing to mean values the energy inequality of Proposition~\ref{prop:energy-est-zero-order}
takes the form
\begin{align*}
    \frac{\lambda^{p-2}}{R^2}\sup_{\tau\in (t_o-\lambda^{2-p}r^2,t_o]}
    &
      \mint_{B_{r}(x_o)\times\{\tau\}}|u-\xi|^2\,\dx+
      \biint_{Q_{r}^{(\lambda)}(z_o)}
      \big(\mu^2+|Du|^2\big)^{\frac{p}2}\,\dx\dt\\ 
    &\le
      C
      \biint_{Q_{R}^{(\lambda)}(z_o)}\bigg[
      \frac{|u-\xi|^p}{(R-r)^p} +\lambda^{p-2}\frac{|u-\xi|^2}{(R-r)^2} +\mu^p\bigg]
 \,\dx\dt,
  \end{align*}
where now the constant $C=C(N,p,C_o,C_1)$ additionally depends on the dimension $N$.
\end{remark}

\subsection{Comparison principle}
Needless to say, a comparison principle does not hold for systems in general. However, it does hold when the system has certain special structures.
\begin{lemma}\label{lem:comp-0}
Let $p>1$ and $\mu\in [0,1]$. Let $u$
be a weak solution to
\begin{equation*}
         \partial_t v-\Div\Big( a(x_o,t)\big(\mu^2+|Dv|^2\big)^\frac{p-2}2Dv\Big)=0
    \quad\mbox{in $Q_{R,S}(z_o)$}
\end{equation*}
 in the sense of Definition \ref{def:weak-loc}, and $\xi,\zeta\in\R^k$. Suppose that $v^i\le \xi^i$ on $\partial_\mathrm{par}Q_{R,S}(z_o)$
for any $i\in\{1,\dots,k\}$. Then $v^i\le \xi^i$ a.e. in $Q_{R,S}(z_o)$. Moreover, $v^i\ge \zeta^i$ on $\partial_\mathrm{par}Q_{R,S}(z_o)$
for any $i\in\{1,\dots,k\}$ implies $v^i\ge \zeta^i$ a.e. in $Q_{R,S}(z_o)$.
\end{lemma}
On the lateral boundary $\partial B_R(x_o)\times (t_o-S,t_o)$  the condition $v^i\le \xi^i$ has to be understood in the sense that $ (v^i -\xi^i)_+\in L^p\big(t_o-S,t_o;W^{1,p}_0(B_R(x_o))\big) $, while at the initial time $t_o-S$ it is to be understood in terms of the pointwise condition $ v^i(\cdot, t_o-S)\le \xi^i$ a.e.~on $B_R(x_o)$.

\begin{proof}
We only consider the case where $v^i\le\xi^i$ on $\partial_\mathrm{par}Q_{R,S}(z_o)$. The other one is similar.
Fix  a time $\tau\in (t_o-S,t_o)$ and choose in the weak formulation of the parabolic system for $v$  the testing function
\begin{equation*}
    \varphi(x,t)=\mathbf 1_{\{t\le\tau\}}(v(x,t)- \xi)_+,
\end{equation*}
where
\begin{equation*}
    (v- \xi)_+
    =
    \big( (v^1-\xi^1)_+,\dots ,(v^k-\xi^k)_+\big).
\end{equation*}
 This leads to
$$
    \iint_{B_R(x_o)\times (t_o-S,\tau)}  
    \partial_t(v-\xi)\cdot (v- \xi)_+\,\dx\dt=-\mathbf{II},
$$
where the right-hand side is given by
\begin{align*}
    \mathbf{II}
    &:=
     \iint_{B_R(x_o)\times (t_o-S,\tau)} a(x_o,t) \mathsf A\,\dx\dt,
\end{align*}
with $\mathsf A$ defined by
\begin{align*}
    \mathsf A
     &:=
    \big(\mu^2+|Dv|^2\big)^\frac{p-2}2 Dv
    \cdot 
   D(v- \xi)_+.
\end{align*}
Observe that $\mathsf A\ge 0$,
and therefore the integral $\mathbf{II}\ge 0$ is non-negative. This allows us to conclude that
\begin{align*}
	\tfrac12 \int_{B_R(x_o)\times\{\tau\}} &|(v- \xi )_+|^2 \,\dx 
    -
    \underbrace{
    \tfrac12 \int_{B_R(x_o)\times\{t_o-S\}} |(v- \xi )_+|^2 \,\dx}_{=0}\\
    &=
    \tfrac12 \iint_{B_R(x_o)\times (t_o-S,\tau)} \partial_t|(v- \xi)_+|^2 \,\dx\dt \\
	&=
	\iint_{B_R(x_o)\times (t_o-\tau,\tau)} \partial_t(v- \xi)\cdot (v- \xi)_+ \,\dx\dt \\
	&=
	-\mathbf{II} 
    \le 
    0. 
\end{align*}
Here we used that $(v-\xi)_+(\cdot,  t_o-S)=0$ in $B_R(x_o)$.
Hence, $v^i(\cdot,\tau)\le  \xi^i$ a.e.~in $B_R(x_o)$ for any $i\in\{1,\dots,k\}$. Since $\tau\in (t_o-S,t_o)$ was arbitrary, this shows $v^i\le  \xi^i$ a.e.~in $Q_{R,S}(z_o)$ for any $i\in\{1,\dots,k\}$. This proves the claim.
\end{proof}

\begin{lemma}\label{lem:comp-1}
Let $p>1$, $\mu\in [0,1]$, and $u$ a bounded weak solution to the parabolic $p$-Laplace-type system \eqref{p-laplace-intro} in $E_T$. Let $Q_{R,S}(z_o)\subset E_T$ and $v$
be the unique weak solution to the Cauchy-Dirichlet problem
\begin{equation*}
\left\{
\begin{array}{cl}
         \partial_t v-\Div\Big( a(x_o,t)\big(\mu^2+|Dv|^2\big)^\frac{p-2}2Dv\Big)=0,&
    \mbox{in $Q_{R,S}(z_o)$}\\[7pt]
    v=u,&\mbox{on $\partial_{\mathrm{par}}Q_{R,S}(z_o)$}
\end{array}
\right.
\end{equation*}
in the sense of Definition \ref{def:weak}.
Then,
\begin{equation*}
    \osc_{Q_{R,S}(z_o)} v\le \sqrt{k}\osc_{Q_{R,S}(z_o)}u,
\end{equation*}
where the oscillation has to be understood with respect to the euclidean norm.
\end{lemma}

\begin{proof}
Abbreviate $Q:=Q_{R,S}(z_o)$ and 
apply the previous comparison Lemma with $\xi^i:=\sup_{Q}u^i$ respectively $\zeta^i:=\inf_{Q}u^i$. We
deduce
\begin{align*}
    \sup_Q v^i\le\sup_{Q}u^i
\qquad\mbox{and}\qquad
    \inf_Q v^i\ge \inf_{Q}u^i
\qquad\mbox{for }i=1,\ldots,k,
\end{align*}
so that for a.e. $(x,t),(\tilde x,\tilde t)\in Q$ we have
\begin{align*}
   -\osc_Q u^i=\inf_{Q}u^i -\sup_{Q}u^i\le  v^i(x,t)-v^i(\tilde x,\tilde t)\le \sup_{Q}u^i-\inf_{Q}u^i=\osc_Q u^i.
\end{align*}
Therefore, we have
\begin{align*}
    |v (x,t)-v(\tilde x,\tilde t)|^2
    &=
    \sum_{i=1}^k |v^i (x,t)-v^i(\tilde x,\tilde t)|^2\le
     \sum_{i=1}^k \big(\osc_Q u^i\big)^2\le k \big( \osc_Qu\big)^2
\end{align*}
for a.e. $(x,t), (\tilde x,\tilde t)\in Q$. This implies the claim.
\end{proof}
\begin{remark}
\upshape The previous result has been proved for cylinders whose cross-sections are balls, but any kind of set besides balls can do here. 
\end{remark}
\subsection{Comparison estimates}

Let $p>1$, $\mu\in (0,1]$ and let
\begin{equation*}
    v\in L^p\big(0,T;W^{1,p}(E,\R^k)\big)\cap C\big([0,T];L^2(E,\R^k)\big)
\end{equation*}
be a weak solution to the parabolic $p$-Laplace-type system \eqref{p-laplace-intro},
where the coefficients $a(x,t)$ satisfy the assumptions \eqref{prop-a-intro}$_1$. 
Let
\begin{equation*}
    w\in L^p\big(0,T;W^{1,p}(E,\R^k)\big)\cap C\big([0,T];L^2(E,\R^k)\big)
\end{equation*}
solve
\begin{equation}\label{EqB*}
    \left\{
    \begin{array}{cl}
        \partial_tw-\Div
         \Big( b(x,t) \big(\mu^2+|Dw|^2\big)^\frac{p-2}2 Dw\Big)=0,
         &\mbox{in $E_T$,}\\[6pt]
         w=v, &\mbox{on $\partial_{\mathrm{par}} E_T$}.
    \end{array}
    \right.
\end{equation}
The coefficients $b\in L^\infty (E_T)$ are assumed to satisfy
\begin{equation}\label{ass:a-b}
    b(x,t)\ge C_o>0
\end{equation}
for a.e.~$(x,t)\in E_T$. The following comparison estimate is the main result of this section.

\begin{lemma}\label{lem:comp-2-mu>0}
Let $v$ be a weak solution to the parabolic system \eqref{p-laplace-intro} in the sense of Definition \ref{def:weak-loc} and $w$ be a weak solution to the Cauchy-Dirichlet problem \eqref{EqB*} with \eqref{ass:a-b} in the sense of Definition \ref{def:weak}.
Then, we have the comparison estimate
\begin{align}\label{comparison:p<2}
  \iint_{E_T}&|Dv-Dw|^p\,\dx\dt
  \le
  C \big[ \mathsf A^p + \mathsf A^\frac{p}{p-1} \big]
    \iint_{E_T}\big(\mu^2+|Dv|^2\big)^{\frac{p}2}\,\dx\dt,
\end{align}
where we used the abbreviation $\mathsf A:=\| a-b\|_{L^\infty (E_T)}$. In the super-quadratic case
$p\ge 2$ we have the stronger comparison estimate
\begin{align}\label{comparison:p>2}
  \iint_{E_T}&|Dv-Dw|^p\,\dx\dt
  \le
  C \mathsf A^\frac{p}{p-1} 
    \iint_{E_T}\big(\mu^2+|Dv|^2\big)^{\frac{p}2}\,\dx\dt.
\end{align}
Furthermore, for any $p>1$ we have the energy inequality
\begin{align}\label{EAbsch1}
    \iint_{E_T}&\big(\mu^2+|Dw|^2\big)^{\frac{p}{2}}\,\dx\dt 
    \le C
    \big[ 1+ \mathsf A^p + \mathsf A^\frac{p}{p-1} \big]
      \iint_{E_T}\big(\mu^2+|Dv|^2\big)^{\frac{p}2}\,\dx\dt.
\end{align}
In all estimates, the constant $C $ depends only on $p$  and $C_o$. 
\end{lemma}

\begin{proof}
Writing down the Steklov formulation of the parabolic systems for $v$ and $w$ on $Q$
and subtracting from one another, we obtain
\begin{align*}
    0=& -\iint_{E_T}
    [v-w]_h\cdot \partial_t\zeta\,\dx\dt \\
    &+
    \iint_{E_T}
    \Big[a\big(\mu^2+|Dv|^2\big)^\frac{p-2}2Dv
    -
    b\big(\mu^2+|Dw|^2\big)^\frac{p-2}2Dw
    \Big]_h\cdot D\zeta\,\dx\dt
\end{align*}
for any 
$$
    \zeta\in  W^{1,2}\big(0,T; L^{2}(E,\R^k)\big)
    \cap L^p \big(0,T; W^{1,p}_{0}(E,\R^k)\big)
$$
such that $\zeta (\cdot,0)=0=\zeta(\cdot, T)$. We now choose $\zeta=[v-w]_h\psi_\epsilon$. Here the Lipschitz function $\psi_\epsilon\colon  [0,T]\to\R$
is chosen to be equal to $1$ in $[t_1,t_2]\subset(0,T)$, vanishes outside 
$[t_1-\epsilon, t_2+\epsilon]$ with $0<\epsilon<\min\big\{t_1, T-t_2\big\}$, and is otherwise linearly interpolated. In this context, $[v-w]_h$ is the Steklov average of $v-w$ and $0<h<T-(t_2+\eps)$. In the above Steklov formulation 
there are two terms to be treated.  We start with the term containing the time derivative. This term is calculated as follows
\begin{align*}
  -\iint_{E_T}&
  [v-w]_h\cdot \partial_t\big([
  v-w]_{h}\psi_\epsilon\big)\,\dx\dt \\
  &=
  -\tfrac12\iint_{E_T}\partial_t\big|[v-w]_{h}\big|^2\psi_{\epsilon}\,\dx\dt -
  \iint_{E_T}\big|[v-w]_{h}\big|^2 \partial _t\psi_\epsilon \,\dx\dt\\
  &=
  -\tfrac12 \iint_{E_T}\big|[v-w]_{h}\big|^2 \partial _t\psi_\epsilon\,\dx\dt.
\end{align*}
Using Lemma~\ref{lm:Stek} (i) we obtain 
$$ 
-\lim_{h\downarrow 0}
\tfrac12 \iint_{E_T}\big|[v-w]_{h}\big|^2 \partial _t\psi_\epsilon\,\dx\dt
=
-\tfrac12 \iint_{E_T}|v-w|^2 \partial _t\psi_\epsilon\,\dx\dt.
$$
Here we pass to the limit $\epsilon\downarrow0$ and obtain 
\begin{align*}
    -\lim_{\eps\downarrow 0}
    \lim_{h\downarrow 0}&
    \tfrac12 \iint_{E_T}\big|[v-w]_{h}\big|^2 \partial _t\psi_\epsilon\,\dx\dt\\
    &=
    \tfrac12\int_{E\times\{t_2\}} |v-w|^2\,\dx -\tfrac12\int_{E\times\{t_1\}} |v-w|^2\,\dx.
\end{align*}
The first integral yields a non-negative contribution and can be discarded, while the second integral converges to $0$ at $t_1\downarrow0$ due to the initial condition in the Cauchy-Dirichlet problem \eqref{EqB*}. The remaining term in the above Steklov
formulation converges to
\[
    \iint_{E\times(t_1,t_2)}
    \Big[a\big(\mu^2+|Dv|^2\big)^\frac{p-2}2Dv
    -
    b\big(\mu^2+|Dw|^2\big)^\frac{p-2}2Dw
    \Big]\cdot D(v-w)
    \,\dx\dt
\]
if we first let $h\downarrow0 $ and then $\epsilon\downarrow0$. The limit $h\downarrow 0$ can be justified with Lemma~\ref{lm:Stek} (ii).  Consequently, after letting $t_1\downarrow0$ and $t_2\uparrow T$, we have
\begin{align*}
  \nonumber\iint_{E_T}
  &
  b\Big[\big(\mu^2+|Dv|^2\big)^\frac{p-2}2Dv
  -
  \big(\mu^2+|Dw|^2\big)^\frac{p-2}2Dw\Big]\cdot (Dv-Dw)\,\dx\dt\\
  &\le
    \iint_{E_T} 
  (b-a)\big(\mu^2+|Dv|^2\big)^\frac{p-2}2Dv\cdot
 (Dv-Dw)\,\dx\dt.
\end{align*}
Using the monotonicity of the vector-field $\xi\mapsto \big(\mu^2+|\xi|^2\big)^\frac{p-2}2\xi$ stated in Lemma~\ref{AFGM}, and the bound from below for the coefficients $b$ from \eqref{ass:a-b} (which is $C_o$), the
left-hand side can be bounded from below, while the right-hand side is estimated from above using 
H\"older's inequality. This way we get
\begin{align}\label{pre-comparison}\nonumber
    \tfrac{C_o}{C(p)}\mathbf I
    &:=
    \tfrac{C_o}{C(p)}
    \iint_{E_T}
    \big(\mu^2+|Dv|^2+|Dw|^2
    \big)^{\frac{p-2}{2}}|Dv-Dw|^2\,\dx\dt\\
    &\,\le
    \mathsf A
    \bigg[
    \underbrace{\iint_{E_T}\big(\mu^2+|Dv|^2\big)^{\frac{p}{2}}\dx\dt}_{:=\mathsf E}\bigg]^{1-\frac1{p}}
    \bigg[\iint_{E_T}|Dv-Dw|^p\,\dx\dt\bigg]^{\frac1p}.
\end{align}
In the case $p\ge2$ this immediately implies \eqref{comparison:p>2}. Indeed, we have
\begin{align*}
    \iint_{E_T}
    |Dv&-Dw|^p\,\dx\dt
    \le 
    2^\frac{p-2}2\mathbf I\le
    \tfrac{C(p)}{C_o}
    \mathsf A
    \mathsf E^{1-\frac1{p}}
    \bigg[\iint_{Q}|Dv-Dw|^p\,\dx\dt\bigg]^{\frac1p},
\end{align*}
which yields
\begin{align*}
    \iint_{E_T}|Dv-Dw|^p\,\dx\dt
  &\le
  C 
  \mathsf A^\frac{p}{p-1}
 \mathsf E
\end{align*}
for a constant $C =C (p,C_o)$. In the case $1<p<2$, however, we apply first H\"older's inequality with exponents $\frac2p$ and $\frac2{2-p}$ and then \eqref{pre-comparison} to obtain 
\begin{align*}
  \iint_{E_T}&|Dv-Dw|^p\, \dx\dt\\
  &\le
    \mathbf I^{\frac{p}2}
    \bigg[\iint_{E_T}\big(\mu^2+|Dv|^2+|Dw-Dv+Dv|^2\big)^{\frac{p}2}\, \dx\dt\bigg]^{\frac{2-p}{2}}\\
  &\le
  C 
  \mathsf A^\frac{p}{2}
   \mathsf E^{\frac{p-1}2}
    \bigg[\iint_{E_T}|Dv-Dw|^p\,\dx\dt\bigg]^{\frac12}
    \bigg[\iint_{E_T}|Dv-Dw|^p\, \dx\dt
    +\mathsf E\bigg]^{\frac{2-p}{2}},
\end{align*}
which is the same as
\begin{align*}
    \iint_{E_T}|Dv&-Dw|^p\, \dx\dt
    \le
    C 
   \mathsf A^p
   \mathsf E^{p-1}
    \bigg[\iint_{E_T}|Dv-Dw|^p\, \dx\dt
    +\mathsf E\bigg]^{2-p}.
\end{align*}
Applying Young's inequality with exponents $\frac1{p-1}$ and $\frac1{2-p}$ we get
\begin{align*}
   \iint_{E_T}&|Dv-Dw|^p\, \dx\dt
   \le
    \tfrac12 \iint_{E_T}|Dv-Dw|^p\, \dx\dt+
    C \big[ 
    \mathsf A^p
    +
    \mathsf A^\frac{p}{p-1}
    \big]
    \mathsf E.
\end{align*}
After re-absorbing the integral of $|Dv-Dw|^p$ to the left-hand side, we
obtain the asserted comparison estimate~\eqref{comparison:p<2}
also in the sub-quadratic case. The 
energy estimate~\eqref{EAbsch1} follows from
\begin{align*}
  \iint_{E_T}&\big(\mu^2+|Dw|^2\big)^\frac{p}2\,\dx\dt\\
  &\le
  C(p)\iint_{E_T}\big(\mu^2+|Dv|^2\big)^\frac{p}2\,\dx\dt
  +
  C(p)\iint_{E_T}|Dv-Dw|^p\,\dx\dt\\
  &\le
  C(p,C_o)
  \big[ 1+\mathsf A^p +\mathsf A^\frac{p}{p-1}\big]
  \iint_{E_T}\big(\mu^2+|Dv|^2\big)^\frac{p}2\,\dx\dt.
\end{align*}
This completes
the proof of the lemma.  
\end{proof}

Here, we record an application of the last lemma; it will be useful when dealing with the $p$-Laplace system \eqref{p-laplace-intro} that has H\"older continuous coefficients as in \eqref{prop-a-intro}. The comparison function $w$ is then taken as the unique weak solution to the Cauchy-Dirichlet problem 
on some cylinder $Q_{R,S}(z_o)$ with frozen coefficients $a(x_o,t)$.
That is, we denote by 
\begin{equation*}
    w\in L^p\big(t_o-S,t_o;W^{1,p}(B_R(x_o),\R^k)\big)
    \cap 
    C\big([t_o-S,t_o];L^2(B_R(x_o),\R^k)\big)
\end{equation*}
the unique  weak solution to the Cauchy-Dirichlet problem
\begin{equation}\label{EqB}
    \left\{
    \begin{array}{cl}
        \partial_tw-\Div
         \Big( a(x_o,t)\big(\mu^2+|Dw|^2\big)^\frac{p-2}2 Dw\Big)=0,
         &\mbox{in $Q_{R,S}(z_o)$,}\\[6pt]
         w=u, &\mbox{on $\partial_{\mathrm{par}} Q_{R,S}(z_o)$.}
    \end{array}
    \right.
\end{equation}
In this situation we have the following result.
\begin{lemma}\label{lem:comp-2}
Let $u$ be a weak solution to the parabolic system \eqref{p-laplace-intro} in the sense of Definition \ref{def:weak-loc} with assumption \eqref{prop-a-intro}
and $w$ be the unique weak solution to the Cauchy-Dirichlet problem \eqref{EqB} on some cylinder $Q_{R,S}(z_o)\subset E_T$ with $R\le 1$ in the sense of Definition \ref{def:weak}.
Then, we have the comparison estimate
\begin{align*}
  \iint_{Q_{R,S}(z_o)}&|Du-Dw|^p\,\dx\dt
  \le
  C R^{\alpha_\ast p}
    \iint_{Q_{R,S}(z_o)}\big(\mu^2+|Du|^2\big)^{\frac{p}2}\,\dx\dt.
\end{align*}
Here, the exponent $\alpha_\ast$ is given by
\begin{equation}\label{alpha_ast}
    \alpha_\ast:=
    \left\{
    \begin{array}{cl}
    \alpha,&\mbox{if $1<p<2$,}\\[5pt]
    \tfrac{\alpha}{p-1},&\mbox{if $p\ge 2$.}
    \end{array}
    \right.
\end{equation}
Furthermore,  we have the energy inequality
\begin{align*}
    \iint_{Q_{R,S}(z_o)}&\big(\mu^2+|Dw|^2\big)^{\frac{p}{2}}\,\dx\dt 
    \le C
      \iint_{Q_{R,S}(z_o)}\big(\mu^2+|Du|^2\big)^{\frac{p}2}\,\dx\dt.
\end{align*}
In all estimates, the constant $C $ depends only on $p$, $C_o$ and $C_1$.
\end{lemma}

\begin{proof} We apply Lemma \ref{lem:comp-2-mu>0} with $b(x,t):=a(x_o,t)$.
Here, we only need to determine the pre-factor in the estimates in this specific situation. Since
\begin{align*}
    \|a -b\|_{L^\infty({Q_{R,S}(z_o)})}
    \le C_1R^\alpha,
\end{align*}
we have
\begin{align*}
    \|a -b\|_{L^\infty({Q_{R,S}(z_o)})}^p +\|a -b\|_{L^\infty({Q_{R,S}(z_o)})}^\frac{p}{p-1}
    \le
    C_1^p R^{\alpha p}+ C_1^\frac{p}{p-1}R^\frac{\alpha p}{p-1}
    \le 
    CR^{\alpha_\ast p}\le C.
\end{align*}
The claimed inequalities immediately follow.
\end{proof}

At a later stage, we will need a version of the comparison estimate when $v$ is a solution to the degenerate parabolic $p$-Laplace system \eqref{p-laplace-intro} with $\mu=0$ and the comparison mapping is a solution to the non-degenerate $p$-Laplace system \eqref{EqB*} with $\mu\in (0,1]$. 
More precisely let
\begin{equation*}
    v\in L^p\big(0,T;W^{1,p}(E,\R^k)\big)\cap 
    C\big([0,T];L^2(E,\R^k)\big)
\end{equation*}
be a weak solution to the parabolic $p$-Laplace-type system
\begin{equation}\label{EqA:mu=0}
    \partial_tv-\Div \big(a(x,t) |Dv|^{p-2} Dv\big)=0\quad  \mbox{in $E_T$,}
\end{equation}
where the coefficient $a(x,t)$ satisfies the assumption \eqref{prop-a-intro}$_1$, and let
\begin{equation*}
    w\in L^p\big(0,T;W^{1,p}(E,\R^k)\big)\cap C\big([0,T];L^2(E,\R^k)\big)
\end{equation*}
be the unique solution to the Cauchy-Dirichlet problem \eqref{EqB*} in $E_T$ with $\mu\in (0,1]$ and with coefficients $b$ satisfying \eqref{ass:a-b}.

\begin{lemma}\label{lem:comp-2-mu=0}
Let $v$ be a weak solution to the parabolic system \eqref{EqA:mu=0} in the sense of Definition \ref{def:weak-loc} with coefficient $a$ satisfying \eqref{prop-a-intro}$_1$,
and let $w$ be the unique weak solution to the Cauchy-Dirichlet problem \eqref{EqB*} with  \eqref{ass:a-b} in the sense of Definition \ref{def:weak}.
Then, we have the comparison estimate
\begin{align*}
  \iint_{E_T}|Dv-Dw|^p\,\dx\dt
  \le
  C \big[ (\mathsf A +\mu)^p + (\mathsf A+\mu)^\frac{p}{p-1} \big]
    \iint_{E_T}\big(1+|Dv|^2\big)^{\frac{p}2}\,\dx\dt,
\end{align*}
where we used the abbreviation $\mathsf A:=\| a-b\|_{L^\infty (E_T)}$. In the super-quadratic case
$p\ge 2$ the inequality holds without the term $(\mathsf A+\mu)^p$, i.e.~we have
\begin{align*}
  \iint_{E_T}|Dv-Dw|^p\,\dx\dt
  \le
  C (\mathsf A+\mu)^\frac{p}{p-1} 
    \iint_{E_T}\big(1+|Dv|^2\big)^{\frac{p}2}\,\dx\dt.
\end{align*}
Furthermore, for any $p>1$ we have the energy inequality
\begin{align*}
    \iint_{E_T} & \big(\mu^2+|Dw|^2\big)^{\frac{p}{2}}\,\dx\dt \\
    &\le 
    C \Big[ 1+ (\mathsf A+\mu)^p + (\mathsf A+\mu)^\frac{p}{p-1} \Big]
      \iint_{E_T}\big(1+|Dv|^2\big)^{\frac{p}2}\,\dx\dt.
\end{align*}
In all estimates, the constant $C$ depends  on $p,C_o$  and $C_1$. 
\end{lemma}

\begin{proof} 
The reasoning proceeds largely along the lines of the proof of the comparison Lemma~\ref{lem:comp-2-mu>0}. In principle, we only have to replace $\big(\mu^2+|Dv|^2\big)^\frac{p-2}2Dv$ by $|Dv|^{p-2}Dv$ in the argument. This leads to 
\begin{align*}
  \iint_{E_T}
  &
  b\Big[\big(\mu^2+|Dv|^2\big)^\frac{p-2}2Dv
  -
  \big(\mu^2+|Dw|^2\big)^\frac{p-2}2Dw\Big]\cdot (Dv-Dw)\,\dx\dt\\
  &\le
\iint_{E_T}
  \Big[ 
  b\big(\mu^2+|Dv|^2\big)^\frac{p-2}2Dv- a|Dv|^{p-2}Dv\Big]\cdot
 (Dv-Dw)\,\dx\dt\\
 &=
 \iint_{E_T}
 (b-a) \big(\mu^2+|Dv|^2\big)^\frac{p-2}2Dv\cdot
 (Dv-Dw)\,\dx\dt+\mathbf{II},
\end{align*}
where we abbreviated
\begin{align*}
 \mathbf{II}
 &:=\iint_{E_T} a
 \Big[
 \big(\mu^2+|Dv|^2\big)^\frac{p-2}2 -|Dv|^{p-2}
 \Big]Dv \cdot
 (Dv-Dw)\,\dx\dt.
\end{align*}
Taking into account that
\begin{align*}
	\Big| \big(\mu^2+|Dv|^2\big)^\frac{p-2}2 -|Dv|^{p-2}\Big|
    &\le
	C(p) \mu \big(\mu^2 +|Dv|^2\big)^\frac{p-3}2,
\end{align*}
and using H\"older's inequality we obtain
\begin{align*}
    |\mathbf{II}|
    &\le 
    C(p)C_1\mu \iint_{E_T} \big(\mu^2 +|Dv|^2\big)^\frac{p-3}{2}|Dv||Dv-Dw|\, \dx\dt\\
    &\le 
    C(p)C_1\mu \iint_{E_T} \big(\mu^2 +|Dv|^2\big)^\frac{p-2}{2}|Dv-Dw|\, \dx\dt\\
    &\le
    C(p)C_1\mu \bigg[ 
     \iint_{E_T} \big(\mu^2 +|Dv|^2\big)^{\frac{p-2}{2}\frac{p}{p-1}}\, \dx\dt
    \bigg]^{1-\frac1p}\bigg[\iint_{E_T} |Dv-Dw|^p\, \dx\dt\bigg]^\frac1p.
\end{align*}
In the case $p\ge 2$ we have $0\le \frac{p-2}{p-1}\le 1$, so that 
\begin{align*}
    |\mathbf{II}|
    &\le
    C(p)C_1\mu \bigg[ 
     \iint_{E_T} \big(1 +|Dv|^2\big)^{\frac{p}{2}}\, \dx\dt
    \bigg]^{1-\frac{1}{p}}\bigg[\iint_{E_T} |Dv-Dw|^p\, \dx\dt\bigg]^\frac1p ,
\end{align*}
while in the case $1<p<2$ we have $\big(\mu^2 +|Dv|^2\big)^{\frac{p-2}{2}}\le\mu^{p-2}$, so that
\begin{align*}
    |\mathbf{II}|
    &\le
    C(p)C_1\mu^{p-1} |E_T|^{1-\frac1p}
    \bigg[\iint_{E_T} |Dv-Dw|^p\, \dx\dt\bigg]^\frac1p \\
    &\le
    C(p)C_1\mu^{p-1} \bigg[ 
     \iint_{E_T} \big(1 +|Dv|^2\big)^{\frac{p}{2}}\, \dx\dt
    \bigg]^{1-\frac{1}{p}}\bigg[\iint_{E_T} |Dv-Dw|^p\, \dx\dt\bigg]^\frac1p.
\end{align*}
Together with the estimates for the left-hand side and the first term on the right-hand side from the proof of Lemma \ref{lem:comp-2-mu>0}, we get in the super-quadratic case $p\ge 2$ 
that
\begin{align*}
    \iint_{E_T}
    |Dv&-Dw|^p \,\dx\dt
    \le 
    C
    \big(\mathsf A +\mu\big)^\frac{p}{p-1}
    \iint_{E_T}\big(1 +|Dv|^2\big)^\frac{p}{2}\, \dx\dt
\end{align*}
for a constant $C=C(p,C_o,C_1)$, while for exponents $1<p<2$, the arguments from the proof of Lemma \ref{lem:comp-2-mu>0} and the estimate for $\mathbf{II}$ yield
\begin{align*}
    \iint_{E_T}
    |Dv&-Dw|^p \,\dx\dt
    \le 
    C
    \Big[\big(\mathsf A +\mu\big)^\frac{p}{p-1} +(\mathsf A+\mu)^p\Big]
    \iint_{E_T}\big(1 +|Dv|^2\big)^\frac{p}{2}\, \dx\dt,
\end{align*}
again with a constant $C=C(p,C_o,C_1)$.
The energy inequality is obtained by combining
\begin{align*}
  \iint_{E_T}&\big(\mu^2+|Dw|^2\big)^\frac{p}2 \,\dx\dt\\
  &\le
  C(p)\iint_{E_T}\big(\mu^2+|Dv|^2\big)^\frac{p}2 \,\dx\dt
  +
  C(p)\iint_{E_T}|Dv-Dw|^p\,\dx\dt
\end{align*}
with  the comparison estimate.
\end{proof}

\subsection{Gluing Lemma and oscillation estimates for Lipschitz solutions}
The first lemma of this section serves to compare slice-wise mean values at different times. Such a result is often called {\em Gluing Lemma}. In the formulation of the lemma, $\eta$ stands for a non-negative function in $C^\infty_0(B_1)$ with $\|\eta\|_{L^1(B_1)} =1$, whereas $\eta_R(x):=R^{-N}\eta\big(\frac{x}{R}\big)$ is equal to $1$ on $B_R$ and retains a unit integral.
\begin{lemma}\label{lm:gluing}
Let $p>1$, $\mu\in [0,1]$, and $u$ be a weak solution to the parabolic  system~\eqref{p-laplace-intro} with \eqref{prop-a-intro}$_1$ in the sense of Definition \ref{def:weak-loc}. Assume that $|Du|\in L^\infty_{\rm loc}(E_T)$. Then, for every cylinder $Q=B_R(x_o)\times (\tau_1,\tau_2]\Subset E_T$ and any $\tau_1<t_1<t_2<\tau_2$ we have
\begin{equation*}
    \bigg|
    \int_{B_R(x_o)}\big[ u(\cdot,t_2) - u(\cdot,t_1)\big]\eta_R(\cdot-x_o)\,\dx\bigg|
    \le
    C\frac{t_2-t_1}{R}\big( \mu^2 +\|Du\|_{L^\infty(Q)}^2\big)^\frac{p-1}2
\end{equation*}
for a constant $C=C(N,C_1)$. 
\end{lemma}

\begin{proof} 
In the proof we assume $x_o=0$.
We define  $\xi_\epsilon=\xi_\epsilon(t)$ to be $1$ in $[t_1,t_2]$,
$0$ outside $[t_1-\epsilon,t_2+\epsilon]$, and linearly interpolated otherwise; here, we assume $\epsilon\in \big(0,\min\{  t_1-\tau_1, \tau_2-t_2\}\big)$.
Then we test the weak form of \eqref{p-laplace-intro} for a fixed index $i\in\{ 1,\dots ,k\}$ with
$$
    \varphi(x,t):=\xi_\epsilon (t)\eta_R(x)\mathrm e_i,
$$
where $\mathrm e_i$ denotes the $i$-th canonical basis vector in $\R^k$. In the limit $\epsilon\downarrow 0$ we obtain
\begin{align*}
   \int_{B_R}& \big[ u(\cdot,t_2) -u(\cdot,t_1)]\cdot \mathrm e_i
   \eta_R\,\dx\\
   &=
   \iint_{B_R\times (t_1,t_2)}a\, \big(\mu^2 +|Du|^2\big)^\frac{p-2}2Du
   \cdot(\mathrm e_i\otimes \nabla \eta_R)\,\dx\dt.
\end{align*}
We multiply the preceding identity by $\mathrm e_i$ and sum over $i=1,\dots ,k$. This yields
\begin{align*}
   \int_{B_R}& \big[ u(\cdot,t_2) -u(\cdot,t_1)]
   \eta_R\,\dx
   =
   \iint_{B_R\times (t_1,t_2)}a\,\big(\mu^2 +|Du|^2\big)^\frac{p-2}2Du\cdot
   \nabla \eta_R\,\dx\dt.
\end{align*}
Recalling the upper bound for the coefficients $a$ from \eqref{prop-a-intro}$_1$, we obtain
\begin{align*}
   \bigg|\int_{B_R} \big[ u(\cdot,t_2) -u(\cdot,t_1)]
   \eta_R\,\dx\bigg|
   &\le
   \frac{C(N)C_1}{R^{N+1}}\iint_{B_R\times (t_1,t_2)}\big(\mu^2 +|Du|^2\big)^\frac{p-1}2\,\dx\dt\\
   &\le 
   C\frac{t_2-t_1}{R}\big( \mu^2 +\|Du\|_{L^\infty(Q)}^2\big)^\frac{p-1}2,
\end{align*}
proving the asserted gluing estimate.
\end{proof}

Next, we collect two \textit{a priori} oscillation estimates of weak solutions whose spatial gradient is assumed to be locally bounded.
\begin{lemma}\label{lem:osc-all-p}
Let $p>1$, $\mu\in [0,1]$, and $u$ be a weak solution to the parabolic $p$-Laplace system~\eqref{p-laplace-intro} with \eqref{prop-a-intro}$_1$ in the sense of Definition \ref{def:weak-loc}. Assume that $|Du|\in L^\infty_{\rm loc}(E_T)$. Then, for every cylinder $Q=B_R(x_o)\times (\tau_1,\tau_2]\Subset E_T$  we have
\begin{equation*}
    \osc_{Q} u
    \le
    4R\| Du\|_{L^\infty (Q)}
    +
     C\frac{\tau_2-\tau_1}{R}\big( \mu^2 +\|Du\|_{L^\infty(Q)}^2\big)^\frac{p-1}2
\end{equation*}
for a constant $C=C(N,C_1)$.
\end{lemma}
\begin{proof} Applying the Gluing Lemma \ref{lm:gluing} we obtain for a.e.~$(x_1,t_1), (x_2,t_2)\in Q$  that
\begin{align*}
    | u(x_1,t_1)&-u(x_2,t_2)|\\
    &\le 
    \int_{B_R}\underbrace{| u(x_1,t_1)-u(x,t_1)|}_{\le 2R\| Du\|_{L^\infty (Q)}} \eta_R(x)\,\dx
    +
    \int_{B_R}\underbrace{| u(x,t_2)-u(x_2,t_2)|}_{\le 2R\| Du\|_{L^\infty (Q)}}\eta_R(x)\,\dx\\
    &\phantom{=\,}
    +
    \bigg| 
    \int_{B_R}[ u(x,t_1)-u(x,t_2)]\eta_R(x)\,\dx\bigg|\\
    &\le
    4R\| Du\|_{L^\infty (Q)}
    +
     C\frac{\tau_2-\tau_1}{R}\big( \mu^2 +\|Du\|_{L^\infty(Q)}^2\big)^\frac{p-1}2.
\end{align*}
This proves the asserted oscillation estimate.
\end{proof}

\begin{lemma}\label{lem:osc-p<2}
Let $1<p\le 2$, $\mu\in [0,1]$, and $u$ be a weak solution to the parabolic system~\eqref{p-laplace-intro}  with \eqref{prop-a-intro}$_1$ in the sense of Definition \ref{def:weak-loc}. Assume that $|Du|\in L^\infty_{\rm loc}(E_T)$. Then, for every cylinder $Q_R^{(\lambda)}(z_o)\Subset E_T$  we have
\begin{equation*}
    \osc_{Q_R^{(\lambda)}(z_o)} u
    \le
    CR\Big( \|Du\|_{L^\infty(Q_R^{(\lambda)}(z_o))}+\mu +\lambda\Big)
\end{equation*}
with $C=C(N,C_1)$.
\end{lemma}

\begin{proof} We apply the oscillation Lemma \ref{lem:osc-all-p} with
$Q=Q_R^{(\lambda)}(z_o)$. Consequently for a.e.~$(x_1,t_1), (x_2,t_2)\in Q_R^{(\lambda)}(z_o)$ we have
\begin{align*}
    | u(x_1,t_1)&-u(x_2,t_2)|\\
    &\le
    4R\| Du\|_{L^\infty (Q_R^{(\lambda)}(z_o))}
    +
     C\lambda^{2-p}R\Big( \mu^2 +\|Du\|_{L^\infty(Q_R^{(\lambda)}(z_o))}^2\Big)^\frac{p-1}2.
\end{align*}

To estimate the second term on the right-hand side, we distinguish two cases. First, let $\lambda< \big( \mu^2 +\|Du\|_{L^\infty(Q_R^{(\lambda)}(z_o))}^2\big)^\frac12$ be satisfied. 
Then, we have
\begin{align*}
    | u(x_1,t_1)-u(x_2,t_2)|
    &\le
    4R\| Du\|_{L^\infty (Q_R^{(\lambda)}(z_o))}
    +
     CR\Big( \mu^2 +\|Du\|_{L^\infty(Q_R^{(\lambda)}(z_o))}^2\Big)^\frac{1}2\\
     &\le 
     C R \Big( \mu+\|Du\|_{L^\infty(Q_R^{(\lambda)}(z_o))}\Big).
\end{align*}
In the other case, i.e.~when $\lambda \ge \big( \mu^2 +\|Du\|_{L^\infty(Q_R^{(\lambda)}(z_o))}^2\big)^\frac12$ holds, we get
  \begin{align*}
    | u(x_1,t_1)-u(x_2,t_2)|
    &\le
     C R \Big( \|Du\|_{L^\infty(Q_R^{(\lambda)}(z_o))}+\lambda \Big).
\end{align*}
Merging both cases yields the result.
\end{proof}
\begin{remark}\label{rem:osc}\upshape
In the course of the proof of the Schauder estimates, we always assume 
$A\lambda\ge \mu$ for some $A\ge1$. Therefore, we can limit ourselves to state the estimate
in this case, i.e.
\begin{equation*}
    \osc_{Q_R^{(\lambda)}(z_o)} u
    \le
    CR\Big( \|Du\|_{L^\infty(Q_R^{(\lambda)}(z_o))}+\lambda\Big)
\end{equation*}
with $C=C(N,C_1,A)$.
\end{remark}

\section{Schauder-type estimates}\label{sec:Schauder-for-Lipschitz}
In this section we prove Schauder estimates for bounded weak solutions to the parabolic $p$-Laplace system \eqref{p-laplace-intro} with Hölder continuous coefficients $a(x,t)$ satisfying assumptions \eqref{prop-a-intro} with $\mu\in(0,1]$. We restrict our considerations
to {\it bounded} weak solutions 
$$
u\in L^{\infty}(E_T, \R^k)\cap L^p\big(0,T;W^{1,p}(E,\R^k)\big)\cap C\big([0,T];L^2(E,\R^k)\big).
$$ 
For the moment, we assume additionally that
\begin{equation}\label{ass:Du}
       |Du|\in L^\infty_{\mathrm{loc}}(E_T).
\end{equation}
Assumption \eqref{ass:Du} will be removed in Section \ref{S:6}.

The {\bf general geometric setup} is as follows. With $\varrho:= \tfrac14\min\{1,
\mathrm{dist}_{\mathrm{par}}(\mathsf {K},\partial_\mathrm{par} E_T)\}$ and $\tilde z\in \mathsf {K}$, we  consider standard cylinders
\begin{equation}\label{def:cylinders}
    Q_\varrho (\tilde z)\subset Q_{R_1}(\tilde z)\subset Q_{R_2}(\tilde z)\subset  Q_{2\varrho} (\tilde z)\Subset E_T
\end{equation}
with $\varrho\le R_1<R_2\le 2\varrho \le 1$ and a parameter $\lambda\ge\frac{\mu}{A}$ such that
\begin{equation}
  \label{def-lambda}
  \Big(\mu^2+\|Du\|_{L^\infty(Q_{R_2}(\tilde z))}^2\Big)^{\frac12}
  \le A\lambda,
\end{equation}
where the number $A\ge 1$ will be chosen in terms of given data.
With these quantities we define
\begin{equation}\label{choice-R_o}
  R_o:=\tfrac12\min\big\{1,\lambda^{\frac{p-2}{2}}\big\}(R_2-R_1).
\end{equation}
Then, for any $z_o\in  Q_{R_1}(\tilde z)$  and any $R\le 2R_o$ the intrinsic cylinders $Q_R^{(\lambda)}(z_o)$ are contained in $Q_{R_2}(\tilde z)$. 
For a parameter $\kappa\in(0,1)$ to be specified later, and a radius
$r\in(0,\frac18 R_o)$ we define
\begin{equation}\label{choice-R}
  R:=\Big(\frac{8r}{R_o}\Big)^\kappa R_o
  \quad \Longleftrightarrow \quad
  r=\tfrac18\Big(\frac{R}{R_o}\Big)^{\frac1\kappa-1} R.
\end{equation}
With this particular specification of $R$ we note that $R<R_o$ and $r<\frac18 R$. Moreover, we have
\begin{equation}\label{est:radii}
    \frac{R}{R_o}
    =
    8^\kappa\Big(\frac{r}{R_o}\Big)^\kappa
    \quad\mbox{and}\quad
    \frac{R}{r}
    =
    8^\kappa
    \Big( \frac{R_o}{r}\Big)^{1-\kappa}.
\end{equation}
The choice of the intermediate radius $R$ together with the selection of $\kappa$ facilitates an interpolation argument in Section \ref{S:Campanato} in order to derive a Campanato-type estimate.

\subsection{Freezing coefficients} Let $z_o=(x_o,t_o)\in  Q_{R_1}(\tilde z)$, $r\in(0,\frac18 R_o)$ and $R$ defined according to \eqref{choice-R}. In the following we consider the unique weak solution $w$
to the Cauchy-Dirichlet problem
\begin{equation} \label{comparison-problem}
    \left\{
    \begin{array}{cl}
        \partial_tw-\Div
         \Big( a(x_o,t)\big(\mu^2+|Dw|^2\big)^\frac{p-2}2 Dw\Big)=0
         &\mbox{in $Q_{R}^{(\lambda)}(z_o)$,}\\[6pt]
         w=u &\mbox{on $\partial_{\mathrm{par}} Q_{R}^{(\lambda)}(z_o)$,}
    \end{array}
    \right.
\end{equation}
where $\mu\in[0,1]$, in the sense of Definition \ref{def:weak}.

In view of the $C^{1,\alpha}$-regularity theory by DiBenedetto and Friedman \cite{DB-1d, DiBenedetto-Friedman, DiBenedetto-Friedman2, DiBenedetto-Friedman3} the auxiliary function $w$ satisfies good {\it a priori} estimates which we adapt and present next. Essentially, the result states that if the inclusion $Q_{2\mathfrak{R}}^{(\lambda)}(z_o)\Subset Q_{R}^{(\lambda)}(z_o)$ and the intrinsic relation \eqref{intrinsic} hold, then the Campanato-type estimate \eqref{campanato} holds.

\begin{proposition}\label{prop:apriori}
Let $p>1$, $\mu\in[0,1]$ and $\mathfrak{A}\ge 1$, and suppose that
assumptions~\eqref{prop-a-intro}$_1$ are satisfied. 
There exist constants 
$C >0$ and $\beta\in(0,1)$, depending
only on the data $N,p, C_o, C_1$, and $\mathfrak{A}$, such that whenever $w$
is a weak solution to the parabolic system~\eqref{comparison-problem}$_1$ in the sense of Definition \ref{def:weak-loc}, if 
$Q_{2\mathfrak{R}}^{(\lambda)}(z_o)\Subset Q_{R}^{(\lambda)}(z_o)$ with $\lm>0$ and
$\mathfrak{R}\in(0,\frac12 R)$ 
satisfies
\begin{equation}\label{intrinsic}
    \sup_{Q_{2\mathfrak{R}}^{(\lambda)}(z_o)}\big(\mu^2+|Dw|^2\big)^\frac12\le \mathfrak{A}\lambda,
\end{equation}
then we have  
\begin{align}\label{campanato}
	\biint_{Q_\tau^{(\lambda)}(z_o)}\big|Dw-(Dw)_{z_o,\tau}^{(\lambda)}\big|^p \,\dx\dt
	\le
	C \Big(\frac{\tau}{\mathfrak{R}}\Big)^{\beta p}\lambda_\mu^{p}
    \qquad\mbox{for all $\tau\in(0,\mathfrak{R}]$},
\end{align}
where $\lambda_\mu=\sqrt{\mathfrak{A}^2\lambda^2-\mu^2}$, and $(Dw)_{z_o,\tau}^{(\lambda)}$ is the integral average of $Dw$ on
$Q_\tau^{(\lambda)}(z_o)$.
\end{proposition}
\begin{proof}
In view of \eqref{intrinsic} we may assume $\lambda>\frac{\mu}{\mathfrak{A}}$, since otherwise $Dw\equiv 0$ on $Q_{2\rho}^{(\lambda)}(z_o)$ and the result is trivial. 
In the following we abbreviate $b(t):=a(x_o,t)$. We introduce the auxiliary function 
$$
    v
    :=
    \mathfrak{A}^{-1} w
    \quad\mbox{and}\quad
    \tilde\mu:=
    \mathfrak{A}^{-1}\mu.
$$
Then, $v$ is a weak solution to the parabolic system 
\begin{equation*}
    \partial_t v
    =
    \Div\Big( b(t)\mathfrak{A}^{p-2}\big(\tilde\mu^2+|Dv|^2\big)^\frac{p-2}2Dv\Big)=0
    \quad\mbox{in $E_T$.}
\end{equation*}
Hence, \cite[Theorem~1.3]{BDLS-Tolksdorf} is applicable to $v$ after replacing $C_o,C_1$ by $\mathfrak{A}^{p-2}C_o, \mathfrak{A}^{p-2}C_1$ and $\mu$ by $\tilde\mu$.  
The proof of \cite[Theorem~1.3]{BDLS-Tolksdorf} is presented in \cite[Section~5.2]{BDLS-Tolksdorf}; the estimates are stated for cylinders of the form
\begin{equation*}
    \widetilde Q_\tau^{(\lambda)}(z_o)
    :=
    B_\tau(x_o)\times\big(t_o-(\tilde\mu^2+\lambda^2)^{\frac{2-p}{2}}\tau^2,t_o\big].
\end{equation*}
In particular, with $\lambda_{\tilde\mu}:=\sqrt{\lambda^2-\tilde\mu^2}>0$ we have 
\begin{equation}\label{two-cylinders}
    \widetilde Q_\tau^{(\lambda_{\tilde\mu})}(z_o)
    =
    Q_\tau^{(\lambda)}(z_o)
\end{equation}
for every $\tau\in(0,2\mathfrak{R}]$,
and assumption \eqref{intrinsic} can be rewritten in the form
\begin{equation*}
     \sup_{\widetilde Q_{2\mathfrak{R}}^{(\lambda_{\tilde\mu})}(z_o)}|Dv|\le\lambda_{\tilde\mu}.
\end{equation*}
This corresponds to the sup-estimate in \textit{Step~1} on p.~27 of \cite{BDLS-Tolksdorf}.
Then, we can repeat the
arguments in \textit{Steps~2 -- 4} on pp.~27 -- 29 and arrive at \cite[Eqn. (5.16)]{BDLS-Tolksdorf} with $\lambda_{\tilde\mu}$
in place of $\lambda$. Since the coefficients $b(t)$ are independent of $x$, the upper bound $\rho_o$ on the radius $\mathfrak{R}$ can be avoided in the present situation. 
Therefore, with the Lebesgue representative $\Gamma_{z_o}$ of $Du(z_o)$, 
we obtain the Campanato-type estimate 
\begin{equation*}
	\biint_{\widetilde Q_\tau^{(\lambda_{\tilde\mu})}(z_o)}|Dv-\Gamma_{z_o}|^p \,\dx\dt
	\le
	C \Big(\frac{\tau}{\mathfrak{R}}\Big)^{\beta p} 
	\lambda_{\tilde\mu}^{p}
        \le
	C \Big(\frac{\tau}{\mathfrak{R}}\Big)^{\beta p} 
	\lambda^{p}
\end{equation*}
for every $\tau\in(0,\mathfrak{R}]$, with constants
$C >0$ and $\beta\in(0,1)$ depending on
$N,p,C_o,C_1,\mathfrak{A}$. In view of~\eqref{two-cylinders},
this implies 
\begin{align*}
  \biint_{Q_\tau^{(\lambda)}(z_o)}\big|Dv-(Dv)_{z_o,\tau}^{(\lambda)}\big|^p \,\dx\dt
  \le
    C (p)\biint_{Q_\tau^{(\lambda)}(z_o)}|Dv-\Gamma_{z_o}|^p \,\dx\dt
  \le
    C \Big(\frac{\tau}{\mathfrak{R}}\Big)^{\beta p}\lambda^{p}.
\end{align*}
Recalling that $w=\mathfrak{A}v$, this yields the claim~\eqref{campanato}.
\end{proof}

In the case $p=2$ the intrinsic coupling \eqref{intrinsic} is superfluous, whereas the Campanato-type estimate \eqref{campanato} improves to be 
\[
\biint_{Q_\tau(z_o)}\big|Dw-(Dw)_{z_o,\tau}\big|^2 \,\dx\dt
	\le
	C \Big(\frac{\tau}{\mathfrak{R}}\Big)^{2\beta}\biint_{Q_{\mathfrak R}(z_o)} 
    |Dw|^2 \,\dx\dt.
\]
The right-hand integral can be estimated by the integral of $|Du|^2$ instead of $|Dw|^2$ thanks to the energy inequality in Lemma~\ref{lem:comp-2}, while the comparison estimate in the same lemma yields
\begin{align*}
  \iint_{Q_{\mathfrak{R}}(z_o)}&|Du-Dw|^2\,\dx\dt
  \le
  C \mathfrak{R}^{\alpha p}
    \iint_{Q_{\mathfrak{R}}(z_o)}|Du|^2 \,\dx\dt.
\end{align*}
Then, it is not hard to combine them and obtain that
\begin{align*}
    \biint_{Q_\tau(z_o)}&\big|Du-(Du)_{z_o,\tau}\big|^2 \,\dx\dt \\
    &\le C\Big(\frac{\tau}{\mathfrak{R}}\Big)^{2\beta}\biint_{Q_{\mathfrak R}(z_o)} 
    |Du|^2 \,\dx\dt+C \mathfrak{R}^{2\alpha}
    \iint_{Q_{\mathfrak{R}}(z_o)}|Du|^2 \,\dx\dt.
\end{align*}
Interpolating the radii $\tau$ and $\mathfrak{R}$ will give a Campanato-type estimate for $Du$.

Even though our program follows this spirit, yet the non-linearity of the system brings essential difficulties. In general, we cannot replace $\lm_\mu$ in \eqref{campanato} with integrals of $Dw$ or $Du$.
The following proposition, however, gives an ``interpolative" remedy for this difficulty.

\begin{proposition}\label{prop:camp-w}
Let $w$ be the unique weak solution to the Cauchy-Dirichlet problem \eqref{comparison-problem} in the sense of Definition \ref{def:weak}, and assume that \eqref{def-lambda} is satisfied for some $A\ge 1$. Then, there exist  exponents $\theta_2=\theta_2(N,p)>0$ and $\beta\in (0,1)$, and a constant $C>0$ such that for any $\epsilon\in (0,1]$ and any $\tau\in(0, \tfrac18R]$  the Campanato-type estimate
\begin{align*}
     \biint_{Q_\tau^{(\lambda)}(z_o) }&\big|Dw -(Dw)_{z_o,\tau}^{(\lambda)}\big|^p\,\dx\dt\\
     &\le
     C\Big(\frac{\tau}{R}\Big)^{\beta }
     \max\Bigg\{
     \epsilon^{-\theta_2}
    \biint_{Q_{R}^{(\lambda)}(z_o)} 
    \big(\mu^2+|Du|^2\big)^\frac{p}2\,\dx\dt\,,\,
    \epsilon^{p-1}\lambda^p
   \Bigg\},
\end{align*}
holds true. The constants $\beta$ and $C$ depend only on $N,p,C_o,C_1,A$ and in the case $1<p<2$ also on $k$.
\end{proposition}

\begin{proof}
Prior to estimating $Dw$, we recall two comparison estimates to be used. In fact, Lemma \ref{lem:comp-1} gives the oscillation estimate
\begin{equation}\label{est:osc-w}
    \osc_{Q_{R}^{(\lambda)}(z_o)} w
    \le 
    \sqrt{k}
    \osc_{Q_{R}^{(\lambda)}(z_o)} u
    =:
    \sqrt{k}\,
    \boldsymbol\omega_{z_o,R}^{(\lambda)},
\end{equation}
whereas Lemma \ref{lem:comp-2} gives the {\it energy inequality}
\begin{align}\label{energy-lem:comp-2}
    \iint_{Q_{R}^{(\lambda)}(z_o)}
    \big(\mu^2+|Dw|^2\big)^\frac{p}2\,\dx\dt
    \le
    C \iint_{Q_{R}^{(\lambda)}(z_o)}
    \big(\mu^2+|Du|^2\big)^\frac{p}2\,\dx\dt
\end{align}
with $C=C(p,C_o,C_1)$. 

Next, we aim to bound $Dw$ via the $L^p$-norm of $Du$ interpolating with $\lm$. The gradient bounds from Section \ref{S:grad-bound} are at our disposal as $a(x_o,t)$ depends only on $t$.   
Due to the different forms of quantitative $L^\infty$-gradient estimates we distinguish between two cases. In the {\bf sub-quadratic case} $1<p<2$, Corollary \ref{cor:L-infty-est-p<2} and the oscillation estimate \eqref{est:osc-w} yield 
\begin{align}\label{sup-bound-Dw}\nonumber
    \sup_{Q_{\frac14 R}^{(\lambda)}(z_o)}& |Dw|\\
    &\le 
    \frac{C\lambda^\frac12}{\epsilon^\theta}
    \bigg[ 
        \frac{\boldsymbol\omega_{z_o,R}^{(\lambda)}}{\lambda R} +
        \Big(\frac{\boldsymbol\omega_{z_o,R}^{(\lambda)}}{\lambda R}\Big)^\frac{2}{p}
    \bigg]^\frac{N(2-p)+2p}{4p}
    \bigg[
    \biint_{Q_{R}^{(\lambda)}(z_o)} 
    |Dw|^p\,\dx\dt
    \bigg]^\frac1{2p}\vee\eps\lambda
\end{align}
for every $\epsilon\in (0,1]$, since $\lambda\ge\frac{\mu}{A}$ as a consequence of \eqref{def-lambda}. 
The constant $C$ depends only on $N,k,p,C_o, C_1$ and $A$, while $\theta$ on $N$ and $p$. Here, $k$ enters $C$ via \eqref{est:osc-w}. 
Now we claim  $\boldsymbol\omega_{z_o,R}^{(\lm)}
    \le C\lm R$.
Indeed, by \eqref{def-lambda} and $Q_{ R}^{(\lambda)}(z_o)\subset Q_{R_2}(\tilde z)$ we have
\begin{equation*}
      A\lambda
      \ge 
      \Big(\mu^2+\|Du\|_{L^\infty(Q_{R_2}(\tilde z))}^2\Big)^{\frac12}
      \ge
    \Big(\mu^2+\|Du\|_{L^\infty(Q_{ R}^{(\lambda)}(z_o))}^2\Big)^{\frac12}.
\end{equation*}
Therefore, from Remark \ref{rem:osc} we infer
\begin{equation*}
    \boldsymbol\omega_{z_o,R}^{(\lm)}
    \le 
    C(N,C_1,A)R
    \Big( \underbrace{\|Du\|_{L^\infty(Q_{ R}^{(\lambda)}(z_o))}}_{\le\, A\lambda}+\lambda\Big)
    \le
    C(N,C_1,A)R\lambda.
\end{equation*}
Using this  as well as  \eqref{energy-lem:comp-2} in \eqref{sup-bound-Dw}, we have shown in the case $1<p<2$ that 
\begin{align}\label{sup-bound-Dw-p<2}
    \sup_{Q_{\frac14 R}^{(\lambda)}(z_o)} |Dw|
    &\le 
    \frac{C\lambda^\frac12}{\epsilon^\theta}
    \bigg[
    \biint_{Q_{R}^{(\lambda)}(z_o)} 
    \big(\mu^2+|Du|^2\big)^\frac{p}2\,\dx\dt
    \bigg]^\frac1{2p}\vee \eps\lambda
\end{align}
for every $\epsilon\in (0,1]$ with  constants $C=C(N,k,p,C_o,C_1,A)$ and $\theta=\theta (N,p)>0$.

Next, we shall derive an similar estimate as \eqref{sup-bound-Dw-p<2} in the {\bf super-quadratic case $p\ge 2$}. For this purpose we apply Corollary \ref{cor:L-infty-est-p>2-intrinsic-improved} with $C_2=0$ and $\sigma=\frac14$ on $Q_{R}^{(\lambda)}(z_o)$ and then the energy estimate \eqref{energy-lem:comp-2}.  In the resulting inequality we reduce the exponent of the energy integral of $u$ from $\frac12$ to $\frac{1}{2p}$ using condition \eqref{def-lambda}. In this way we get
\begin{align}\label{est:sup-Dw-p>2}\nonumber
     \sup_{Q_{\frac14 R}^{(\lambda)}(z_o)} |Dw|
    &\le 
    \frac{C}{\epsilon^\theta}
    \bigg[
    \lambda^{2-p}
    \biint_{Q_{R}^{(\lambda)}(z_o)} 
    |Dw|^p\,\dx\dt
    \bigg]^\frac1{2}\vee\eps\lm\\\nonumber
    &\le 
    \frac{C\lambda^\frac{2-p}2}{\epsilon^\theta}
    \bigg[
    \biint_{Q_{R}^{(\lambda)}(z_o)} 
    \big(\mu^2+|Du|^2\big)^\frac{p}2\,\dx\dt
    \bigg]^\frac1{2}\vee\eps\lm\\\nonumber
    &=
    \frac{C\lambda^\frac{1}2}{\epsilon^\theta}
    \bigg[
    \biint_{Q_{R}^{(\lambda)}(z_o)} 
    \big(\mu^2+|Du|^2\big)^\frac{p}2\,\dx\dt
    \bigg]^\frac1{2p}\\\nonumber
    &\qquad\qquad\cdot
    \lambda^\frac{1-p}2
    \bigg[
    \underbrace{
    \biint_{Q_{R}^{(\lambda)}(z_o)} 
    \big(\mu^2+|Du|^2\big)^\frac{p}2\,\dx\dt
    }_{\le\, (A\lambda)^p}
\bigg]^{\frac{p-1}{2p}}\vee\eps\lm\\
    &\le
    \frac{C\lambda^\frac{1}2}{\epsilon^\theta}
    \bigg[
    \biint_{Q_{R}^{(\lambda)}(z_o)} 
    \big(\mu^2+|Du|^2\big)^\frac{p}2\,\dx\dt
    \bigg]^\frac1{2p}\vee\eps\lm,
\end{align}
for every $\epsilon\in (0,1]$, with a constant $C=C(N,p,C_o,C_1,A)$ and $\theta=\theta (N,p)> 0$. 

Inequalities \eqref{sup-bound-Dw-p<2} and \eqref{est:sup-Dw-p>2} have the same form and they also play a central role in the following. First, we use them with $\epsilon =1$ together with condition \eqref{def-lambda} to ensure that $w$ satisfies an intrinsic coupling on $Q_{\frac14 R}^{(\lambda)}(z_o)$. Indeed, we have
\begin{align*} 
    \sup_{Q_{\frac14 R}^{(\lambda)}(z_o)} |Dw|
    &\le
    C\lambda^\frac12  \bigg[
    \underbrace{
    \biint_{Q_{R}^{(\lambda)}(z_o)} 
    \big(\mu^2+|Du|^2\big)^\frac{p}2\,\dx\dt}_{\le\,(A\lambda)^p}
    \bigg]^\frac1{2p}\vee\lm
    \le C\lm,
\end{align*}
so that
\begin{equation}\label{sub-bound-Dw-lambda}
     \sup_{Q_{\frac14 R}^{(\lambda)}(z_o)} \big(\mu^2+|Dw|^2\big)^\frac12
     \le 
     C\lm.
\end{equation}
Therefore, the requirement of the {\it a priori} estimate from
Proposition \ref{prop:apriori} is fulfilled for $w$ on $Q_{\frac14R}^{(\lambda)}(z_o)$ with $C$ in place of $\mathfrak{A}$. Replacing 
$\mathfrak{R}$ by $\frac18 R$ in Proposition \ref{prop:apriori}, we obtain for any $0<\tau\le\frac18R$ 
\begin{equation}\label{pre-Campanato-Dw}
     \biint_{Q_\tau^{(\lambda)}(z_o) }\big|Dw -(Dw)_{z_o,\tau}^{(\lambda)}\big|^p\,\dx\dt
     \le
     C\Big(\frac{\tau}{R}\Big)^{\beta p}\lambda^p
\end{equation}
for some exponent $\beta=\beta (N,p,C_o,C_1,A)\in (0,1)$. In the sub-quadratic case $\beta$ depends additionally on $k$, entering via $C$ in \eqref{sub-bound-Dw-lambda}. Let us further estimate \eqref{pre-Campanato-Dw} by the gradient bounds in \eqref{sup-bound-Dw-p<2}, resp. \eqref{est:sup-Dw-p>2}:
\begin{align*}
     \biint_{Q_\tau^{(\lambda)}(z_o) }&\big|Dw -(Dw)_{z_o,\tau}^{(\lambda)}\big|^p\,\dx\dt\\
     &\le
     C(p)\|Dw\|^{p-1}_{L^\infty(Q_{\frac14R}^{(\lambda)}(z_o))}
     \bigg[
     \underbrace{
     \biint_{Q_\tau^{(\lambda)}(z_o) }\big|Dw -(Dw)_{z_o,\tau}^{(\lambda)}\big|^p\,\dx\dt}_{\le C(\tau/R)^{\beta p}\lambda^p}
     \bigg]^\frac{1}{p}\\
     &\le
     C\Bigg[
    \frac{\lambda^\frac{1}2}{\epsilon^\theta}
    \bigg[
    \biint_{Q_{R}^{(\lambda)}(z_o)} 
    \big(\mu^2+|Du|^2\big)^\frac{p}2\,\dx\dt
    \bigg]^\frac1{2p}\vee\eps\lm
     \Bigg]^{p-1} \Big(\frac{\tau}{R}\Big)^{\beta }\lambda\\
     &\le
     C\Big(\frac{\tau}{R}\Big)^{\beta }
     \Bigg[
     \frac{\lambda^\frac{p+1}{2}}{\epsilon^{\theta (p-1)}}
     \bigg[
    \biint_{Q_{R}^{(\lambda)}(z_o)} 
    \big(\mu^2+|Du|^2\big)^\frac{p}2\,\dx\dt
    \bigg]^\frac{p-1}{2p} \vee 
    \eps^{p-1}\lm^p
     \Bigg].
\end{align*}
Although this procedure reduces the exponent from $\beta p$ to $\beta$, it interpolates between the energy of $u$ and $\lambda$.

If the second entry in the maximum dominates the first one, the claimed Campanato-type estimate is obvious. Therefore, we can consider the other case:
\begin{align*}
    \frac{\lambda^\frac{p+1}{2}}{\epsilon^{\theta (p-1)}}
     \bigg[
    \biint_{Q_{R}^{(\lambda)}(z_o)} 
    \big(\mu^2+|Du|^2\big)^\frac{p}2\,\dx\dt
    \bigg]^\frac{p-1}{2p}> \eps^{p-1}\lm^p,
\end{align*}
which is the same as
\begin{align*}
    \lm
    <
       \frac{1}{\epsilon^{2(\theta+1)}}
     \bigg[
    \biint_{Q_{R}^{(\lambda)}(z_o)} 
    \big(\mu^2+|Du|^2\big)^\frac{p}2\,\dx\dt
    \bigg]^\frac{1}{p}
    =:
    \frac{\mathsf E^\frac1p}{\epsilon^{2(\theta+1)}},
\end{align*}
where the meaning of $\mathsf E$ is obvious.
Use this bound of $\lm$ to estimate the first entry by
\begin{align*}
       \frac{\lambda^\frac{p+1}{2}}{\epsilon^{\theta (p-1)}}&
     \bigg[
    \biint_{Q_{R}^{(\lambda)}(z_o)} 
    \big(\mu^2+|Du|^2\big)^\frac{p}2\,\dx\dt
    \bigg]^\frac{p-1}{2p}\\
    &<
    \bigg[
    \frac{\mathsf E^\frac1p}{\epsilon^{2(\theta+1)}}
    \bigg]^\frac{p+1}{2} \cdot 
    \frac{\mathsf E^\frac{p-1}{2p}}{\epsilon^{\theta (p-1)}}
    =
    \frac{1}{\eps^{2\theta p+p+1}}
    \biint_{Q_{R}^{(\lambda)}(z_o)} 
    \big(\mu^2+|Du|^2\big)^\frac{p}2\,\dx\dt.
\end{align*}
Joining the two cases we obtain for any $\tau\in(0, \tfrac18R]$ that
\begin{align*}
     \biint_{Q_\tau^{(\lambda)}(z_o) }&\big|Dw -(Dw)_{z_o,\tau}^{(\lambda)}\big|^p\,\dx\dt\\
     &\le
     C\Big(\frac{\tau}{R}\Big)^{\beta }
     \max\Bigg\{
     \epsilon^{-\theta_2}
    \biint_{Q_{R}^{(\lambda)}(z_o)} 
    \big(\mu^2+|Du|^2\big)^\frac{p}2\,\dx\dt \,,\, \epsilon^{p-1}\lambda^p
   \Bigg\},
\end{align*}
where $\theta_2(N,p):= 2\theta p +p+1$. 
\end{proof}

\subsection{A Campanato-type estimate} \label{S:Campanato}
Having the Campanato-type estimate for $Dw$ at hand, the goal of this section is to derive a Campanato-type estimate for $Du$. In order for that, recall $r\in(0,\frac18 R_o)$ and the intermediate radius $R$ so-defined in \eqref{choice-R}. First we write down the mean oscillation of $Du$ and use the quasi-minimality of the mean value to estimate. Then we invoke the comparison estimate from Lemma~\ref{lem:comp-2}, that is,
\begin{align}\nonumber
    \iint_{Q_{R}^{(\lambda)}(z_o)}|Du-Dw|^p\,\dx\dt
    &\le C\, R^{\alpha_\ast p}
    \iint_{Q_{R}^{(\lambda)}(z_o)}
    \big(\mu^2+|Du|^2\big)^\frac{p}2\,\dx\dt,
\end{align}
where $\alpha_\ast$ is defined in \eqref{alpha_ast}, 
as well as the Campanato-type estimate from Proposition~\ref{prop:camp-w} for $Dw$. As such we obtain a preliminary version as follows:
\begin{align}\label{campanato-u}\nonumber
     \biint_{Q_r^{(\lambda)}(z_o) }&\big|Du -(Du)_{z_o,r}^{(\lambda)}\big|^p\,\dx\dt
     \le
     C(p) \biint_{Q_r^{(\lambda)}(z_o) }\big|Du -(Dw)_{z_o,r}^{(\lambda)}\big|^p\,\dx\dt\\\nonumber
     &\le
     C(p)\Big(\frac{R}{r}\Big)^{N+2}
     \biint_{Q_R^{(\lambda)}(z_o) } |Du -Dw|^p\,\dx\dt\\\nonumber
     &\qquad\qquad\qquad\qquad
    +
     C(p)\biint_{Q_r^{(\lambda)}(z_o) }\big|Dw -(Dw)_{z_o,r}^{(\lambda)}\big|^p\,\dx\dt\\\nonumber
     &\le
     C\Big(\frac{R}{r}\Big)^{N+2}R^{\alpha_\ast p}
     \mathsf E_{z_o,R}^{(\lambda)} +
     C\Big(\frac{r}{R}\Big)^{\beta }
     \max
     \Big\{
     \epsilon^{-\theta_2}
    \mathsf E_{z_o,R}^{(\lambda)} \,,\, \epsilon^{p-1}\lambda^p
     \Big\}\\
     &=:\mathbf{I}+\mathbf{II},
\end{align}
where we abbreviated 
$$
    \mathsf E_{z_o,R}^{(\lambda)}
    :=
    \biint_{Q_R^{(\lambda)}(z_o)}\big(\mu^2+|Du|^2\big)^\frac{p}2\,\dx\dt.
$$
Note that the Campanato-type estimate from Proposition~\ref{prop:camp-w} is applicable since $r<\frac18 R$ by the choice of $R$ in~\eqref{choice-R}. The constant $C$ in the term $\mathbf I$ depends only on $p,C_o,C_1$, while the one in $\mathbf{II}$ depends additionally on $N$ and $A$, and in the sub-quadratic case also on $k$. The H\"older exponent $\alpha_\ast$ was defined in \eqref{alpha_ast}, while $\beta$ is the H\"older exponent from \eqref{pre-Campanato-Dw}.

The radius $R$ so-defined in \eqref{choice-R} is intermediate. To proceed further we need to estimate the integral
$\mathsf E_{z_o,R}^{(\lambda)}$ in terms of 
$\mathsf E_{z_o,R_o}^{(\lambda)}$. If we simply enlarge the domain from $Q_{R}^{(\lambda)}(z_o)$ to $Q_{R_o}^{(\lambda)}(z_o)$ in the integral, we would get a factor $(R_o/R)^{N+2}\ge 1$ whose exponent is too large to be absorbed later.
For this reason we perform an interpolation argument and split
\begin{equation*}
\mathsf E_{z_o,R}^{(\lambda)}=\big(\mathsf E_{z_o,R}^{(\lambda)}\big)^{1-\delta}\big(\mathsf E_{z_o,R}^{(\lambda)}\big)^\delta,
\end{equation*}
with some $\delta\in (0,1)$ to be chosen later, and estimate the first term with the help of condition \eqref{def-lambda}.
Thereby we take advantage of the fact that $Q_{R}^{(\lambda)}(z_o)$ is contained in $Q_{R_2}(\tilde z)$. In the second term, we enlarge the domain of integration. In this way we get
\begin{align*}
    \mathsf E_{z_o,R}^{(\lambda)}
    &\le
    (A\lambda)^{p(1-\delta)}\Big(\frac{R_o}{R}\Big)^{\delta (N+2)} \big(\mathsf E_{z_o,R_o}^{(\lambda)}\big)^\delta.
\end{align*}
Assuming that $\alpha_\ast p -\delta (N+2)\ge 0$, and using \eqref{est:radii} we have
\begin{align*}
    \mathbf I
    &\le
    C \Big( \frac{R}{r}\Big)^{N+2}
    \Big( \frac{R_o}{R}\Big)^{\delta (N+2)}R^{\alpha_\ast p}
    (A\lambda)^{p(1-\delta )} 
    \big(\mathsf E_{z_o,R_o}^{(\lambda)}\big)^\delta\\
    &=
    C \Big( \frac{R}{r}\Big)^{N+2} \Big(\frac{R}{R_o}\Big)^{\alpha_\ast p-\delta (N+2)} R_o^{\alpha_\ast p}(A\lambda)^{p(1-\delta )} 
    \big(\mathsf E_{z_o,R_o}^{(\lambda)}\big)^\delta\\
    &\le
    C\Big( \frac{r}{R_o}\Big)^{\alpha_\ast p\kappa- (1+\kappa \delta -\kappa)(N+2)}R_o^{\alpha_\ast p}
    (A\lambda)^{p(1-\delta )} 
    \big(\mathsf E_{z_o,R_o}^{(\lambda)}\big)^\delta\\
    &\le
    C\Big( \frac{r}{R_o}\Big)^{\alpha_\ast p\kappa- (1+\kappa \delta -\kappa)(N+2)}(A\lambda)^{p(1-\delta )} 
    \big(\mathsf E_{z_o,R_o}^{(\lambda)}\big)^\delta.
\end{align*}
To obtain the last line  we used $R_o\le\frac12 R_2\le\frac12$, which holds true  by the choice of $R_o$ in \eqref{choice-R_o}.
For the estimate of $\mathbf{II}$ we use $R\le R_o$ and again \eqref{est:radii} to get
\begin{align*}
    \mathbf{II}
    &\le
    C\Big(\frac{r}{R}\Big)^{\beta }
     \max\bigg\{
     \Big(\frac{R_o}{R}\Big)^{\delta (N+2)}
     \frac{\lambda^{p(1-\delta)}}{\epsilon^{\theta_2}}
    \big(\mathsf E_{z_o,R_o}^{(\lambda)}
    \big)^\delta,\epsilon^{p-1}\lambda^p
     \bigg\}\\
     &=
     C 
     \Big(\frac{r}{R_o}\Big)^{\beta }
    \Big(\frac{R}{R_o}\Big)^{-\beta -\delta (N+2)}
    \max\bigg\{
     \frac{\lambda^{p(1-\delta)}}{\epsilon^{\theta_2}} 
    \big(\mathsf E_{z_o,R_o}^{(\lambda)}
    \big)^\delta,
    \underbrace{
    \Big(\frac{R}{R_o}\Big)^{\delta (N+2)}}_{\le 1}
     \epsilon^{p-1}\lambda^p
     \bigg\}\\
     &\le
     C \Big(\frac{r}{R_o}\Big)^{\beta -\kappa\beta -\kappa\delta (N+2) }
        \max\bigg\{
     \frac{\lambda^{p(1-\delta)}}{\epsilon^{\theta_2}}
    \big(\mathsf E_{z_o,R_o}^{(\lambda)}
    \big)^\delta,
     \epsilon^{p-1}\lambda^p
     \bigg\}.
\end{align*}
The choice of
\begin{equation*}
    \kappa:=\frac{N+2+\beta}{N+2+\beta+\alpha_\ast p}\in (0,1)
\end{equation*}
ensures that the exponents of $r/R_o$  in the estimates for $\mathbf I$ and $\mathbf{II}$ coincide; the common exponent is
\begin{equation*}
    \frac{\alpha_\ast\beta p}{N+2+\beta +\alpha_\ast p}-\delta\frac{(N+2)(N+2+\beta)}{N+2+\beta +\alpha_\ast p}.
\end{equation*}
Therefore, it is natural to choose $\delta$ in the form
\begin{equation*}
    \delta:=\frac{\alpha_\ast\beta p}{2(N+2)(N+2+\beta)}\in (0,1).
\end{equation*}
Note that the condition $\delta\le \frac{\alpha_\ast p}{N+2}$
is satisfied by this  choice.
In this way, the exponents of $r/R_o$ on the right-hand side of the estimates of $\mathbf I$ and $\mathbf{II}$ are equal to 
$\alpha_op$, where
\begin{equation}\label{def:alpha_o}
    \alpha_o:=\frac{\alpha_\ast\beta }{2(N+2+\beta +\alpha_\ast p)}
    \in (0,1).
\end{equation}
Now we use the estimates of $\mathbf I$ and $\mathbf{II}$ in \eqref{campanato-u} to obtain
for every $r\in (0,\frac18 R_o)$ that
\begin{align*}
    \biint_{Q_r^{(\lambda)}(z_o) }\big|Du -(Du)_{z_o,r}^{(\lambda)}\big|^p\,\dx\dt
     &\le
    C \Big(\frac{r}{R_o}\Big)^{\alpha_op }
        \max\bigg\{
     \frac{\lambda^{p(1-\delta)}}{\epsilon^{\theta_2}}
    \big(\mathsf E_{z_o,R_o}^{(\lambda)}
    \big)^\delta,
    \epsilon^{p-1}\lambda^p
     \bigg\}
\end{align*}
for a constant $C=C(N,p,C_o,C_1,A)$. In the sub-quadratic case the constant $C$ additionally depends on $k$. If the second entry in the maximum dominates the first, we have
\begin{align*}
    \biint_{Q_r^{(\lambda)}(z_o) }\big|Du -(Du)_{z_o,r}^{(\lambda)}\big|^p\,\dx\dt
     &\le
    C \Big(\frac{r}{R_o}\Big)^{\alpha_op }
    \epsilon^{p-1}\lambda^p.
\end{align*}
In the opposite case, i.e.~if
\begin{equation*}
    \frac{\lambda^{p(1-\delta)}}{\epsilon^{\theta_2}}
    \big(\mathsf E_{z_o,R_o}^{(\lambda)}
    \big)^\delta
    >
    \epsilon^{p-1}\lambda^p
    \quad\Longleftrightarrow\quad
    \lm<\frac{\big(\mathsf E_{z_o,R_o}^{(\lambda)}
    \big)^\frac1p}{\eps^\frac{\theta_2+p-1}{p\delta}},
\end{equation*}
we obtain for the first entry in the maximum
\begin{align*}
    \frac{\lambda^{p(1-\delta)}}{\epsilon^{\theta_2}}
    \big(\mathsf E_{z_o,R_o}^{(\lambda)}
    \big)^\delta
    &<
    \frac{1}{\epsilon^{\theta_2}}
    \bigg[
    \frac{\big(\mathsf E_{z_o,R_o}^{(\lambda)}
    \big)^\frac1p}{\eps^\frac{\theta_2+p-1}{p\delta}}
    \bigg]^{p(1-\delta)}\big(\mathsf E_{z_o,R_o}^{(\lambda)}
    \big)^\delta
    \equiv
    \frac{\mathsf E_{z_o,R_o}^{(\lambda)}}{\eps^{\theta_3}},
\end{align*}
where $\theta_3:= \frac{\theta_2+(p-1)(1-\delta)}{\delta}$. Note that 
$\theta_3=\theta_3(N,p,C_o,C_1,\alpha)$. In the 
sub-quadratic case $\theta_3$ depends additionally on $k$. Joining the two cases we obtain
\begin{align*}
    \biint_{Q_r^{(\lambda)}(z_o) }\big|Du -(Du)_{z_o,r}^{(\lambda)}\big|^p\,\dx\dt
     &\le
    C \Big(\frac{r}{R_o}\Big)^{\alpha_op }
        \max\Big\{
     \epsilon^{-\theta_3}
    \mathsf E_{z_o,R_o}^{(\lambda)}
    ,
    \epsilon^{p-1}\lambda^p
     \Big\}.
\end{align*}
The preceding inequality holds for any center $z_o\in Q_{R_1}(\tilde z)$ and for any radius $r\in(0,\frac18 R_o)$. For radii $r\in [\frac18 R_o,R_o]$ it still holds as 
\begin{align*}
    \biint_{Q_r^{(\lambda)}(z_o) }&\big|Du -(Du)_{z_o,r}^{(\lambda)}\big|^p\,\dx\dt
     \le
    C(p) \biint_{Q_r^{(\lambda)}(z_o) }|Du|^p\,\dx\dt\\
    &\le
    C(p) \Big( \frac{r}{R_o}\Big)^{\alpha_op}\Big( \frac{R_o}{r}\Big)^{N+2+\alpha_op}
    \biint_{Q_{R_o}^{(\lambda)}(z_o) }|Du|^p\,\dx\dt\\
    &\le
    C(p)8^{N+2+\alpha_o p} \Big( \frac{r}{R_o}\Big)^{\alpha_op}
    \biint_{Q_{R_o}^{(\lambda)}(z_o) }|Du|^p\,\dx\dt\\
    &\le
     C(N,p) \Big( \frac{r}{R_o}\Big)^{\alpha_op} \mathsf E_{z_o,R_o}^{(\lambda)}.
\end{align*}
Together the last two inequalities imply for any 
$z_o\in Q_{R_1}(\tilde z)$, any $r\in (0,R_o]$,
and any $\epsilon\in (0,1]$ that
\begin{align}\label{Campanato-u-1}
       \biint_{Q_r^{(\lambda)}(z_o) }&\big|Du -(Du)_{z_o,r}^{(\lambda)}\big|^p\,\dx\dt
     \le
    C \Big(\frac{r}{R_o}\Big)^{\alpha_op }
      \max  \Big\{
     \epsilon^{-\theta_3}
    \mathsf E_{z_o,R_o}^{(\lambda)},
    \epsilon^{p-1}\lambda^p
     \Big\}.
\end{align}
Note that $C$ depends on $N,p,C_o,C_1,A$, while $\alpha_o\in (0,1)$ and $\theta_3>0$ additionally depend on $\alpha$.  Moreover, in the sub-quadratic case $1<p<2$ all constants depend also on $k$. Choosing $\epsilon =1$ in \eqref{Campanato-u-1} we obtain for any $ z_o\in Q_{R_1}(\tilde z)$ the {\bf Campanato-type estimate}
\begin{align}\label{Campanato-u-2}
       \biint_{Q_r^{(\lambda)}(z_o) }&\big|Du -(Du)_{z_o,r}^{(\lambda)}\big|^p\,\dx\dt
     \le
    C \Big(\frac{r}{R_o}\Big)^{\alpha_op }\lambda^p
    \quad\forall\, r\in (0,R_o].
\end{align}
Here we exploited  the fact that $Q_{R_o}^{(\lambda)}(z_o)\subset Q_{R_2}(\tilde z)$ and condition \eqref{def-lambda} in order to estimate the integral $\mathsf E_{z_o,R_o}^{(\lambda)}$ in terms of $\lambda^p$.

\subsection{Quantitative local gradient bound} \label{sec:gradient-bound}
Based on the Campanato-type estimates \eqref{Campanato-u-1} and \eqref{Campanato-u-2} and still assuming that $|Du|$ is locally bounded, in this section we show an explicit quantitative $L^\infty$-gradient estimate in terms of the oscillation. 

\begin{proposition}\label{prop:quant-bound-for-Du}
Let $p>1$, $\mu\in (0,1]$, and $u$ be a bounded weak solution to the parabolic system \eqref{p-laplace-intro} with \eqref{prop-a-intro} in the sense of Definition \ref{def:weak-loc}. Assume \eqref{ass:Du} holds true. Then there exists a constant $C=C(N,p,C_o,C_1,\alpha,k)$  such that
\begin{align*}
    \sup_{Q_{\varrho}(\tilde z)}|Du|
    &\le
    C\bigg[\frac{\osc_{Q_{2\varrho}(\tilde z)}u}{\varrho}+\Big(
     \frac{\osc_{Q_{2\varrho}(\tilde z)}u}{\varrho}\Big)^\frac{2}{p}+\mu
    \bigg],
\end{align*}
for any $Q_{2\varrho}(\tilde z)\Subset E_T$ with $2\varrho\le 1$.
\end{proposition}

\begin{proof}
The setup  is the same as described in \eqref{def:cylinders}. In particular, $Q_{2\rho}(\tilde z)\Subset E_T$ denotes a standard cylinder with radius $2\varrho\le 1$ and $R_1,R_2$ are radii such that $\rho\le R_1<R_2\le 2\rho$.  
Consider $z_o\in Q_{R_1}(\tilde z)$ and $\lambda\ge \frac{\mu}{A}$ satisfying \eqref{def-lambda} for some $A\ge 1$. 
This choice allows us to apply the Campanato-type estimates \eqref{Campanato-u-1} based on which we first show that the means $(Du)_{z_o,r}^{(\lambda)}$ are convergent in the limit $r\downarrow 0$. To this aim we consider a dyadic sequence of radii $r_i:=2^{-i}R_o$, with $i\in\N_0$, and show that the means $(Du)_{z_o,r_i}^{(\lambda)}$ form a Cauchy sequence. Thus, let $j<\ell$ be in $\N_0$. Using~\eqref{Campanato-u-1} with $r_i$ for $i\in\{ j,\dots,\ell-1\}$ instead of $r$, we have
\begin{align}\label{Campanato-r_i}
    \biint_{Q_{r_{i}}^{(\lambda)}(z_o)} \big| Du-(Du)_{z_o,r_i}^{(\lambda)}\big|^p\,\dx\dt
    &\le
    C\Big(\frac{r_i}{R_o}\Big)^{\alpha_op}
     \max\Big\{
     \epsilon^{-\theta_3}
     \mathsf E_{z_o,R_o}^{(\lambda)},
     \eps^{p-1}\lm^p
    \Big\},
\end{align}
so that 
\begin{align}\label{Cauchy-sequence}\nonumber
    \big| (Du)_{z_o,r_\ell}^{(\lambda)}-(Du)_{z_o,r_j}^{(\lambda)}\big|
    &\le
    \sum_{i=j}^{\ell-1}
    \big| (Du)_{z_o,r_{i+1}}^{(\lambda)}-(Du)_{z_o,r_i}^{(\lambda)}\big|\\\nonumber
    &\le
    \sum_{i=j}^{\ell-1}\biint_{Q_{r_{i+1}}^{(\lambda)}(z_o)} \big| Du-(Du)_{z_o,r_i}^{(\lambda)}\big|\,\dx\dt\\\nonumber
    &\le 2^{N+2}\sum_{i=j}^{\ell-1}
    \bigg[
    \biint_{Q_{r_{i}}^{(\lambda)}(z_o)} \big| Du-(Du)_{z_o,r_i}^{(\lambda)}\big|^p\,\dx\dt
    \bigg]^\frac1p\\\nonumber
    &\le
    C 
    \max\Big\{
     \epsilon^{-\frac{\theta_3}p}
     \big(\mathsf E_{z_o,R_o}^{(\lambda)}\big)^\frac1p,
     \eps^\frac{p-1}p\lm
    \Big\} \sum_{i=j}^{\ell-1}\Big(\frac{r_i}{R_o}\Big)^{\alpha_o }\\
    &\le
    C \frac{2^{-j\alpha_o}}{1-2^{-\alpha_o }}
    \max\Big\{
     \epsilon^{-\frac{\theta_3}p}
     \big(\mathsf E_{z_o,R_o}^{(\lambda)}\big)^\frac1p,
     \eps^\frac{p-1}p\lm
    \Big\}.
\end{align}
Recall that $C$ depends on $N,p,C_o,C_1,A$, while $\alpha_o\in (0,1)$ and $\theta_3>0$ additionally depend on $\alpha$.
Since
the right-hand side converges to $0$ as $j\to\infty$, we conclude that the sequence of mean values $ \big( (Du)_{z_o, r_i}^{(\lambda)}\big)_{i\in \N_0}$ is a Cauchy sequence. Therefore, the limit 
\begin{equation*}
    \Gamma_{z_o}:=\lim_{i\to\infty} (Du)_{z_o, r_i}^{(\lambda)}\in \R^{kN}
\end{equation*}
exists. This allows us to pass to the limit $\ell\to\infty$ in \eqref{Cauchy-sequence} with the result
\begin{align}\label{discrete-Lebesgue}
    \big| \Gamma_{z_o}-(Du)_{z_o,r_j}^{(\lambda)}\big|
    &\le
    \frac{C}{1-2^{-\alpha_o }}\Big(\frac{r_j}{R_o}\Big)^{\alpha_o }
    \max\Big\{
     \epsilon^{-\frac{\theta_3}p}
     \big(\mathsf E_{z_o,R_o}^{(\lambda)}\big)^\frac1p,
     \eps^\frac{p-1}p\lm
    \Big\}.
\end{align}
The last inequality true for the sequence $r_j$ actually admits a ``continuous" version for all radii $r\in (0,R_o]$.
In fact, given $r\in (0,R_o]$, we choose $j\in\N_0$ with $r_{j+1}<r\le r_j$. In view of \eqref{discrete-Lebesgue} and \eqref{Campanato-r_i} we obtain
\begin{align}\nonumber\label{cont-Lebesgue}
    \big| \Gamma_{z_o}-(Du)_{z_o,r}^{(\lambda)}\big|
    &\le
    \big| \Gamma_{z_o}-(Du)_{z_o,r_j}^{(\lambda)}\big|
    +
    \big| (Du)_{z_o,r_j}^{(\lambda)}-(Du)_{z_o,r}^{(\lambda)}\big|\\\nonumber
    &\le
    \big| \Gamma_{z_o}-(Du)_{z_o,r_j}^{(\lambda)}\big|
    + 
    2^{N+2}
     \bigg[\biint_{Q_{r_j}^{(\lambda)}(z_o)}\big| Du- (Du)_{z_o,r_j}^{(\lambda)}\big|^p\,\dx\dt\bigg]^{\frac{1}{p}}\\
    &\le 
    C \Big(\frac{r}{R_o}\Big)^{\alpha_o }
    \max\Big\{
     \epsilon^{-\frac{\theta_3}p}
     \big(\mathsf E_{z_o,R_o}^{(\lambda)}\big)^\frac1p,
     \eps^\frac{p-1}p\lm
    \Big\}.
\end{align}
This shows 
\begin{equation*}
    \Gamma_{z_o}=\lim_{r\downarrow 0} (Du)_{z_o,r}^{(\lambda)},
\end{equation*}
which means that $\Gamma_{z_o}$ is the Lebesgue representative of $Du$ in $z_o$. {Estimate \eqref{cont-Lebesgue} essentially yields the $C^{1,\al}$-regularity of $u$ which will be elaborated in the next section. Here, we first exploit it to obtain the claimed gradient bound. Indeed, we choose
\begin{equation}\label{lambda-gradient-bound}
    \lambda
    :=
    \Big(\mu^2+\|Du\|_{L^\infty(Q_{R_2}(\tilde z) )}^2\Big)^{\frac12},
\end{equation}
so that \eqref{def-lambda} is  satisfied with $A=1$.}
Choosing $r=R_o$ in \eqref{cont-Lebesgue}, we get
\begin{align}\nonumber\label{est:Gamma_z_o}
    \big|\Gamma_{z_o}\big|
    &\le
    \big| (Du)_{z_o,R_o}^{(\lambda)}\big|
    +
    \big|\Gamma_{z_o}-(Du)_{z_o,R_o}^{(\lambda)}\big|\\\nonumber
    &\le
    \biint_{Q_{R_o}^{(\lambda)}(z_o)}|Du|\,\dx\dt
    +
    C\max\Big\{
     \epsilon^{-\frac{\theta_3}p}
     \big(\mathsf E_{z_o,R_o}^{(\lambda)}\big)^\frac1p,
     \eps^\frac{p-1}p\lm
    \Big\}\\
    &\le
    C\max\Big\{
     \epsilon^{-\frac{\theta_3}p}
     \big(\mathsf E_{z_o,R_o}^{(\lambda)}\big)^\frac1p,
     \eps^\frac{p-1}p\lm
    \Big\}.
\end{align}
We estimate the integral  on the right-hand side of \eqref{est:Gamma_z_o} using the energy inequality from Proposition~\ref{prop:energy-est-zero-order} in the form given in Remark~\ref{rem:energy-est-zero-order} 
on $Q_{R_o}^{(\lambda)}(z_o)\subset Q_{2R_o}^{(\lambda)}(z_o)$ with $\xi= (u)_{z_o,2R_o}^{(\lambda)}$. In this way we get
\begin{align}\label{energy-bound-Ro}\nonumber
      \mathsf E_{z_o,R_o}^{(\lambda)}
      &\le
      C
      \biint_{Q_{2R_o}^{(\lambda)}(z_o)}\Bigg[
      \frac{\big|u-(u)_{z_o,2R_o}^{(\lambda)}\big|^p}{(2R_o)^p} +
      \lambda^{p-2}\frac{\big|u-(u)_{z_o,2R_o}^{(\lambda)}\big|^2}{(2R_o)^2} +\mu^p\Bigg]
 \,\dx\dt\\
    &\le 
    C\Bigg[
    \bigg(\frac{\osc_{Q_{2R_o}^{(\lambda)}(z_o)}u}{2R_o}\bigg)^p
    +
    \lambda^{p-2}
    \bigg(\frac{\osc_{Q_{2R_o}^{(\lambda)}(z_o)}u}{2R_o}\bigg)^2
    +
    \mu^p
    \Bigg]
\end{align}
with $C=C(N,p,C_o,C_1)$. 

We now distinguish between the sub- and super-quadratic case. First consider the case $1<p<2$.  Since $Q_{2R_o}^{(\lambda)}(z_o)\subset Q_{R_2}(\tilde z)$, the choice of $\lambda$ from  \eqref{lambda-gradient-bound} yields
\begin{equation*}
      \lambda 
      =
      \Big(\mu^2+\|Du\|_{L^\infty(Q_{R_2}(\tilde z))}^2\Big)^{\frac12}
      \ge
    \Big(\mu^2+\|Du\|_{L^\infty(Q_{ 2R_o}^{(\lambda)}(z_o))}^2\Big)^{\frac12}.
\end{equation*}
This allows us to apply Remark \ref{rem:osc} and conclude that
\begin{align*}
    \osc_{Q_{ 2R_o}^{(\lambda)}(z_o)} u
    &\le 2CR_o\Big( \underbrace{\| Du\|_{L^\infty( Q_{ 2R_o}^{(\lambda)}(z_o))}}_{\le\, \lambda}+\lambda\Big)
    \le 
    4CR_o\lambda,
\end{align*}
where $C=C(N,C_1)$.
Using this inequality in \eqref{energy-bound-Ro} to bound the second term on the right-hand side yields
\begin{align*}
    \mathsf E_{z_o,R_o}^{(\lambda)}
    &\le
    C
    \Bigg[
    \bigg(\frac{\osc_{Q_{2R_o}^{(\lambda)}(z_o)}u}{2R_o}\bigg)^p
    +
    \mu^p
    \Bigg].
\end{align*}
In the super-quadratic case $p\ge 2$ we observe by Young's
inequality that
\begin{align*}
    X^p+\lm^{p-2}X^2
    &\le \tilde\eps^{p-1}\lm^p+ 2\tilde\eps^{-\frac12 (p-1)(p-2)}X^p,    
\end{align*}
where we denoted $X=\osc_{Q_{2R_o}^{(\lambda)}(z_o)}u/2R_o$ for short and $\tilde\eps\in(0,1]$. This observation in \eqref{energy-bound-Ro} yields
\begin{align}\label{energy-Ro}
      \mathsf E_{z_o,R_o}^{(\lambda)}
    &\le 
    C \tilde \epsilon^{p-1}\lambda^p
    +
    C\tilde\epsilon^{-\frac12(p-1)(p-2)} 
    \bigg(\frac{\osc_{Q_{2R_o}^{(\lambda)}(z_o)}u}{2R_o}\bigg)^p
    +C\mu^p.
\end{align}
Overall we have shown that \eqref{energy-Ro}
holds true for any $p>1$ with a constant $C$ depending only on $N,p,C_o$, and $C_1$, and with $\frac12 (p-1)(p-2)$ replaced by $\frac12 (p-1)(p-2)_+$. 

Using this inequality in \eqref{est:Gamma_z_o} we obtain
\begin{align*}
     |\Gamma_{z_o}|
     &\le 
     C\Bigg[
     \Big( \epsilon^\frac{p-1}{p}+
     \eps^{-\frac{\theta_3}{p}}\tilde\eps^\frac{p-1}{p}\Big)\lm+
     \eps^{-\frac{\theta_3}{p}}\bigg(
     \tilde\eps^{-\frac{(p-1)(p-2)_+}{2p}}
     \frac{\osc_{Q_{2R_o}^{(\lambda)}(z_o)}u}{2R_o}
    +
    \mu\bigg)\Bigg].
\end{align*}
Choosing $\tilde\eps :=\eps^{1+\frac{\theta_3}{p-1}}$ we get
\begin{align*}
     |\Gamma_{z_o}|
     &\le 
     C\Bigg[
     \epsilon^\frac{p-1}{p}\lambda
     +
     \eps^{-\theta_4}
     \bigg(
     \frac{\osc_{Q_{2R_o}^{(\lambda)}(z_o)}u}{2R_o}
    +
    \mu\bigg)\Bigg].
\end{align*}
with a positive exponent $\theta_4$ depending on $N,p,C_o,C_1,\alpha$, and in the sub-quadratic case $p<2$ also on $k$. The preceding inequality holds true for any $z_o\in Q_{R_1}(\tilde z)$. The oscillation term on the right-hand side can be estimated from above using the set inclusion $Q_{2R_o}^{(\lambda)}(z_o)\subset Q_{R_2}(\tilde z)\subset Q_{2\varrho}(\tilde z)$. The latter applies because $R_2\le 2\varrho$. Using the abbreviation $\boldsymbol\omega :=\osc_{Q_{2\varrho}(\tilde z)} u$, and recalling that $\Gamma_{z_o}$ is the Lebesgue representative of $Du$ at $z_o$
we have for almost every $z_o\in Q_{R_1}(\tilde z)$ that
\begin{align*}
    |Du(z_o)|
     &\le 
     C\bigg[
     \epsilon^\frac{p-1}{p}\lambda + \epsilon^{-\theta_4}
     \Big(
     \frac{\boldsymbol\omega}{2R_o}
    +
    \mu\Big)\bigg].
\end{align*}
Recalling the definition of $R_o$ from \eqref{choice-R_o} and passing to the supremum over $Q_{R_1}(\tilde z)$ we get
\begin{align}\label{sup-1}
    \sup_{Q_{R_1}(\tilde z)}|Du|
    &\le
     C\bigg[
     \epsilon^\frac{p-1}{p}\lambda + \epsilon^{-\theta_4}\big(1+\lambda^\frac{2-p}{2}\big)
     \frac{\boldsymbol\omega}{R_2-R_1}+\epsilon^{-\theta_4}\mu\bigg].
\end{align}

Based the above estimate, we now perform an interpolation argument which differs slightly depending on $p$.
In the case $1<p<2$ the exponent $\frac{2-p}{2}$ is non-negative, so that by Young's inequality with exponents $\frac{2}{2-p}$ and $\frac{2}{p}$ we can conclude
\begin{align*}
    \epsilon^{-\theta_4}\lambda^\frac{2-p}{2}\frac{\boldsymbol\omega}{R_2-R_1}
    &\le
    \epsilon^\frac{p-1}{p}\lambda +\underbrace{\epsilon^{-\frac2p(\theta_4+\frac{(p-1)(2-p)}{2p})}}_{=:\,\epsilon^{-\theta_5}}
    \Big(
    \frac{\boldsymbol\omega}{R_2-R_1}\Big)^\frac{2}{p}.
\end{align*}
Together with the choice of $\lambda$ from \eqref{lambda-gradient-bound} this leads to
\begin{align*}
    \sup_{Q_{R_1}(\tilde z)}|Du|
    &\le
     C\bigg[
     \epsilon^\frac{p-1}{p}
     \lambda
     +
     \epsilon^{-\theta_4}\frac{\boldsymbol\omega}{R_2-R_1}
     + \epsilon^{-\theta_5}
     \Big(
     \frac{\boldsymbol\omega}{R_2-R_1}\Big)^\frac{2}{p}+\epsilon^{-\theta_4}\mu\bigg]\\
     &\le
     C
     \epsilon^\frac{p-1}{p}
     \sup_{Q_{R_2}(\tilde z)}\big(\mu^2+|Du|^2\big)^\frac12\\
     &\phantom{\le\,}
     +
     C\bigg[
     \epsilon^{-\theta_4}\frac{\boldsymbol\omega}{R_2-R_1}
     +
     \epsilon^{-\theta_5}
     \Big(
     \frac{\boldsymbol\omega}{R_2-R_1}\Big)^\frac{2}{p}+\epsilon^{-\theta_4}\mu\bigg]\\
     &\le
     C
     \epsilon^\frac{p-1}{p}
     \sup_{Q_{R_2}(\tilde z)}|Du|
     +C\bigg[
     \epsilon^{-\theta_4}\frac{\boldsymbol\omega}{R_2-R_1}
     + \epsilon^{-\theta_5}
     \Big(
     \frac{\boldsymbol\omega}{R_2-R_1}\Big)^\frac{2}{p}+
     \epsilon^{-\theta_4}\mu\bigg].
\end{align*}
Here, we choose $C\epsilon^\frac{p-1}{p}=\frac12$. This fixes $\epsilon $ in dependence on $N,p,C_o,C_1,k$ and yields for any $\varrho\le R_1<R_2\le 2\varrho$ that
\begin{align*}
    \sup_{Q_{R_1}(\tilde z)}|Du|
    &\le
    \tfrac12 \sup_{Q_{R_2}(\tilde z)}|Du|
    +
    C\bigg[\frac{\boldsymbol\omega}{R_2-R_1}+\Big(
     \frac{\boldsymbol\omega}{R_2-R_1}\Big)^\frac{2}{p}+\mu
    \bigg],
\end{align*}
with a constant $C=C(N,p,C_o,C_1,\alpha,k)$. At this point we can apply Lemma \ref{lem:tech} and get the {\bf quantitative local gradient bound}
\begin{align}\label{local-gradient-bound}
    \sup_{Q_{\varrho}(\tilde z)}|Du|
    &\le
    C\bigg[\frac{\boldsymbol\omega}{\varrho}+\Big(
     \frac{\boldsymbol\omega}{\varrho}\Big)^\frac{2}{p}+\mu
    \bigg].
\end{align}
In the super-quadratic case $p\ge 2$ the exponent $\frac{2-p}{2}$ is negative. Therefore, we have to argue in a different way. Using  \eqref{lambda-gradient-bound} we first decrease the exponent in the sup-term from $1$ to $\frac{2}{p}$, gaining a factor $\lambda^\frac{p-2}{2}$. Subsequently, we apply \eqref{sup-1}. The additional power of $\lambda$ then compensates the negative  power $\lambda$ in \eqref{sup-1}. The precise argument is as follows
\begin{align*}
     \sup_{Q_{R_1}(\tilde z)}|Du|
     &\le
     \lambda^\frac{p-2}{p} \Big(\sup_{Q_{R_1}(\tilde z)}\big|Du\big|\Big)^\frac{2}{p}\\
     &\le
     C\lambda^\frac{p-2}{p}
     \bigg[
     \epsilon^\frac{p-1}{p}\lambda + \epsilon^{-\theta_4}\big(1+\lambda^\frac{2-p}{2}\big)
     \frac{\boldsymbol\omega}{R_2-R_1}+\epsilon^{-\theta_4}\mu\bigg]^\frac{2}{p}\\
     &\le
     C\big(\epsilon^\frac{p-1}{p}\big)^\frac{2}{p}\lambda
     +
     C\epsilon^{-\frac{2\theta_4}{p}} \big(\lambda^\frac{p-2}{p}+1\big)
     \Big(\frac{\boldsymbol\omega}{R_2-R_1}\Big)^\frac{2}{p}
     +
     C \lambda^\frac{p-2}{p} \epsilon^{-\frac{2\theta_4}{p}}\mu^\frac{2}{p}.
\end{align*}
This time we choose $\varepsilon$ to satisfy $C\big(\epsilon^\frac{p-1}{p}\big)^\frac{2}{p}=\frac14$ and obtain
\begin{align*}
     \sup_{Q_{R_1}(\tilde z)}|Du|
     &\le
     \tfrac14 \lambda
     +
     C \big(\lambda^\frac{p-2}{p}+1\big)
     \Big(\frac{\boldsymbol\omega}{R_2-R_1}\Big)^\frac{2}{p}
     +
     C \lambda^\frac{p-2}{p}\mu^\frac{2}{p}.
\end{align*}
Applying Young's inequality with exponents $\frac{p}{p-2}$ and $\frac{p}{2}$  to those terms containing  $\lambda^\frac{p-2}{p}$, we obtain for all radii $\varrho\le R_1<R_2\le 2\varrho$ that
\begin{align*}
     \sup_{Q_{R_1}(\tilde z)}|Du|
     &\le
     \tfrac12 \lambda
     +
     C \bigg[
     \frac{\boldsymbol\omega}{R_2-R_1}+ \Big(\frac{\boldsymbol\omega}{R_2-R_1}\Big)^\frac{2}{p} +\mu
     \bigg]\\
     &
     \le
      \tfrac12 \sup_{Q_{R_2}(\tilde z)}|Du|
    +
    C\bigg[\frac{\boldsymbol\omega}{R_2-R_1}+\Big(
     \frac{\boldsymbol\omega}{R_2-R_1}\Big)^\frac{2}{p}+\mu
    \bigg].
\end{align*}
 As in the sub-quadratic case, we recall the choice of $\lambda$ from \eqref{lambda-gradient-bound} and apply the iteration Lemma \ref{lem:tech} which yields the gradient bound  \eqref{local-gradient-bound}. In the super-quadratic case, the constant $C$ depends on $N,p,C_o,C_1$ and $\alpha$. This completes the proof.
\end{proof}

The following result gives a version of the previous gradient estimate over a compact subset of $E_T$.
\begin{corollary}\label{cor:global-gradient-bound}
Let $p>1$, $\mu\in (0,1]$, and $u$ be a bounded weak solution to the parabolic system \eqref{p-laplace-intro} with \eqref{prop-a-intro}  in the sense of Definition \ref{def:weak-loc}. Assume  \eqref{ass:Du} holds true. 
Then there exists a constant $C=C(N,p,C_o,C_1,\alpha,k)$ such that  for any compact subset $\mathsf{K}\subset E_T$ 
with $\rho:= \tfrac14\min\{1,\mathrm{dist}_{\mathrm{par}}(\mathsf{K},\partial_\mathrm{par} E_T)\}$, we have
the quantitative local gradient bound
\begin{equation*}
    \sup_{\mathsf{K}} |Du|
    \le
     C\bigg[\frac{\osc_{E_T}u}{\varrho}+\Big(
     \frac{\osc_{E_T}u}{\varrho}\Big)^\frac{2}{p}+\mu
    \bigg].
\end{equation*}
\end{corollary}
\begin{proof}
Recall that
\begin{equation*}
    \dist_{\mathrm{par}} (z_1,z_2):=|x_1-x_2|+\sqrt{|t_1-t_2|}.
\end{equation*}
Then, for any $\tilde z\in\mathsf{K}$ we have $Q_{2\varrho}(\tilde z)= B_{2\varrho}(\tilde x)\times( \tilde t-(2\varrho)^2,\tilde t]\Subset E_T$. Let
\begin{equation}\label{def:U}
    \mathsf{U}:=\bigcup_{\tilde z\in \mathsf{K}} Q_{\varrho}(\tilde z).
\end{equation}
Then, the local gradient estimate from Proposition~\ref{prop:quant-bound-for-Du} applied on $Q_{2\varrho}(\tilde z)$ implies that
\begin{equation*}
    \sup_{\mathsf{U}} \big(\mu^2+|Du|^2\big)^\frac12
    \le
     C\bigg[\frac{\osc_{E_T}u}{\varrho}+\Big(
     \frac{\osc_{E_T}u}{\varrho}\Big)^\frac{2}{p}+\mu
    \bigg]
\end{equation*}
with $C=C(N,p,C_o,C_1,\alpha,k)$. Since $\mathsf{K}\subset\mathsf{U}$ we get the asserted inequality.
\end{proof}

\begin{remark}\upshape \label{rem:super-critical}
In the \textbf{super-critical case} $p>\frac{2N}{N+2}$, we can
replace~\eqref{energy-Ro} by an estimate that is independent
of the oscillation of $u$. In fact, in view of the definition of
$R_o$ in~\eqref{choice-R_o} and $Q_R^{(\lambda)}(z_o)\subset Q_{R_2}(\tilde z)$, we obtain 
\begin{align}\label{alternative-energy-bound}\nonumber
   \mathsf E_{z_o,R_o}^{(\lambda)}
     &\le
       \frac{R_2^{N+2}}{R_o^{N+2}\lambda^{2-p}}
       \underbrace{\biint_{Q_{R_2}(\tilde z)}\big(\mu^2+|Du|^2\big)^{\frac{p}{2}}\, \dx\dt}_{=:\,\mathsf E_{\tilde z,R_2}}\\
     &=
       \Big(\frac{2R_2}{R_2-R_1}\Big)^{N+2}
       \max\Big\{\lambda^{p-2},\lambda^{\frac{N(2-p)}{2}}\Big\}
    \mathsf E_{\tilde z,R_2}.
\end{align}
By the choice of $\lambda$ in \eqref{lambda-gradient-bound}, the integrand on the left-hand side is bounded from above by $\lambda^p$. In the case $\frac{2N}{N+2}<p<2$, we use this fact to estimate 
\begin{align*}
     \mathsf E_{z_o,R_o}^{(\lambda)}
     &\le
     \lambda^{p(1-d_2)}\big(\mathsf E_{z_o,R_o}^{(\lambda)}\big)^{d_2}\\\nonumber
     &\le
     \Big(\frac{2R_2}{R_2-R_1}\Big)^{d_2(N+2)}
     \max\Big\{1,\lambda^{\frac{p(2-p)(N+2)}{4}}\Big\}
     \big(\mathsf E_{\tilde z,R_2}\big)^{d_2}\\\nonumber
     &\le
     \eps\lambda^p
     +
     \Big(\frac{2R_2}{R_2-R_1}\Big)^{d_2(N+2)}
     \big(\mathsf E_{\tilde z,R_2}\big)^{d_2}
    +
     C_\eps\Big(\frac{2R_2}{R_2-R_1}\Big)^{d_1(N+2)}
    \big(\mathsf E_{\tilde z,R_2}\big)^{d_1}
\end{align*}
for every $\epsilon\in(0,1)$, where 
\[
    d_1:=\frac{2p}{p(N+2)-2N}\quad\mbox{and}\quad d_2:=\tfrac{p}2
\] 
denote the \textbf{scaling deficits}. 
For the last estimate, we applied Young's inequality with exponents 
$\frac{4}{(2-p)(N+2)}$ and $\frac{4}{p(N+2)-2N}=\frac{d_1}{d_2}$. Note that $C_\eps= \eps^{-\frac{(2-p)(N+2)}{p(N+2)-2N}}$.

For exponents in the super-quadratic range  $p\ge2$, we argue similarly.
In fact, we have 
\begin{align*}
     \mathsf {E}_{z_o,R_o}^{(\lambda)}
     &\le
     \lambda^{p(1-d_1)}\big(\mathsf E_{z_o,R_o}^{(\lambda)}\big)^{d_1}\\\nonumber
     &\le
     \Big(\frac{2R_2}{R_2-R_1}\Big)^{d_1(N+2)}
     \max\big\{\lambda^{p-2d_1},1\big\}
     \big(\mathsf \mathsf E_{\tilde z,R_2}\big)^{d_1}\\\nonumber
     &\le
     \eps\lambda^p
     +
     C_\eps\Big(\frac{2R_2}{R_2-R_1}\Big)^{d_2(N+2)}
    \big(\mathsf E_{\tilde z,R_2}\big)^{d_2}+
     \Big(\frac{2R_2}{R_2-R_1}\Big)^{d_1(N+2)}
     \big(\mathsf E_{\tilde z,R_2}\big)^{d_1},
\end{align*}
for every $\epsilon\in(0,1)$. 
In the last line, we applied Young's inequality with exponents 
$\frac{p}{p-2d_1}$ and $\frac{p}{2d_1}$. Here we have $C_\eps= \eps^{-\frac14 (p-2)(N+2)}$.

Proceeding similarly as for the derivation of~\eqref{local-gradient-bound}, but replacing \eqref{energy-Ro} by the preceding inequality, 
now leads to a gradient sup-bound that is independent of the oscillation of $u$. 
More precisely, for every $p>\frac{2N}{N+2}$ we obtain  
\begin{equation*}
    \sup_{Q_{\rho}(\tilde z)}|Du|
    \le
    C \sum_{k=1}^{2} \bigg[\biint_{Q_{2\rho}(\tilde z) }(\mu^2+|Du|^2)^{\frac{p}{2}}\dx\dt\bigg]^{\frac{d_k}{p}},
\end{equation*}
with a constant $C =C (N,p,C_o,C_1,\alpha)$ and the scaling deficits $d_1$ and $d_2$.
For exponents in the sub-critical range $1<p\le\frac{2N}{N+2}$, however, the above procedure is
not applicable, since then the exponent $\frac{N(2-p)}{2}$ of $\lambda$ that
appears in~\eqref{alternative-energy-bound} is not smaller than
$p$. Therefore, it seems to be unavoidable that the right-hand side
of estimate~\eqref{local-gradient-bound} depends on the oscillation of
$u$ in the sub-critical range. \hfill$\Box$
\end{remark}

\subsection{H\"older continuity of the gradient} \label{sec:Hoelder} Having established the quantitative $L^\infty$-gradient bound for locally bounded weak solutions to the parabolic system \eqref{p-laplace-intro} under the qualitative assumption $|Du|\in L^\infty_{\rm loc}(E_T)$, we will show in this section (under the same assumptions) that a quantitative gradient H\"older estimate also holds. 
The main result is as follows.

\begin{proposition}\label{prop:h-grad}
Let $p>1$  and $\mu\in (0,1]$. Then there is a constant $C$ and a H\"older exponent $\alpha_o\in (0,1)$, both depending on $N,p,C_o,C_1,\alpha$ and $k$, so that whenever $u$ is a bounded weak solution to the parabolic system \eqref{p-laplace-intro} with \eqref{prop-a-intro}  in the sense of Definition \ref{def:weak-loc},  such that the assumption \eqref{ass:Du} holds true, then for any compact subset $\mathsf{K}\subset E_T$  with $\rho:= \tfrac14\min\{1,\mathrm{dist}_{\mathrm{par}}(\mathsf{K},\partial_\mathrm{par} E_T)\}$ we have
 \begin{align*}
    | Du(x_1,t_1)-Du(x_2,t_2)|
    &\le 
    C \lambda \Bigg[ \frac{|x_1-x_2|+\sqrt{\lambda^{p-2}|t_1-t_2|}}{\min\{1,\lambda^\frac{p-2}2\}\varrho}\Bigg]^{\alpha_o } 
\end{align*}
for any $(x_1,t_1),(x_2,t_2)\in\mathsf{K}$, where 
\begin{equation*}
     \lambda:=\frac{\osc_{E_T}u}{\rho}
    +
    \Big(\frac{\osc_{E_T}u}{\rho}\Big)^{\frac{2}{p}}
    +\mu.
\end{equation*}
\end{proposition}

\begin{proof}
Let
$R_1:=\frac12\varrho$, $R_2:=\varrho$, and $R_o:=\frac14\min\big\{ 1,\lambda^\frac{p-2}{2}\big\}\varrho$. We define the set $\mathsf{U}$ as in~\eqref{def:U}. 
Then, for any $\tilde z\in\mathsf{K}$ the gradient bound from Corollary~\ref{cor:global-gradient-bound} and the definition of $\lambda$ imply
\begin{equation*}
    \big(\mu^2 +\|Du\|_{L^\infty (Q_{R_2}(\tilde z))}^2\big)^\frac12
    \le
    \big(\mu^2 +\|Du\|_{L^\infty (\mathsf{U})}^2\big)^\frac12
    \le A\lambda,
\end{equation*}
where $A=A(N,p,C_o,C_1,\alpha,k)$ denotes the constant previously denoted by $C$ in Corollary \ref{cor:global-gradient-bound}.
This ensures that \eqref{def-lambda} holds on $Q_{R_2}(\tilde z)$ for any $\tilde z\in \mathsf{K}$. Therefore, the Campanato-type estimate \eqref{Campanato-u-2} is available for any $z_o\in \mathsf{K}$ and we obtain for every
$0<r\le R_o$ that
\begin{align*}
       \biint_{Q_r^{(\lambda)}(z_o) }&\big|Du -(Du)_{z_o,r}^{(\lambda)}\big|^p\,\dx\dt
     \le
    C \Big(\frac{r}{R_o}\Big)^{\alpha_op }\lambda^p
    \quad \forall\, r\in (0,R_o].
\end{align*}
In addition, \eqref{cont-Lebesgue} from the proof of Proposition~\ref{prop:quant-bound-for-Du} applies. Note that for the derivation of \eqref{cont-Lebesgue} we have chosen $\lambda$ according to \eqref{def-lambda}, which is also satisfied in our present setting.
With $\epsilon =1$ inequality \eqref{cont-Lebesgue} takes the form
\begin{align*}
    \big|\Gamma_{z_o}- (Du)_{z_o,r}^{(\lambda)}\big|^p
    &\le 
    C\Big(\frac{r}{R_o}\Big)^{\alpha_o p}
    \max\Bigg\{\underbrace{\bigg[
   \biint_{Q_{R_o}^{(\lambda)}(z_o)}
    \big(\mu^2+|Du|^2\big)^\frac{p}{2}\,\dx\dt\bigg]^\frac{1}{p}}_{\le\,A\lambda} \,,
    \,\lambda\Bigg\}^p\\
    &\le 
    C\Big(\frac{r}{R_o}\Big)^{\alpha_o p}\lambda^p.
\end{align*}
Joining this inequality with the Campanato-type estimate for $u$ from above, we obtain
\begin{align}\label{decay}
       \biint_{Q_r^{(\lambda)}(z_o) }&|Du -\Gamma_{z_o}|^p\,\dx\dt
     \le
    C \Big(\frac{r}{R_o}\Big)^{\alpha_op }\lambda^p
\end{align}
for any $z_o\in \mathsf{K}$ and any $r\in(0, R_o]$.
Now, consider $z_1,z_2\in\mathsf{K}$ such that
\begin{equation*}
    0<\mathrm{d}_{\mathrm{par}}^{(\lambda)} (z_1,z_2)=
    |x_1-x_2|+\sqrt{\lambda^{p-2}|t_1-t_2|}\le \tfrac12\varrho,
\end{equation*}
and let
\begin{equation*}
    r:=2\mathrm{d}_{\mathrm{par}}^{(\lambda)} (z_1,z_2)
    \quad\mbox{and}\quad
    z_\ast:= \big(\tfrac12 (x_1+x_2), \min\{ t_1,t_2\}\big).
\end{equation*}
Then, $Q_{\frac12 r}^{(\lambda)}(z_\ast)\subset Q_{ r}^{(\lambda)}(z_i)$ for $i=1,2$. Assuming that $r\le R_o$,
we can apply the decay estimate \eqref{decay} on $Q_{ r}^{(\lambda)}(z_1)$ and $Q_{ r}^{(\lambda)}(z_2)$ to conclude that
\begin{align*}
    | Du(z_1)&-Du(z_2)|^p\\
    &=
    \biint_{Q_{\frac12 r}^{(\lambda)}(z_\ast)}
    |\Gamma_{z_1}-\Gamma_{z_2}|^p\,\dx\dt\\
    &\le
    C(N,p)\Bigg[ 
    \biint_{Q_{ r}^{(\lambda)}(z_1)}
    |Du- \Gamma_{z_1}|^p\,\dx\dt
    +
    \biint_{Q_{ r}^{(\lambda)}(z_2)}
    |Du-\Gamma_{z_2}|^p\,\dx\dt\Bigg]\\
    &\le 
    C \Big(\frac{r}{R_o}\Big)^{\alpha_op }\lambda^p.
\end{align*}
In the case  $r>R_o$ we have
\begin{align*}
    | Du(z_1)-Du(z_2)|
    &\le 2\underbrace{\| Du\|_{L^\infty (\mathsf{K})}}_{\le\,A\lambda}
    \le 
    2A\lambda \Big(\frac{r}{R_o}\Big)^{\alpha_o }.
\end{align*}
The same inequality therefore holds true in both cases.  If we also keep in mind the definitions of $r$ and $R_o$ we then get 
\begin{align*}
    | Du(z_1)-Du(z_2)|
    &\le 
    C \lambda \bigg[ \frac{\mathrm{d}_{\mathrm{par}}^{(\lambda)} (z_1,z_2)}{\min\{1,\lambda^\frac{p-2}2\}\varrho}\bigg]^{\alpha_o }
\end{align*}
for any $z_1,z_2\in\mathsf K$ with $\mathrm{d}_{\mathrm{par}}^{(\lambda)} (z_1,z_2)\le\frac12\varrho$. In cases the  distance $\mathrm{d}_{\mathrm{par}}^{(\lambda)} (z_1,z_2)$
is larger than $\frac12\varrho$, we again argue using the quantitative gradient estimate. Indeed, we have
\begin{align*}
    | Du(z_1)-Du(z_2)|
    &\le 2\underbrace{\| Du\|_{L^\infty (\mathsf K)}}_{\le\,A\lambda}
    \le 
    2A \lambda\bigg[\frac{\mathrm{d}_{\mathrm{par}}^{(\lambda)} (z_1,z_2)}{\frac12 \varrho}\bigg]^{\alpha_o}\\
    &\le 
    4A\lambda \bigg[ \frac{\mathrm{d}_{\mathrm{par}}^{(\lambda)} (z_1,z_2)}{\min\{1,\lambda^\frac{p-2}2\}\varrho}\bigg]^{\alpha_o }.
\end{align*}
All together, the last two inequalities prove the assertion.
\end{proof}

\subsection{Schauder estimates for parabolic 
\texorpdfstring{$p$-}{p-}Laplace systems}\label{S:6}
In the previous sections, we considered bounded weak solutions to the parabolic $p$-Laplace system \eqref{p-laplace-intro} with Hölder continuous coefficients $a(x,t)$ satisfying \eqref{prop-a-intro}, assuming that $\mu>0$ and the gradients are locally bounded.
Now, we shall remove the assumption that  $|Du|\in L^\infty_{\rm loc}$ and also consider $\mu=0$.
To achieve this we consider smooth approximation $a_i(x,t)$ of $a(x,t)$ that retains the conditions \eqref{prop-a-intro} and generate approximating solutions $u_i$. Then we apply to $u_i$ the {\it a priori} estimates from previous sections and manage to pass to the limit. In particular, we can cover the case $\mu=0$ as these {\it a priori} estimates are independent of $\mu$. 

\begin{proof}[\rm{\textbf{Proof of Theorem~\ref{theorem:schauder}}}]
Let $\mathsf K\subset E_T$ and $\varrho:= \tfrac14\min\{1,\mathrm{dist}_{\mathrm{par}}(\mathsf K,\partial_\mathrm{par} E_T)\}$. Let $\epsilon_o\in (0,\tfrac12\varrho)$ and define the inner parallel set with distance $\epsilon_o$ to $\partial E$ by
\begin{equation*}
    \widetilde E:=\big\{ x\in E\colon \dist (x,\partial E)>\epsilon_o\big\} .  
\end{equation*}
Choose a standard mollifier $\phi\in C^\infty_0(B_1)$ with $\int_{B_1}\phi\,\dx =1$. For a sequence $\epsilon_i\in (0,\epsilon_o)$ with $\epsilon_i\downarrow 0$ define the scaled mollifiers by $\phi_{\epsilon_i}(x):= \epsilon_i^{-N}\phi (\frac{x}{\epsilon_i})$. The regularized coefficients $a_i$ are defined by
\begin{equation*}
    a_i(x,t):=\int_{B_{\epsilon_i}}a(x-y,t)\phi_{\epsilon_i}(y)\,\dy
\end{equation*}
for every $x\in\widetilde E$ and $t\in(0,T]$. Using the properties \eqref{prop-a-intro} of the coefficients $a$, it is not difficult to show that the following properties of $a_i$ are valid:
\begin{equation}\label{properties:a_i}
    \left\{
    \begin{array}{c}
        C_o\le a_i(x,t)\le C_1,\\[6pt]
        |a_i(x,t)-a_i(y,t)|\le C_1|x-y|^\alpha,\\[6pt]
        |D a_i(x,t)|\le c(N)C_1\epsilon_i^{-1}\|D\phi\|_{L^\infty},\\[6pt]
        | a_i(x,t)-a(x,t)|\le C_1\epsilon_i^\alpha,
    \end{array}
    \right.
\end{equation}
for every $i\in\N$, for any $x,y\in\widetilde E$, and $t\in (0,T]$. Here, $C_o,C_1$ are the constants from \eqref{prop-a-intro}. The regularization parameter $\mu_i$ is defined by 
\begin{equation*}
    \mu_i:=\left\{
    \begin{array}{cl}
        \epsilon_i,&\mbox{for $\mu=0$,}\\[5pt]
         \mu,&\mbox{for $\mu\in (0,1]$.}
    \end{array}
    \right.
\end{equation*}
With the coefficients $a_i$ and the parameters $\mu_i$ we denote by 
\begin{equation*}
    u_i\in L^p\big(0,T; W^{1,p}(\widetilde E,\R^k)\big)
    \cap C\big( [0,T]; L^2(\widetilde E,\R^k)\big)
\end{equation*}
the unique weak solution to the regularized Cauchy-Dirichlet problem (see Definition \ref{def:weak})
\begin{equation*}
    \left\{
    \begin{array}{cl}
        \partial_t u_i-\Div\Big( a_i(x,t)\big(\mu_i^2 + |Du_i|^2
        \big)^\frac{p-2}{2}Du_i\Big)=0,&\mbox{in $\widetilde E_T$,}\\[6pt]
        u_i=u,&\mbox{on $\partial_{\mathrm{par}}\widetilde E_T$.}
    \end{array}
    \right.
\end{equation*}
To $u_i$ and $u$ 
we apply the comparison estimate from Lemmas~\ref{lem:comp-2-mu>0} and~\ref{lem:comp-2-mu=0},  respectively,  with $a(x,t)$ and $b(x,t):= a_i(x,t)$ on $\widetilde E_T$. Then,  $
\mathsf A=\|a-a_i\|_{L^\infty (\widetilde E_T)}\le C_1\epsilon_i^\alpha$ and therefore,
\begin{align*}
    \iint_{\widetilde E_T}|Du-Du_i|^p\,\dx\dt
    &\le
    C\Big( \epsilon_i^{\alpha p}+\epsilon_i^\frac{\alpha p}{p-1}\Big)
    \iint_{\widetilde E_T} \big(\mu^2 +|Du|^2\big)^\frac{p}{2}\,\dx\dt,
\end{align*}
if $\mu\in (0,1]$, while if $\mu=0$ we have
\begin{align*}
    \iint_{\widetilde E_T}|Du-Du_i|^p\,\dx\dt
    &\le
    C\Big( (\epsilon_i^{\alpha}+\epsilon_i)^p+
    (\epsilon_i^{\alpha}+\epsilon_i)^\frac{p}{p-1}\Big)
    \iint_{\widetilde E_T} \big(1+|Du|^2\big)^\frac{p}{2}\,\dx\dt.
\end{align*}
In both cases we have strong convergence of $Du_i$ to $Du$, i.e.
\begin{equation*}
    Du_i\to Du \quad \mbox{in $L^p\big(\widetilde E_T,\R^k\big)$ as $i\to\infty$.}
\end{equation*}
Next, the comparison principle from Lemma \ref{lem:comp-1} yields
that $u_i\in L^\infty (\widetilde E_T,\R^k)$ and
that
\begin{align}\label{est:omega_i}
    \boldsymbol\omega_i
    :=
    \osc_{\widetilde E_T} u_i
    \le
    \sqrt{k}\osc_{\widetilde E_T} u
    \le
    \sqrt{k}\osc_{ E_T} u
    =:\sqrt{k}\boldsymbol\omega.
\end{align}
Now, \eqref{properties:a_i}$_1$ and \eqref{properties:a_i}$_3$ ensure that the coefficients $a_i$ satisfy the requirements of \eqref{ass:b}, i.e.~they are trapped by $C_o$ and $C_1$, and their gradient is bounded by $C_{2,i}:=c(N)C_1\epsilon_i^{-1}\|D\phi\|_{L^\infty}$,  $i\in\N$. 
By Propositions~\ref{lem:L-infty-est-qual-p<2} and \ref{lem:L-infty-est-p>2-intrinsic} we have that $|Du_i|\in L^\infty_{\rm loc} (\widetilde E_T)$. However, to obtain quantitative gradient bounds we cannot rely on those results in Section \ref{S:grad-bound} since $C_{2,i}\uparrow\infty$ as $i\to\infty$.  Instead, we apply the {\it a priori} estimates from Sections \ref{sec:gradient-bound} and \ref{sec:Hoelder}. Such application is viable as $a_i$ fulfill, thanks to \eqref{properties:a_i}$_1$ and \eqref{properties:a_i}$_2$, the conditions \eqref{prop-a-intro}, and moreover $|Du_i|\in L^\infty_{\rm loc}(\widetilde E_T)$. Consequently, Corollary~\ref{cor:global-gradient-bound} and Proposition~\ref{prop:h-grad} give that
\begin{align*}
    \sup_{\mathsf K} |Du_i|
    &\le 
    C\bigg[ 
    \frac{\boldsymbol \omega_i}{\varrho}
    +
    \Big(\frac{\boldsymbol \omega_i}{\varrho}\Big)^\frac2p
    +\mu_i
    \bigg]=:
    C\lambda_i
\end{align*}
and
\begin{align}\label{est:Hoelder-Du_i}
    |Du_i(z_1)-Du_i(z_2)|
    &\le
    C\lambda_i 
    \Bigg[ 
    \frac{\mathrm{d}_{\mathrm{par}}^{(\lambda_i)} (z_1,z_2)}{\min\{1,\lambda_i^\frac{p-2}2\}\varrho}
    \Bigg]^{\alpha_o}
\end{align}
for any subset $\mathsf K\subset \widetilde E_T$, and for any $z_1,z_2\in\mathsf K$. Note that we also used the simple fact $\tfrac14\min\{1,\mathrm{dist}_{\mathrm{par}}(\mathsf K,\partial_\mathrm{par} \widetilde E_T)\}\ge\rho-\varep_o\ge\frac12\rho$. The constants $C$ and $\alpha_o\in (0,1)$  in the preceding estimates
depend only on $N$, $p$, $C_o$, $C_1$, and $\alpha$ (and on $k$ if $p<2$), and are in particular independent of $i\in\N$. 
We abbreviate
\begin{align*}
    \lambda
    := 
    \frac{\osc_{E_T}u }{\varrho} +\Big(
    \frac{\osc_{E_T}u }{\varrho}\Big)^\frac2p +\mu.
\end{align*}
Then, \eqref{est:omega_i} implies
\begin{equation}\label{li-l}
    \limsup_{i\to\infty}\lambda_i\le \max\big\{ k^\frac12, k^\frac1p\big\}\lambda
    =:
    \mathsf k\lambda
    \quad\forall\, i\in\N,
\end{equation}
so that
\begin{equation*}
    \limsup_{i\to\infty}\Big(\sup_{\mathsf K} |Du_i|\Big)
    \le
    C\mathsf k\lambda.
\end{equation*}

Now, let $z_1,z_2\in\mathsf K$. 
Note that
the right-hand side of \eqref{est:Hoelder-Du_i} can be re-written as
\begin{equation}\label{dependence-lambda-i}
    C
    \bigg[
    \max\Big\{ \lambda_i^\frac{1}{\alpha_o}, \lambda_i^{\frac{1}{\alpha_o}+\frac{2-p}2}\Big\}
    \frac{|x_1-x_2|}{\varrho} +
    \max\Big\{ \lambda_i^{\frac{1}{\alpha_o}+\frac{p-2}2}, \lambda_i^\frac{1}{\alpha_o}\Big\}
    \frac{\sqrt{|t_1-t_2|}}{\varrho}
    \bigg]^{\alpha_o}.
\end{equation}
All exponents of $\lambda_i$ are non-negative. In fact, recalling the definition of
$\alpha_o$ from \eqref{def:alpha_o} we have $\alpha_o=\frac{\alpha_\ast\beta}{2(N+2+\beta+\alpha_\ast p)}\le\frac12\alpha_\ast$. In view of the definition of 
$\alpha_\ast$ from \eqref{alpha_ast} this shows in the case $1<p\le 2$ that $\frac{1}{\alpha_o}+\frac{p-2}2\ge 2 + \frac{p-2}2=\frac{p+2}{2}>0$ and in the case $p\ge 2$ that $\frac{1}{\alpha_o}+\frac{2-p}2\ge 2(p-1)+\frac{2-p}2\ge\frac{3p-2}{2}>0$.
This proves that the right-hand side of \eqref{est:Hoelder-Du_i} is non-decreasing in $\lambda_i$. Moreover, in view of \eqref{li-l} we have for the first term in \eqref{dependence-lambda-i} that
\begin{align*}
    \max\Big\{ \lambda_i^\frac{1}{\alpha_o}, \lambda_i^{\frac{1}{\alpha_o}+\frac{2-p}2}\Big\}
    &\le 
    \max\Big\{ (\mathsf k\lambda) ^\frac{1}{\alpha_o}, (\mathsf k\lambda)^{\frac{1}{\alpha_o}+\frac{2-p}2}\Big\}\\
    &\le
    \left\{
    \begin{array}{cl}
    \mathsf k^{\frac{1}{\alpha_o}+\frac{2-p}2}\max\Big\{ \lambda ^\frac{1}{\alpha_o}, \lambda^{\frac{1}{\alpha_o}+\frac{2-p}2}\Big\},&\mbox{if $1<p<2$,}\\[6pt]
    \mathsf k^{\frac{1}{\alpha_o}}\max\Big\{ \lambda ^\frac{1}{\alpha_o}, \lambda^{\frac{1}{\alpha_o}+\frac{2-p}2}\Big\},&\mbox{if $p\ge 2$,}
    \end{array}
    \right.
\end{align*}
while the second one is estimated by
\begin{align*}
    \max\Big\{ \lambda_i^{\frac{1}{\alpha_o}+\frac{p-2}2}, \lambda_i^\frac{1}{\alpha_o}\Big\}
    &\le 
    \max\Big\{ (\mathsf k\lambda)^{\frac{1}{\alpha_o}+\frac{p-2}2},
    (\mathsf k\lambda) ^\frac{1}{\alpha_o} \Big\}\\
    &\le
    \left\{
    \begin{array}{cl}
    \mathsf k^{\frac{1}{\alpha_o}}\max\Big\{ \lambda ^{\frac{1}{\alpha_o}+\frac{p-2}2}, \lambda^{\frac{1}{\alpha_o}}\Big\},&\mbox{if $1<p<2$,}\\[6pt]
    \mathsf k^{\frac{1}{\alpha_o}+\frac{p-2}2}\max\Big\{ \lambda ^{\frac{1}{\alpha_o}+\frac{p-2}2}, \lambda^{\frac{1}{\alpha_o}}\Big\},&\mbox{if $p\ge 2$.}
    \end{array}
    \right.
\end{align*}
This leads to
\begin{align*}
    \limsup_{i\to\infty}
    \big|Du_i(z_1)-Du_i(z_2)\big|
    &\le
    C\lambda
    \Bigg[ 
    \frac{\mathrm{d}_{\mathrm{par}}^{(\lambda)} (z_1,z_2)}{\min\{1,\lambda^\frac{p-2}2\}\varrho}
    \Bigg]^{\alpha_o},
\end{align*}
for any $z_1,z_2\in\mathsf{K}$. 
Therefore, the sequence $(Du_i)_{i\in\N}$ is bounded in $C^{\alpha_o,\alpha_o/2}\big(\mathsf K,\R^{Nk}\big)$. By the theorem of Ascoli-Arzel\`a there exists a uniformly converging subsequence, and since $Du_i\to Du$ in $L^p\big(\mathsf K,\R^{Nk}\big) $ we get that $Du_i\to Du$ uniformly in $\mathsf K$ as $i\to\infty$. Passing to the limit we find that
\begin{equation*}
    \sup_{\mathsf K} |Du|
    \le
    C\mathsf k\lambda
    \quad\mbox{and}\quad
    \big|Du(z_1)-Du(z_2)\big|
    \le
    C\lambda
    \Bigg[ 
    \frac{\mathrm{d}_{\mathrm{par}}^{(\lambda)} (z_1,z_2)}{\min\{1,\lambda^\frac{p-2}2\}\varrho}
    \Bigg]^{\alpha_o}
\end{equation*}
for any $z_1,z_2\in \mathsf K$. This finishes the proof of \eqref{gradient-sup-bound-final} and \eqref{gradient-holder-bound-final}.
\end{proof}

\section{Applications to doubly non-linear equations}
\label{sec:DNL}

\subsection{Regularity for doubly non-linear parabolic equations}\label{sec:intro-intro}

The main result of the present section deals with the gradient regularity of weak solutions $u: E_T\to \R$
to doubly non-linear parabolic equations 
\begin{equation}\label{doubly-nonlinear-prototype}
    \partial_t u^q - 
	\Div\big(|\nabla u|^{p-2}\nabla u\big)
	=
	0\quad \mbox{in $ E_T$,}
\end{equation}
in the {\it super-critical fast diffusion range} of $p$ and $q$. The formulation of the result is based, among other things, on the pointwise value $u(x_o,t_o)>0$. For this reason, it is convenient to state the result for continuous weak solutions; see Remark~\ref{Rmk:semicontin}. 

\begin{theorem}\label{THM:REGULARITY-INTRO}
Assume that $0<p-1< q<\frac{N(p-1)}{(N-p)_+}$.
There are constants $C,\widetilde{\boldsymbol\gm}>1$ and $\alpha_o\in(0,1)$, which  depend only on $N$, $p$ and $q$, such that if $u$ is a 
non-negative,  continuous, weak solution to \eqref{doubly-nonlinear-prototype} in $E_T$,
if $u_o:=u(x_o,t_o)>0$, and if the set inclusion
\[
\widetilde{\bg}Q_o
:=
K_{\widetilde{\bg}\rho}(x_o) \times \big(t_o - u_o^{q+1-p}(\widetilde{\boldsymbol\gm}\rho)^p, t_o + u_o^{q+1-p}(\widetilde{\boldsymbol\gm}\rho)^p\big)\subset E_T
\]
holds, then we have the gradient bound
\begin{equation}\label{grad-est-intro}
    \sup_{Q_o}|\nabla u|
    \le
    \frac{C u_o}{\rho},
\end{equation}
the Lipschitz estimate
\begin{equation}\label{diff-u-intro}
    |u(z_1)-u(z_2)|
    \le
    C u_o\,
    \Bigg[\frac{|x_1-x_2|}{\rho}+\sqrt{\frac{|t_1-t_2|}{u_o^{q+1-p}\rho^p}} \,\Bigg],
\end{equation}
and the gradient H\"older estimate
\begin{equation}\label{C1alph-intro}
    |\nabla u(z_1) - \nabla u(z_2)|
    \le
    \frac{C u_o}{\rho}\,
    \Bigg[\,\frac{|x_1-x_2|}{\rho}+\sqrt{\frac{|t_1-t_2|}{
      u_o^{q+1-p}\rho^p}}\, \Bigg]^{\alpha_o},
\end{equation}
for all $z_1=(x_1,t_1), z_2=(x_2,t_2)\in Q_o$. 
\end{theorem}

Figure \ref{Fig:har} illustrates the region where Theorem~\ref{THM:REGULARITY-INTRO} holds. 
We emphasize that Theorem~\ref{THM:REGULARITY-INTRO} is false for $q=p-1$ and $q=\frac{N(p-1)}{(N-p)_+}$; see Section~\ref{sec:reg:opt-1}.

\begin{figure}[H]
\psfragscanon
\psfrag{p}{$\scriptstyle p$}
\psfrag{q}{$\scriptstyle q$}
\psfrag{pzer}{$\scriptstyle p=1$}
\psfrag{qzero}{$\scriptstyle q=0$}
\psfrag{pn}{$\scriptstyle p=N$}
\psfrag{line}{$\scriptstyle q=p-1$}
\psfrag{curve}{$\scriptstyle q=\frac{N(p-1)}{N-p}$}
\includegraphics[width=.75\textwidth]{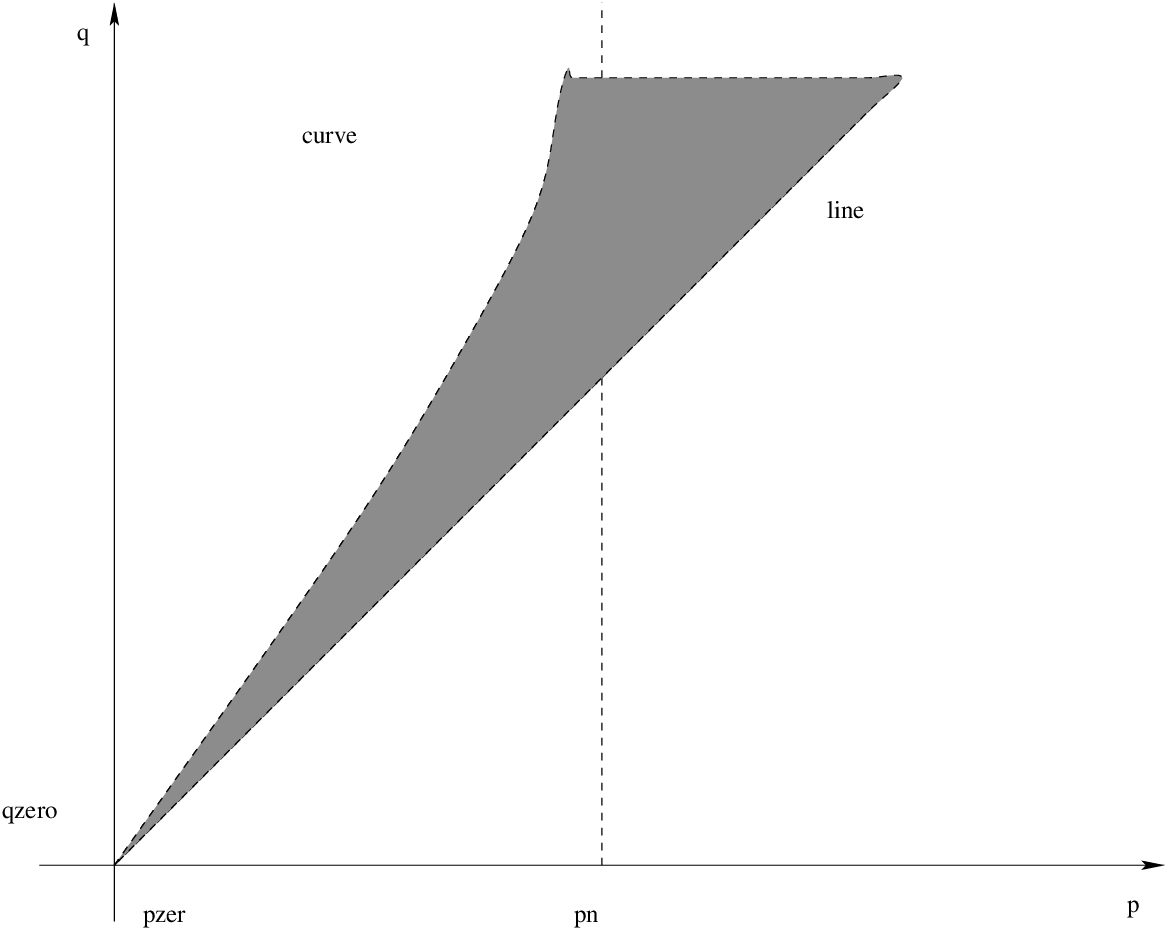}
\caption{\label{Fig:har}}
\end{figure}

\begin{remark}\label{Rmk:semicontin}\upshape
In the statement of gradient regularity, the continuity of $u$ is assumed in order to give an unambiguous meaning to $u(x_o,t_o)$. It suffices to state these estimates for the \textit{ upper semicontinuous regularization} $u^*$ of $u$, which is, for locally bounded solutions, uniquely determined and verifies $u=u^*$~a.e. in $E_T$. See~\cite[Theorem~2.3]{Liao-JMPA-21} for details. 
\end{remark}

Based on Theorem~\ref{THM:REGULARITY-INTRO}, the gradient estimates can also be formulated in a generic compact subset $\mathcal{K}$ of $E_T$. In particular, $\nabla u$ is H\"older continuous over $\mathcal{K}$, and precise estimates can be obtained in terms of $\sup_{\mathcal{K}}u$ and $\dist(\mathcal{K},\pl E_T)$.

\begin{corollary}\label{COR:GRAD-REG}
Under the assumptions of Theorem~\ref{THM:REGULARITY-INTRO} there exists $\alpha_1\in(0,1)$, depending only on $N,p$ and $q$,  such that the following holds.  Let
  $\mathcal{K}\subset E\times(0,T)$ be a compact subset and define
  \begin{equation*}
    \rho_o:=\inf_{\genfrac{}{}{0pt}{2}{(x,t)\in \mathcal K,}{ (y,s)\in\partial E_T}}\big(|x-y|+|t-s|^{\frac1p}\big)
  \end{equation*}
  and $\mathcal{M}:=\max\{1,\sup_{\mathcal{K}} u\}$. Then we have the gradient $L^\infty$-estimate
  \begin{equation}\label{grad-bound-K}
    \sup_{\mathcal{K}}|\nabla u|\le C\frac{\mathcal{M}^{\frac{q+1}{p}}}{\rho_o},
  \end{equation}
  as well as the gradient H\"older estimate
  \begin{equation}\label{grad-holder-K}
    |\nabla u(x_1,t_1)-\nabla u(x_2,t_2)|
     \le
    C\frac{\mathcal{M}^{\frac{q+1}{p}}}{\rho_o}
    \Bigg[\mathcal{M}^{\frac{q+1-p}{p}}\frac{|x_1-x_2|}{\rho_o}
    +\sqrt{\frac{|t_1-t_2|}{\rho_o^p}}\, \Bigg]^{\alpha_1}
\end{equation} 
for every $(x_1,t_1),$ $(x_2,t_2)\in \mathcal{K}$.     
\end{corollary}

For  $p=2$ our result reduces to the case of the singular porous medium equation in the range $1<q<\frac{N}{(N-2)_+}$. With $q=\frac1m$ 
and the substitution $w=u^\frac1m$ this corresponds to the bound $\frac{(N-2)_+}{N}<m<1$ for the standard form $\partial_tw-\Delta w^m=0$ of the porous medium equation. Therefore, Theorem \ref{THM:REGULARITY-INTRO} recovers
as a special case the gradient estimates of DiBenedetto \& Kwong \& Vespri \cite[Theorem 2.1]{DBKV} for weak solutions to the singular porous medium equation.

Our approach relies on a ``time-insensitive" Harnack inequality established in \cite{BDGLS-23} that is particular about this regime and the Schauder estimates for the parabolic $p$-Laplace equation developed here. We will give more details about the strategy in Section \ref{S:DN-strategy}.
The thrust of our contribution is to generate precise, quantitative, local estimates regarding weak solutions as well as their gradient that dictate their behavior. Moreover, it turns out that such estimates are sharp as evidenced by various explicit examples (see~Section~\ref{sec:reg:opt-1}).

\subsection{Notation and notion of solutions}
By $K_\rho(x_o)$
we denote the cube in $\R^N$ with center $x_o\in\R^N$ and side length $2\rho>0$, whose faces are parallel to the coordinate planes in $\mathbb R^N$. When $x_o=0$ we simply write $K_\rho$, omitting the reference to the center. In this section, space-time cylinders are denoted by 
\[
Q_R(z_o):=K_R(x_o)\times(t_o-R^p, t_o]\quad\text{and}\quad \mathcal{Q}_R(z_o):=K_R(x_o)\times(t_o-R^p, t_o+R^p).
\]
The use of cubes instead of balls is brought by the Harnack inequality established in \cite{BDGLS-23}.

\begin{definition}\label{Def:notion-sol}
A function $u$
is termed a {\bf weak super(sub)-solution} to the parabolic equation \eqref{doubly-nonlinear-prototype} in $E_T$, if
\begin{equation*}  
	u\in C\big( [0,T]; L^{q+1}(E)\big) \cap L^p \big(0,T; W^{1,p} (E)\big) 
\end{equation*}
and if the integral identity 
\begin{equation} \label{Eq:weak-form}
	\iint_{E_T} \Big[-|u|^{q-1}u\pl_t\z+|\nabla u|^{p-2}\nabla u\cdot \nabla\z\Big]\dx\dt
	\ge (\le) 0
\end{equation}
is satisfied for all non-negative test functions 
\begin{equation}\label{Eq:test-func}
\z\in  W_0^{1,q+1} \big(0,T; L^{q+1}(E)\big)\cap L^p \big(0,T; W^{1,p}_{0}(E)\big).
\end{equation}
A function that is both a weak super- and sub-solution is called a {\bf weak solution}. 
\end{definition}

\begin{remark}\upshape
    The function space \eqref{Eq:test-func} of $\z$ guarantees the convergence of the integral in \eqref{Eq:weak-form}.
\end{remark}

\begin{remark}\upshape
A notion of {\bf local solution} is commonly used in the literature, cf.~\cite{DB, BDL-21, BDLS-22}. 
We stress that this notion makes no essential difference from Definition~\ref{Def:notion-sol} modulo a localization.  
\end{remark}

\subsection{Strategy}\label{S:DN-strategy}
A few words about the proof  and the underlying ideas are in order here. 
The first observation is general in nature, and reflects a common principle in the treatment of non-linear parabolic equations with degeneracy as in \eqref{doubly-nonlinear-prototype}. One crucial building block in the study of the problem is the choice of the correct geometry in space and time: only the proper choice of geometry will lead to homogeneous estimates. In the present situation it turns out that the natural cylinders for our purposes are given by  $Q_o:=K_\rho(x_o)\times
\big( t_o-u_o^{q+1-p}\rho^p, t_o+u_o^{q+1-p}\rho^p\big)$, provided $u_o>0$. The geometry reflected by such cylinders is \emph{intrinsic} as it varies according to the value of the solution at the center. Once the correct
\textit{ intrinsic geometry}  
is identified, one can turn to the actual regularity assertions.

The starting point is the fact that, in the described parameter range a Harnack inequality without waiting-time phenomenon holds true. Therefore, from the assumption $u_o>0$ we may conclude that $u$ is controlled from above and below in terms of $u_o$ on the symmetric, intrinsic, space-time cylinder $Q_o$ (backward-forward cylinder).
More precisely, if $\bg_\mathrm{h}$ denotes the Harnack constant, then $\bg_\mathrm{h}^{-1}\le u/u_o\le \bg_\textrm{h}$ on $Q_o$. 
Rescaling $u/u_o$ in time and space to
the symmetric unit cylinder $\mathcal Q_1:= K_1(0)\times (-1,1)$ leads to a weak solution $\tilde u$ to the doubly non-linear parabolic equation \eqref{doubly-nonlinear-prototype} in $\mathcal Q_1$ with values in $[\bg_\textrm{h}^{-1}, \bg_\textrm{h}]$. This allows us to pass from   $\tilde u$ to $v:=\tilde u^q$.  Then  $v(y,s)$ is a weak solution to the parabolic equation
\begin{equation*}
	\partial_s v - 
	\Div\Big(\big(\tfrac1q\big)^{p-1}\tilde u^{(p-1)(1-q)}|\nabla v|^{p-2}\nabla v\Big)
	=
	0
	\quad\mbox{in $\mathcal{Q}_1$,}
\end{equation*}
with values in the interval $[\bg_\textrm{h}^{-q}, \bg_\textrm{h}^{q}]$. Furthermore, denoting
\begin{align*}
	a(y,s)
	:=
	\big(\tfrac1q\big)^{p-1}\tilde u(y,s)^{(p-1)(1-q)},
\end{align*}
the above equation can be interpreted as a parabolic equation of $p$-Laplace-type with measurable and bounded coefficients that are also uniformly bounded away from zero. This allows us to apply
the by now classical $C^{0,\alpha}$-regularity results from \cite{DB} to deduce that $v$ is $\alpha$-Hölder continuous in $\frac12\mathcal{Q}_{1}$ with some Hölder exponent $\alpha\in(0,1)$. Together with the lower and upper bound for $v$, such a fact implies that the coefficient $a$ is also $\alpha$-Hölder continuous. In particular, this shows that $v$ is a weak solution to a parabolic equation of $p$-Laplace-type with H\"older continuous coefficients in $\frac12\mathcal{Q}_{1}$. This reduces questions about Lipschitz and gradient H\"older regularity of $v$ to whether or not corresponding Schauder estimates for weak solutions can be established. Since the exponent $p$ can take values close to $1$ in the described parameter range, boundedness of weak solutions must be assumed in the context of Schauder estimates. However, in the range of $p$ and $q$ required, this is indeed the case. Theorem~\ref{theorem:schauder} 
ensures  $\alpha_o$-H\"older continuity of $Dv$ for some  H\"older exponent $\alpha_o\in(0,1)$ together with the quantitative gradient bound 
\[
    \sup_{\frac14\mathcal Q_1}|\nabla v|\le C,
\]
and the quantitative $\alpha_o$-H\"older gradient estimate 
\begin{equation*}
    \big|\nabla v(y_1,s_1) - \nabla v(y_2,s_2)\big|
    \le
    C\Big[ |y_1-y_2|+\sqrt{|s_1-s_2|}\Big]^{\al_o}
\end{equation*}
for any $(y_1,s_1), (y_2,s_2)\in \tfrac14 \mathcal{Q}_{1}$.
The asserted inequalities for the original solution $u$ are obtained by rescaling $v$ to $u$.

 To summarize, the proof strategy consists of three major steps. In the first step, we prove that weak solutions are bounded in the parameter range under consideration. In the second step we show that for non-negative weak solutions an elliptic Harnack-type inequality holds. This means that the Harnack inequality holds without waiting-time (backward-forward Harnack inequality). Thus, in a neighborhood near a point $(x_o,t_o)$ with $u_o=u(x_o,t_o)>0$ weak solutions are controlled from below and from above by $u_o$.
In the third and last step, one  substitutes $v=u^q$. The resulting equation of $v$ can be interpreted as a parabolic $p$-Laplace equation with measurable coefficients, so that  the classical result about Hölder regularity of weak solutions can be directly applied to it, and consequently the coefficients (essentially power functions of the solution) are identified as being themselves Hölder continuous. At this point the desired gradient regularity can be deduced from suitable Schauder estimates for parabolic $p$-Laplace equations with Hölder continuous coefficients. The peculiarity (and new feature) in this step is that we need to establish these Schauder estimates for any $p>1$. In particular, $p$ may be close to $1$.
Needless to say, Schauder estimates  for $p$ close to 1, can only be true under the additional assumption that weak solutions are bounded, which, however, is granted for our application here. 

A unique feature of weak solutions to \eqref{doubly-nonlinear-prototype} in the fast diffusion regime is that they might become extinct abruptly.
As a direct application of the gradient estimates from Theorem~\ref{THM:REGULARITY-INTRO} and the integral Harnack inequality from \cite[Theorem~3.13]{BDGLS-23}, decay estimates for weak solutions at their extinction time can be established.

\begin{corollary}\label{Cor:12:4}
Let $p$ and $q$ be as in Theorem \ref{THM:REGULARITY-INTRO}. Then there exist
 constants $C>1$ and $\bg_o>1$ depending only on $N$, $p$ and $q$, such that if $u$ is a non-negative weak solution
to the doubly non-linear parabolic equation \eqref{doubly-nonlinear-prototype} in $E_T$ and if $T$ is an extinction time for $u$, i.e.~$u(\cdot,T)=0$ a.e.~in $E$,
then, for any $x_o\in E$ and $t_o\in(\tfrac12T,T)$ we have
\begin{align*}\label{ext-speed}
    u(x_o,t_o)
    &
    \le C\bigg[\frac{T-t_o}{d^p(x_o)}\bigg]^{\frac{1}{q+1-p}},\\
    \big|\nabla u(x_o,t_o)\big|
    &
    \le \frac{C}{d(x_o)}\bigg[ \frac{T-t_o}{d^p(x_o)}\bigg]^{\frac{1}{q+1-p}}.
\end{align*}
Here, $d(x_o):=\dist(x_o,\partial E)$ denotes the usual Euclidean distance of $x_o$ to the boundary $\partial E$. Moreover, for any $r\in(0,d(x_o)/\bg_o)$
we have the following decay estimates for the oscillation of $u$ and its gradient, i.e.
\begin{align*}
   \osc_{K_r(x_o)} u(\cdot, t_o) 
   &\le 
   C\frac{ r}{d(x_o)}\bigg[\frac{T-t_o}{d^p(x_o)}\bigg]^{\frac{1}{q+1-p}},\\    \osc_{K_r(x_o)} \nabla u(\cdot, t_o) 
   &\le
\frac{C}{d(x_o)}\bigg[\frac{r}{d(x_o)}\bigg]^{\al_o}\bigg[\frac{T-t_o}{d^p(x_o)}\bigg]^{\frac{1}{q+1-p}}.
\end{align*}
Here,  $\al_o\in(0,1)$ is the H\"older exponent from Theorem~\ref{THM:REGULARITY-INTRO}. 
\end{corollary}

Corollary~\ref{Cor:12:4} will be proved in Section~\ref{sec:ext-time}, where we will also provide an estimate for the extinction time of the solution to a Cauchy-Dirichlet problem.

To prove Theorem~\ref{THM:REGULARITY-INTRO} we will take advantage of the Schauder estimates  established in this manuscript, and the ``time-insensitive'' Harnack inequality from~\cite{BDGLS-23}. However, before giving the proof, we recall some relevant results from the literature, and then briefly describe the optimality of the range of exponents $ 0<p-1<q<\frac{N(p-1)}{(N-p)_+}$ considered in Theorem~\ref{THM:REGULARITY-INTRO}.

\subsection{Remarks on the literature}\label{sec:reg:opt}

The investigation of regularity of weak solutions was initiated by Ivanov in a series of papers. Local boundedness of weak solutions has been studied in~\cite{Ivanov-1995-2}. Regarding this topic, we refer to \cite{BDGLS-23} for a more detailed exposition. The issue of H\"older continuity of weak solutions requires more delicate analysis and have attracted a number of authors. The well-studied ranges are the borderline case $p-1=q$, the doubly singular case, i.e.~$p-1<q$ and $1<p<2$, as well as the doubly degenerate case, i.e.~$q<p-1$ and $p>2$. See~\cite{Ivanov-1989, Ivanov-1994, Ivanov-1995-3, Porzio-Vespri, Urbano-08, KSU-12, BDL-21, BDLS-22, Liao-Schaetzler, Vespri-Vestberg}, and also Figure~\ref{Fig:hold}, where the white regions give a graphical representation of the well-studied ranges.

\begin{figure}
\psfragscanon
\psfrag{p}{$\scriptstyle p$}
\psfrag{q}{$\scriptstyle q$}
\psfrag{pzero}{$\scriptstyle p=1$}
\psfrag{qzero}{$\scriptstyle q=0$}
\psfrag{ptwo}{$\scriptstyle p=2$}
\psfrag{line}{$\scriptstyle q=p-1$}
\includegraphics[width=.75\textwidth]{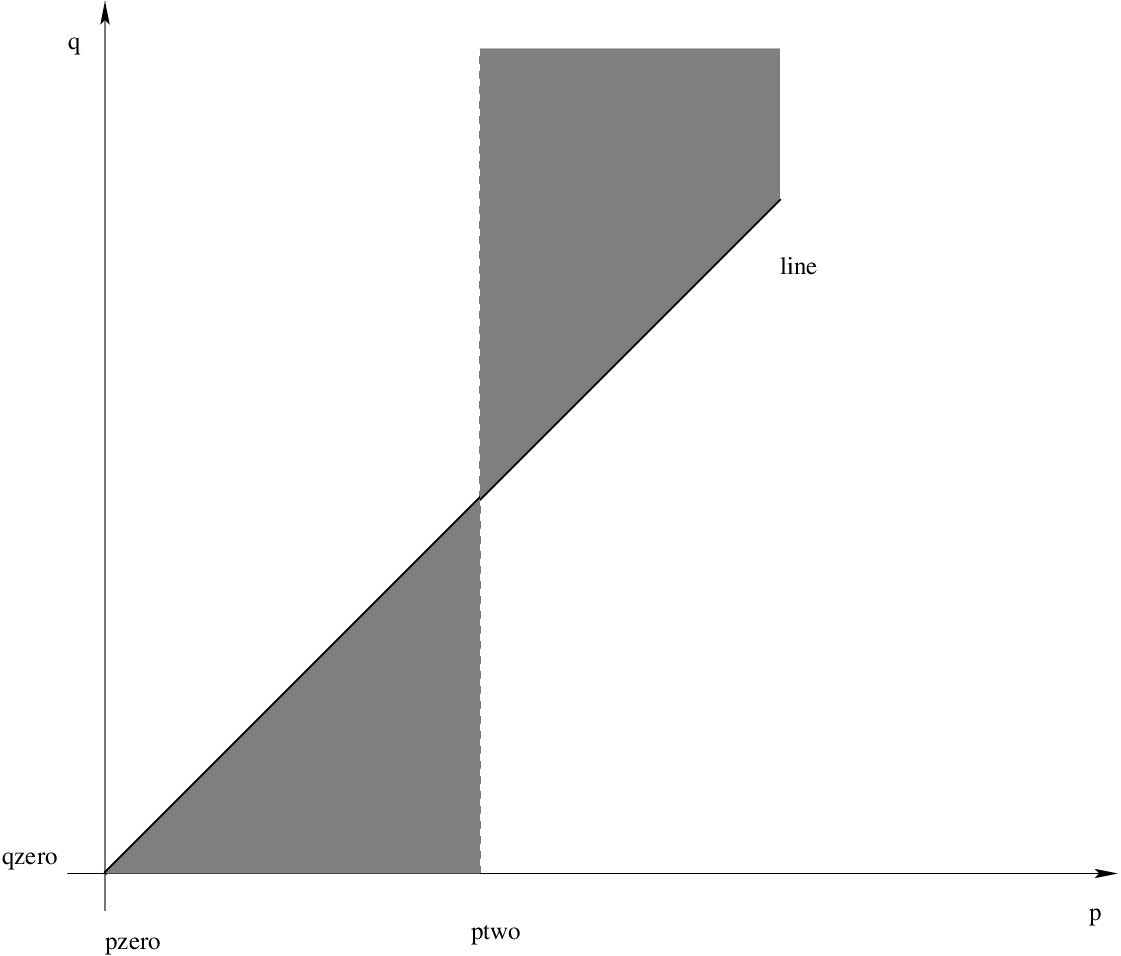}
\caption{\label{Fig:hold}}
\end{figure}

The theory becomes even more fragmented when it comes to gradient regularity. 
The well-studied case is the parabolic $p$-Laplace equation starting from DiBenedetto \& Friedman \cite{DiBenedetto-Friedman, DiBenedetto-Friedman3}. Indeed, the gradient of weak solutions is H\"older continuous.  

On the other hand, gradient regularity for solutions to the porous medium equation is more involved and much less understood. 
In the slow diffusion regime ($p=2$ \& $q<1$), only the one-dimensional case is reasonably understood. Aronson studied the Lipschitz continuity of $x\mapsto u^{1-q}(x,t)$ 
in \cite{Aronson-69}. For $0<q\le\frac12$ DiBenedetto~\cite{DB-1d} investigated the Lipschitz continuity of $t\mapsto u^{1-q}(x,t)$. The full range $0<q<1$ has later been established by B\'enilan~\cite{Benilan} and Aronson \& Caffarelli~\cite{Aronson-Caffarelli}. 
In the same regime, the optimal regularity in the multi-dimensional case is still a major open problem. 
Caffarelli \& V\'azquez \& Wolanski \cite{CVW} proved for a solution $u$ to the Cauchy problem that $\nabla u^{1-q}$ becomes bounded after some waiting time. This result was localized and improved for $0<q\le\frac12$ by Gianazza \& Siljander \cite{Gianazza-Siljander} who were, roughly speaking, able to quantify the waiting time. More precisely, if at some point $(x_o,t_o)$ the average $a^{q}=\bint_{B_r(x_o)} u^{q}(x,t_o) \,\dx$ is strictly positive, then $\nabla u^{1-q}$ is locally bounded in a certain cylinder, whose center lies at time $t_o+a^{q-1} r^2$. 

In the fast diffusion regime, ($p=2$ \& $q>1$) the picture is clearer. In the super-critical range $1<q<\frac{N}{(N-2)_+}$ DiBenedetto \& Kwong \& Vespri \cite{DBKV} proved that weak solutions are locally analytic in the space variable and at least Lipschitz-continuous in time up to the extinction time. They presented very precise, quantitative, regularity estimates up to the boundary.

The first investigations of the gradient regularity of solutions to the doubly non-linear equation \eqref{doubly-nonlinear-prototype} are due to Ivanov \& Mkrtychyan \cite{Ivanov-Mk}, Ivanov \cite{Ivanov-1996}, and Savar\'e \& Vespri \cite{Savare}. 
The results of  \cite{Ivanov-Mk} are not easily comparable with ours, since they are given for so-called {\em regular} solutions, that is, weak solutions to a Cauchy-Dirichlet problem built as suitable limits.
In the slow diffusion regime ($p>1$ \& $q<p-1$) which we do not deal with here,  a bound for $\nabla u^\alpha$ with $\alpha\in(0,1)$ holds for $q\le \frac{p-1}p$. 

As far as $\nabla u$ is concerned, combining \cite[Theorem~6.1]{Ivanov-1996} and the correction pointed out in \cite[\S~5]{Ivanov-1997}, yields the claim that the range of values of $p$ and $q$ in order for $\nabla u$ to be bounded is 
\[
p>1,\quad 0<q\le 1,\quad  \frac{(p-1)(1-q)}{q}<\frac{2\min\{1,p-1\}}{2\min\{1,p-1\}+\frac N4}.
\]
Even if we limit ourselves to the fast diffusion regime ($0<p-1<q$), which is the case we are interested in here, this does not correspond to the results of Theorem~\ref{THM:REGULARITY-INTRO}, and it is not clear how the Cauchy-Dirichlet problem Ivanov starts from, affects the interior regularity of the gradient. 

Finally, 
in \cite[Theorem~1.5]{Savare}, with a different notion of solution to \eqref{doubly-nonlinear-prototype} Savar\'e \& Vespri provided a gradient estimate
\[
|\nabla u(x_o,t_o)|\le \frac{C u(x_o,t_o)}{\rho}
\]
for some $C=C(p,q,N)$ and for parameters $p<N$, $0<p-1< q<\frac{N(p-1)}{N-p}$.

\subsection{Optimality}\label{sec:reg:opt-1}
Let us now present some counterexamples which illustrate that the statement of Theorem~\ref{THM:REGULARITY-INTRO} breaks down in the cases $q=p-1$ and $q=\frac{N(p-1)}{(N-p)_+}$. 

We start with $q=p-1$: in this case the function 
\[
u(x,t)=Ct^{-\frac{N}{p(p-1)}} \exp\left(-\frac{p-1}p\left(\frac{|x|^p}{pt}\right)^{\frac1{p-1}}\right)
\]
is a solution to \eqref{doubly-nonlinear-prototype} in $\R^N\times\R_+$ for arbitrary $C>0$. If we choose $t=1$ and $\rho=1$ and pick a sequence of points $x_n$ such that $|x_n|\to\infty$, it is easy to see that \eqref{grad-est-intro} cannot hold. This fact was first pointed out in \cite[Remark~1.6]{Savare}.

The estimate \eqref{grad-est-intro} does not hold in general when $q\ge\frac{N(p-1)}{(N-p)_+}$. Suppose the estimate were to hold for some $u_o>0$. Then letting $\rho\to\infty$ it would have implied $u$ is constant in $\R^N\times\R$. However, this contradicts the existence of non-constant solutions. Indeed, for parameters
\[
N\ge2,\quad N>p,\quad q=\frac{N(p-1)}{N-p},\quad b=\left(\frac Nq\right)^q,
\]
the function
\[
u(x,t)=\left(|x|^{\frac{N(q+1)}{q(N-1)}}+e^{bt}\right)^{-\frac{N-1}{q+1}}
\]
is a non-negative solution to \eqref{doubly-nonlinear-prototype} in $\R^N\times\R$. 

Moreover, the right-hand side of \eqref{grad-est-intro} reduces to $\frac C\varrho$. As for the left-hand side, considering the previous explicit solution,
we have
\[
|\nabla u(x,t)|=\frac Nq\frac{|x|^{\frac{N+q}{q(N-1)}}}{\left[|x|^{\frac{N(q+1)}{q(N-1)}}+e^{bt}\right]^{\frac{N+q}{q+1}}}.
\]
It is apparent that 
\[
\sup_{Q_o}|\nabla u|\ge\sup_{K_\varrho}|\nabla u(\cdot,0)|=\sup_{r\in(0,\varrho)} f(r),
\]
where we have set 
\[
|x|=r,\quad f(r)=\frac Nq\frac{r^{\frac{N+q}{q(N-1)}}}{\left[r^{\frac{N(q+1)}{q(N-1)}}+1\right]^{\frac{N+q}{q+1}}}.
\]
It is a matter of straightforward computations to check that if
\[
\rho>\left(\frac1{N-1}\right)^{\frac{q(N-1)}{N(q+1)}}:=\bar r,
\]
we have 
\[
\sup_{0<r<\varrho} f(r)=f(\bar r)=C(N,q).
\]
Hence, $\sup_{Q_o}|\nabla u|\ge C(N,q)$, and provided $\varrho$ is chosen sufficiently large, \eqref{grad-est-intro} cannot hold.

\subsection{Proof of Theorem~\ref{THM:REGULARITY-INTRO} and Corollary~\ref{COR:GRAD-REG}}
The first step is given by the following result. It shows that once two-sided bounds are imposed to a solution, then it behaves like the one to the parabolic $p$-Laplace equation. Here, $p>1$ and $q>0$ can be arbitrary.
\begin{theorem}\label{thm:doubly-bounded}
Let $p>1$, $q>0$, and $k>1$.
There are constants $C>1$ and $\alpha_o\in(0,1)$, which  depend only on $N, p,q$ and $k$, such that if $u$ is a 
non-negative weak solution to \eqref{doubly-nonlinear-prototype} in $E_T$, and if 
\begin{equation}\label{u-lu-k}
     \tfrac{1}{k}\le u \le k
     \qquad\mbox{in $E_T$}
\end{equation}
holds, then for any cylinder $Q_\rho(z_o)\subset E_T$ we have the gradient bound
\begin{equation}\label{grad-est-bd}
    \sup_{Q_{\rho/2}(z_o)}|\nabla u|
    \le
    \frac{C}{\rho},
\end{equation}
the Lipschitz estimate
\begin{equation}\label{diff-u-bd}
    |u(z_1)-u(z_2)|
    \le
    C
    \Bigg[\frac{|x_1-x_2|}{\rho}+\sqrt{\frac{|t_1-t_2|}{\rho^p}} \,\Bigg],
\end{equation}
and the gradient H\"older estimate
\begin{equation}\label{C1alph-bd}
    |\nabla u(z_1) - \nabla u(z_2)|
    \le
    \frac{C}{\rho}\,
    \Bigg[\,\frac{|x_1-x_2|}{\rho}+\sqrt{\frac{|t_1-t_2|}{
      \rho^p}}\, \Bigg]^{\alpha_o},
\end{equation}
for all $z_1=(x_1,t_1), z_2=(x_2,t_2)\in Q_{\rho/8}(z_o)$. 
\end{theorem}

\begin{proof}
We consider a cylinder $Q_\rho(z_o):=K_\rho(x_o)\times(t_o-\rho^p,t_o]\subset E_T$. Re-scaling to the unit cylinder $Q_1:=K_1(0)\times(-1,0]$ leads to a weak solution
\begin{align*}
	\tilde u(y,s)
	:=
	u\big(x_o+\rho y, t_o+ \rho^p s\big)
	\quad
	\mbox{for $(y,s)\in Q_1$,}
\end{align*}
to the doubly non-linear parabolic equation
\begin{equation}\label{DN-tilde}
	\partial_t \tilde u^q - 
	\Div\big(|\nabla\tilde u|^{p-2}\nabla\tilde u\big)
	=
	0
	\quad\mbox{in $Q_1$,}
\end{equation}
with values in $[k^{-1} ,k]$.
Substituting $v=\tilde u^q$, equation \eqref{DN-tilde} is equivalent to 
\begin{equation*}
	\partial_t v - 
	\Div\Big(\big(\tfrac1q\big)^{p-1} \tilde{u}^{(p-1)(1-q)}|\nabla v|^{p-2}\nabla v\Big)
	=
	0
	\qquad\mbox{in $Q_1$}
\end{equation*}
and assumption~\eqref{u-lu-k} yields
\begin{align}\label{lower-upper-v-k}
	k^{-q}
	\le
	v
	\le
	k^{q}
	\quad
	\mbox{in $Q_1$.}
\end{align}
With the abbreviation
\begin{align*}
	a(y,s)
	:=
	\big(\tfrac1q\big)^{p-1} [\tilde u(y,s)]^{(p-1)(1-q)},
\end{align*}
the equation for $v$ can be interpreted as a parabolic $p$-Laplace-type equation with
measurable coefficients $a(y,s)$ in $Q_1$. Indeed, $v$ is a bounded weak solution to 
\begin{equation}\label{equation-for-v-k}
	\partial_t v - 
	\Div\big(a(y,s)|\nabla v|^{p-2}\nabla v\big)
	=
	0
	\qquad
	\mbox{in $Q_1$.}
\end{equation}
For the coefficients $a$ we have  lower and upper bounds in terms of $k$, i.e.
\begin{equation} \label{bounds-for-a-k}
	\big(\tfrac1q\big)^{p-1}k^{-(p-1)|1-q|}
	\le 
	a(y,s)
	\le 
	\big(\tfrac1q\big)^{p-1}k^{(p-1)|1-q|}
	\quad
	\mbox{for a.e.~$(y,s)\in Q_1$.}
\end{equation}
Therefore, we may apply Chapter III, \S\,1, Theorem~1.1, resp. Chapter IV, \S\,1, Theorem~1.1 in \cite{DB} to deduce that $v$ is locally Hölder continuous in $Q_1$ with some Hölder exponent $\alpha\in(0,1)$ depending only on $N,p,q$ and $k$. 
In particular, $v$ is $\alpha$-Hölder continuous in $Q_{\frac{1}{2}}=K_{\frac{1}{2}}\times(-\frac{1}{4},0]$. 
Recalling \eqref{lower-upper-v-k}, the quantitative H\"older estimates from \cite{DB} yield
\begin{align*}
    |v(y_1,s)-v(y_2,s)|
    \le
    C |y_1-y_2|^\alpha
\end{align*}
for any $y_1,y_2\in K_{\frac{1}{2}}$ and $s\in (-\frac{1}{4},0]$, where $C=C(N,p,q,k)$. 
This and the lower and upper bound for $v$ in \eqref{lower-upper-v-k} imply that $a$ is also $\alpha$-Hölder continuous with the  quantitative  estimate 
\begin{align*}
    |a(y_1,s)-a(y_2,s)|
    \le
    C |y_1-y_2|^\alpha
\end{align*}
for any $y_1,y_2\in K_{\frac{1}{2}}$ and $s\in (-\frac{1}{4},0]$, where $C=C(N,p,q,k)$. 
Together with \eqref{equation-for-v-k} and \eqref{bounds-for-a-k} this shows that the hypotheses \eqref{prop-a-intro} of Theorem~\ref{theorem:schauder} (with $\mu=0$)
are fulfilled with constants $C_o,C_1$ depending only on $N,p,q,k$. 
Therefore, the Schauder-estimates from the same theorem are at our disposal. More precisely, we have $\nabla v\in C^{\alpha_o,\alpha_o/2}_{\mathrm{loc}}(Q_{1/2},\R^N)$ for some  H\"older exponent $\alpha_o\in(0,1)$, depending only on $N,p,q$ and $k$. 
Furthermore, the Schauder-estimates from the same theorem provide quantitative estimates. To be more precise, there exists a constant $C=C(N,p,q,k)\ge1$ such that for any compact subset $\mathcal{K}\subset Q_{1/2}$ with 
$\rho:= \tfrac14\mathrm{dist}_\mathrm{par}(\mathcal{K},\partial_\mathrm{par} Q_{1/2})>0$,
the $L^\infty$-gradient bound
\begin{equation*}
    \sup_{\mathcal{K}}|\nabla v|
    \le
    C\bigg[\frac{\boldsymbol\omega}{\rho}
    +
    \Big(\frac{\boldsymbol\omega}{\rho}\Big)^{\frac{2}{p}}
    \bigg]
    =:
    C\lambda,
\end{equation*}
and the $\alpha_o$-H\"older gradient estimate 
\begin{equation*}
    |\nabla v(\mathfrak z_1) - \nabla v(\mathfrak z_2)|
    \le
    C\lambda
    \bigg[\frac{\mathrm{d}_\mathrm{par}^{(\lambda)}(\mathfrak z_1,\mathfrak z_2)}{\min\{1,\lambda^{\frac{p-2}{2}}\}\rho}\bigg]^{\alpha_o}
    \quad\mbox{for any $\mathfrak z_1,\mathfrak z_2\in\mathcal{K}$,} 
\end{equation*}
hold true. Here we defined $\boldsymbol \omega:=\osc_{Q_{\frac12}} v$.
For $\mathcal K=Q_{\frac14}$, $\rho$ can be precisely quantified,
and $\boldsymbol \omega\le k^q$. Therefore, we have
\begin{equation*}
    \lambda
    =
    \frac{\boldsymbol\omega}{\rho}
    +
    \Big(\frac{\boldsymbol\omega}{\rho}\Big)^{\frac{2}{p}}
    \le 
    C(N,p,q,k),
\end{equation*}
so that 
\begin{equation}\label{sup-bound-Dv-k}
    \sup_{Q_{1/4}}|\nabla v|
    \le
    C(N,p,q,k),
\end{equation}
and 
\begin{equation}\label{osc-bound-Dv-k}
    |\nabla v(\mathfrak z_1) - \nabla v(\mathfrak z_2)|
    \le
    C(N,p,q,k)\,\mathrm{d}_\mathrm{par}(\mathfrak z_1,\mathfrak z_2)^{\alpha_o}\quad \mbox{for any $\mathfrak z_1,\mathfrak z_2\in Q_{1/4}$.}
\end{equation}
In the case $p>2$, to convert the intrinsic parabolic distance $\mathrm{d}_\mathrm{par}^{(\lambda)}$
into the parabolic distance $\mathrm{d}_\mathrm{par}$, we take advantage of the fact that we may assume $\alpha_o\le \frac{2}{p-2}$.

Once boundedness of  $\nabla v$ is established, we can derive a bound for the oscillation of $v$ on $Q_{\frac18}$ by an application of Lemma~\ref{lem:osc-all-p}. Indeed, for two points $\mathfrak z_1=(y_1,s_1),\,\mathfrak z_2=(y_2,s_2)\in  Q_{\frac18}$ with $s_1\le s_2$ we apply Lemma~\ref{lem:osc-all-p} (with $\mu=0$) to $v$ on the cylinder $Q=K_r(\tilde y)\times(s_1,s_2]\subset Q_{1/4}$, where $\tilde y=\frac12(y_1+y_2)$ and $\frac12|y_1-y_2|\le r\le\frac18$.  The application yields
\begin{equation*}
    |v(\mathfrak z_1)-v(\mathfrak z_2)|
    \le
    \osc_{Q} v
    \le 
    C\bigg[r\|\nabla v\|_{L^\infty(Q)} +
    \frac{s_2-s_1}{r} \|\nabla v\|_{L^\infty(Q)}^{p-1}\bigg]
    \le
    C\bigg[r + \frac{s_2-s_1}{r} \bigg].
\end{equation*}
Now, if $\sqrt{s_2-s_1}\le|y_1-y_2|$ we choose $r=\frac12|y_1-y_2|$, so that the right-hand side is bounded by $C|y_1-y_2|$. Otherwise, if $\sqrt{s_2-s_1}>|y_2-y_1|$ we choose $r=\sqrt{s_2-s_1} \le \frac18$ which results in the bound  $C\sqrt{s_2-s_1}$. Joining both cases, we obtain  
\begin{equation*}
    |v(\mathfrak z_1)-v(\mathfrak z_2)|
    \le
    C\,\mathrm{d}_\mathrm{par}(\mathfrak z_1,\mathfrak z_2)
    \quad \mbox{for any $\mathfrak z_1,\mathfrak z_2\in Q_{1/8}$}
\end{equation*}
for a constant $C= C(N,p,q,k)$.
In view of \eqref{lower-upper-v-k} this implies for any $\beta\in \R$ 
that
\begin{equation}\label{holder-v-beta-k}
    |v^\beta(\mathfrak z_1)-v^\beta(\mathfrak z_2)|
    \le
    C \,k^{q|\beta-1|} \mathrm{d}_\mathrm{par}(\mathfrak z_1,\mathfrak z_2)
    \quad \mbox{for any $\mathfrak z_1,\mathfrak z_2\in Q_{1/8}$, } 
\end{equation}
where $C=C(N,p,q,\beta,k)$. 

Finally, we scale back to the original solution $u$. From \eqref{holder-v-beta-k} with $\beta =\frac1q$ we obtain
\begin{equation*}
    |u(z_1)-u(z_2)|
    \le
    C 
    \Bigg[\frac{|x_1-x_2|}{\rho}+\sqrt{\frac{|t_1-t_2|}{\rho^p}} \,\Bigg]
      \quad 
      \mbox{for any $z_1,z_2\in Q_{\rho/8}$,} 
\end{equation*}
where $C=C(N,p,q,k)$.
This proves \eqref{diff-u-bd}. 
Next, computing the spatial gradient $\nabla u$ in terms of $v$ gives
\begin{align*}
    \nabla u(x,t)
    =
    \frac{v^{\frac1q-1}}{q\rho} 
    \bigg(\frac{x-x_o}{\rho}\,,\, \frac{t-t_o}{ \rho^p}\bigg) 
    \nabla v\bigg(\frac{x-x_o}{\rho}\,,\, \frac{t-t_o}{ \rho^p}\bigg).
\end{align*}
Then, \eqref{lower-upper-v-k} and \eqref{sup-bound-Dv-k} imply
\begin{equation*}
    \sup_{Q_{\rho/4}}|\nabla u|
    \le
    \frac{C}{\rho}.
\end{equation*}
By the same modifications as before, this yields \eqref{grad-est-bd}.
Combining  \eqref{sup-bound-Dv-k}, \eqref{osc-bound-Dv-k},  and  \eqref{holder-v-beta-k} (with $\beta =\frac1q-1$) gives
\begin{equation*}
    |\nabla u(z_1) - \nabla u(z_2)|
    \le
    \frac{C}{\rho}
    \Bigg[\,\frac{|x_1-x_2|}{\rho}+\sqrt{\frac{|t_1-t_2|}{
      \rho^p}}\, \Bigg]^{\alpha_o}
\end{equation*}
for any $z_1,z_2\in Q_{\rho/8}$, which implies \eqref{C1alph-bd}  as before. \end{proof}

\begin{remark}\label{rem:sym}\upshape
Note that Theorem~\ref{thm:doubly-bounded} continues to hold for symmetric cylinders of the type $\mathcal{Q}_\rho:=K_\rho(0)\times(-\rho^p,\rho^p)$.
In fact, in the statement and the proof, the one-sided cylinders $Q_\rho(z_o)$ can be replaced by the symmetric ones $\mathcal{Q}_\rho$. 
\end{remark}

We are now in the position to prove the main regularity results for doubly non-linear equations. Essentially, in the super-critical regime the artificial condition \eqref{u-lu-k} can be quantified by Harnack's inequality, and hence regularity estimates are formulated with respect to the equation's own intrinsic geometry.

\begin{proof}[{\textbf{\upshape Proof of Theorem~\ref{THM:REGULARITY-INTRO}}}]
As before, we note that the continuity of $u$ is only used in order to give $u(x_o,t_o)$ an unambiguous meaning; see also Remark \ref{Rmk:semicontin}.
From \cite[Theorem~1.5]{BDGLS-23} 
we infer that $u$ is locally bounded in $E_T$. 
We consider $z_o=(x_o,t_o)\in E_T$ such that $u_o:=u(x_o,t_o)>0$. By $\sigma=\sigma(N,p,q)\in(0,1)$ and 
$\boldsymbol \gm_\mathrm{h}=\boldsymbol\gm_\mathrm{h} (N,p,q)>1$ we denote the constants from the Harnack inequality of \cite[Theorem~1.11]{BDGLS-23}.  
Now, let $\rho>0$ and define $\widetilde{\boldsymbol\gm}:= \frac{8\bg_\mathrm{h}}{\sigma}>1$ and assume that
\[
	K_{\widetilde{\boldsymbol\gm}\rho}(x_o) \times \big(t_o - u_o^{q+1-p}(\widetilde{\boldsymbol\gm}\rho)^p, t_o + u_o^{q+1-p}(\widetilde{\boldsymbol\gm}\rho)^p\big)
	\subset E_T.
\]
Then the smaller cylinder
\[
    K_{8\rho/\sigma^\frac1p}(x_o)\times 
	\Big(t_o- \bg_\mathrm{h} u_o^{q+1-p}\big(8\rho/\sigma^\frac1p\big)^p,t_o+\bg_\mathrm{h} u_o^{q+1-p}\big(8\rho/\sigma^\frac1p\big)^p\Big)
\]
is also contained in $E_T$, so that we are allowed to apply Harnack's inequality on this cylinder. Moreover, we have 
\begin{align*}
    Q_o
    &:=
    K_{\rho}(x_o)\times 
	\big(t_o- u_o^{q+1-p}\rho^p,t_o+ u_o^{q+1-p}\rho^p\big) \\
	&\,\subset
    K_{\rho/\sigma^\frac1p}(x_o)\times 
	\big(t_o- u_o^{q+1-p}\rho^p,t_o+ u_o^{q+1-p}\rho^p\big) \\
	&\,=
    K_{\rho/\sigma^\frac1p}(x_o)\times 
	\Big(t_o- \sigma u_o^{q+1-p}\big(\rho/\sigma^\frac1p\big)^p,t_o+\sigma u_o^{q+1-p}\big(\rho/\sigma^\frac1p\big)^p\Big).
\end{align*}
Therefore, by Harnack's inequality we conclude
\begin{align*}
	\tfrac1{\boldsymbol\gm_\mathrm{h}} u_o
	\le
	u(x,t)
	\le
	\boldsymbol\gm_\mathrm{h} u_o
	\quad\mbox{for a.e.~$(x,t)\in Q_o$.}
\end{align*}
Re-scaling to the symmetric cylinder $\mathcal{Q}_\rho:=K_\rho(x_o)\times(t_o-\rho^p,t_o+\rho^p)$,
i.e. 
\begin{align*}
	\tilde u(y,s)
	:=
	\tfrac1{u_o}\, u\big(y, u_o^{q+1-p} s\big)
	\quad
	\mbox{for $(y,s)\in \mathcal{Q}_\rho(z_o)$,}
\end{align*}
leads to a weak solution $\tilde u$ to the doubly non-linear parabolic equation
\begin{equation*}
	\partial_t \tilde u^q - 
	\Div\big(|\nabla\tilde u|^{p-2}\nabla \tilde u\big)
	=
	0
	\quad\mbox{in $\mathcal{Q}_\rho(z_o)$,}
\end{equation*}
with values in $[\bg_\mathrm{h}^{-1} ,\bg_\mathrm{h}]$. At this point we can apply Theorem~\ref{thm:doubly-bounded} to $\tilde u$ in $\mathcal{Q}_\rho(z_o)$ (recall Remark~\ref{rem:sym}), to infer that 
\begin{equation*}
    \sup_{\mathcal{Q}_{\rho/2}(z_o)}|\nabla\tilde u|
    \le
    \frac{C}{\rho},
\end{equation*}
the Lipschitz estimate
\begin{equation*}
    |\tilde u(z_1)-\tilde u(z_2)|
    \le
    C
    \Bigg[\frac{|x_1-x_2|}{\rho}+\sqrt{\frac{|t_1-t_2|}{\rho^p}} \,\Bigg],
\end{equation*}
and the gradient H\"older estimate
\begin{equation*}
    |\nabla\tilde u(z_1) - \nabla\tilde u(z_2)|
    \le
    \frac{C}{\rho}\,
    \Bigg[\,\frac{|x_1-x_2|}{\rho}+\sqrt{\frac{|t_1-t_2|}{
      \rho^p}}\, \Bigg]^{\alpha_o},
\end{equation*}
for all $z_1=(x_1,t_1), z_2=(x_2,t_2)\in \mathcal{Q}_{\rho/8}(z_o)$. 
Rescaling back to the original solution $u$ yields the assertions of Theorem~\ref{THM:REGULARITY-INTRO}. 
\end{proof}

An immediate consequence of Theorem~\ref{THM:REGULARITY-INTRO} is as follows.
\begin{corollary}
Under the assumptions of Theorem~\ref{THM:REGULARITY-INTRO} we have 
\begin{equation*}
    |u(z_1)-u(z_2)|\le \frac{C ru_o}{\rho}.
\end{equation*}
and
\begin{equation*}
    |\nabla u(z_1) - \nabla u(z_2)|
    \le 
    C\Big(\frac{r}{\rho}\Big)^{\al_o}\,\frac{u_o}{\rho}
\end{equation*}
for any points $z_1,z_2\in K_r(x_o)\times(t_o- u_o^{q+1-p}\rho^{p-2}r^2,t_o+ u_o^{q+1-p}\rho^{p-2}r^2)$ and any radius $0<r\le \rho$.

\end{corollary}
\begin{proof}
The direct application of \eqref{diff-u-intro}, respectively \eqref{C1alph-intro} yields the claimed estimates.
\end{proof}

Next, we show how to extend the gradient-sup-bound and the $C^{1,\alpha}$-estimates to general compact subsets in $E_T$. This is exactly the assertion of Corollary~\ref{COR:GRAD-REG}. 

\begin{proof}[{\textbf{\upshape Proof of Corollary~\ref{COR:GRAD-REG}}}]
Let $\widetilde\bg>1$ be the constant from Theorem \ref{THM:REGULARITY-INTRO} and   $\rho:=\tfrac1{\widetilde\bg} \mathcal M^{-\frac{q+1-p}{p}}\rho_o$. Then the definition of $\rho_o$ implies 
  \begin{equation*}
    K_{\widetilde\bg\rho}(x)\times\big(t-\mathcal M^{q+1-p}(\widetilde\bg\rho)^p,t+\mathcal M^{q+1-p}(\widetilde\bg\rho)^p\big)\Subset E_T
    \qquad\mbox{for all }(x,t)\in \mathcal{K}.
  \end{equation*}
  Let us fix two points $(x_i,t_i)\in \mathcal{K}$ and denote $u_i:=u(x_i,t_i)$ for $i=1,2$.
  If $u_i>0$, we can apply the gradient bound \eqref{grad-est-intro} on the cylinders
  \begin{equation*}
    Q_i:= K_\rho(x_i)\times
    \big(t_i-u_i^{q+1-p}\rho^p,t_i+u_i^{q+1-p}\rho^p\big)\Subset E_T,
  \end{equation*}
  which implies in particular
  \begin{align}\label{est:Du_i}
    |\nabla u(x_i,t_i)|
    \le
    C\frac{u_i}{\rho}
    \le
    C\frac{\mathcal{M}}{\rho}
    =
    C\widetilde\bg\frac{\mathcal{M}^{\frac{q+1}{p}}}{\rho_o}.
\end{align}
In the case $u_i=0$, however, we have $u=0$
on the whole time slice $E\times\{t_i\}$, since otherwise,
repeated applications of Harnack's inequality would imply $u_i>0$. 
Therefore, we infer $\nabla u(x_i,t_i)=0$, so that~\eqref{est:Du_i} is trivially satisfied also in this case.
Since $(x_i,t_i)\in \mathcal{K}$ is arbitrary, this implies the asserted gradient bound~\eqref{grad-bound-K}.

For the proof of the $C^{1,\alpha}$-estimate, we distinguish  two cases.

\noindent \textbf{Case 1: } $|x_1-x_2|< \rho$ and $|t_1-t_2|< \max\{u_1,u_2\}^{q+1-p}\rho^p$. 

Without loss of generality, we assume that $u_1\ge u_2$, so that 
  $|t_1-t_2|< u_1^{q+1-p}\rho^p$. Note that this implies in particular  $u_1>0$. The case $u_1=u_2=0$ is included in Case~2. In the present situation we have
  \begin{equation}\label{Case1}
    (x_2,t_2)\in Q_1:= K_\rho(x_1)\times \big(t_1-u_1^{q+1-p}\rho^p,t_1+u_1^{q+1-p}\rho^p\big).
  \end{equation}
  Therefore, the $C^{1,\alpha}$-estimate \eqref{C1alph-intro} on the cylinder $Q_1$ implies
  \begin{equation*}
    |\nabla u(x_1,t_1)-\nabla u(x_2,t_2)|
    \le C\frac{u_1}{\rho}
    \Bigg[\frac{|x_1-x_2|}{\rho}+\sqrt{\frac{|t_1-t_2|}{u_1^{q+1-p}\rho^p}}\, \Bigg]^{\alpha_o}.
  \end{equation*}
  In view of~\eqref{Case1}, the term in brackets on the
  right-hand side is bounded by $2$. Therefore, we may diminish the
  power $\alpha_o$ and replace it by 
  $\alpha_1:=\min\{\alpha_o,\frac{2}{q+1-p}\}$. The resulting term is
  increasing in $u_1$, so that we can bound it from above by replacing
  $u_1$ by $\mathcal M$. Hence, we obtain the estimate
  \begin{align*}
    |\nabla u(x_1,t_1)-\nabla u(x_2,t_2)|
    &\le
      C\frac{\mathcal{M}}{\rho}
      \Bigg[\frac{|x_1-x_2|}{\rho}+\sqrt{\frac{|t_1-t_2|}{\mathcal{M}^{q+1-p}\rho^p}}\,\Bigg]^{\alpha_1}\\
    &\le
      C\frac{\mathcal{M}^{\frac{q+1}{p}}}{\rho_o}
      \Bigg[\mathcal{M}^{\frac{q+1-p}{p}}\frac{|x_1-x_2|}{\rho_o}
      +\sqrt{\frac{|t_1-t_2|}{\rho_o^p}}\,\Bigg]^{\alpha_1}.
  \end{align*} 
  In the last step, we used the definition of $\rho$. This yields the
  asserted estimate~\eqref{grad-holder-K} in the first case.

  \noindent \textbf{Case 2: } $|x_1-x_2|\ge \rho$ or
  $|t_1-t_2|\ge\max\{u_1,u_2\}^{q+1-p}\rho^p$.

In this case, we use the gradient bound~\eqref{est:Du_i}, which implies
  \begin{equation*}
    |\nabla u(x_i,t_i)|
    \le
    C\frac{u_i}{\rho}
    \le
    C\frac{u_i}{\rho}
    \Bigg[\frac{|x_1-x_2|}{\rho}+\sqrt{\frac{|t_1-t_2|}{u_i^{q+1-p}\rho^p}}\,\Bigg]^{\alpha_1},  
\end{equation*}
since the term in the brackets is bounded from below by one in the
current case.  
Similarly as above, we can
further estimate the right-hand side by monotonicity and replace $u_i$ by $\mathcal{M}$. In this way, we again obtain
\begin{align*}
  |\nabla u(x_1,t_1)-\nabla u(x_2,t_2)|
  &\le
  |\nabla u(x_1,t_1)|+|\nabla u(x_2,t_2)|\\
  &\le
  C\frac{\mathcal{M}^{\frac{q+1}{p}}}{\rho_o}
  \Bigg[\mathcal{M}^{\frac{q+1-p}{p}}\frac{|x_1-x_2|}{\rho_o}
  +\sqrt{\frac{|t_1-t_2|}{\rho_o^p}}\,\Bigg]^{\alpha_1}.
\end{align*}
This establishes the claim~\eqref{grad-holder-K} in the second case as well. 
\end{proof}

\subsection{Decay estimates at the extinction time}\label{sec:ext-time}

In the fast diffusion range solutions might become extinct abruptly. This means that for some $T>0$ we have $u(\cdot,T)=0$ a.e. in $E$. Such a time $T$ is called \textit{extinction time}. In this section we prove the decay estimates at the extinction time stated in Corollary~\ref{Cor:12:4}. Subsequently we derive an estimate for the extinction time of the solution to a Cauchy-Dirichlet problem in terms of the measure of the set $E$ and the $L^{q+1}$-norm of the initial values $u_o$.

\begin{proof}[{\textbf{\upshape Proof of Corollary~\ref{Cor:12:4}}}]
  For $x_o\in E$ and $t_o\in(\tfrac12T,T)$, we apply the integral Harnack inequality from \cite[Theorem~1.13]{BDGLS-23} on the
  cylinder $K_{d(x_o)}(x_o)\times(T-2(T-t_o),T]\subset E_T$. Since
  $u(\cdot,T)=0$, we infer the bound
  \begin{equation*}
    u_o:=u(x_o,t_o)\le C \bigg[\frac{T-t_o}{d^p(x_o)}\bigg]^{\frac{1}{q+1-p}}
  \end{equation*}
  with a constant $C$ depending on $N,p$, and $q$. This yields the first asserted estimate. Next, we choose $\rho_o:= \bg_o^{-1} d(x_o)$, where the
  constant $\bg_o> \widetilde\bg$ is at our disposal. From the last displayed estimate we have
  \begin{equation*}
    u_o^{q+1-p}(\widetilde\bg \rho_o)^p
    \le
    C^{q+1-p}\frac{T-t_o}{d^p(x_o)} \bigg[\frac{\widetilde\bg d(x_o)}{\bg_o}\bigg]^p
    \le
    \tfrac12(T-t_o).
  \end{equation*}
  The last inequality follows by choosing $\bg_o$  sufficiently large.
  The above choice of $\rho_o$ therefore implies
  \begin{align*}
    \widetilde\bg Q_o&:=K_{\widetilde\bg \rho_o}(x_o)\times
    (t_o-u_o^{q+1-p}(\widetilde\bg\rho_o)^p,t_o+u_o^{q+1-p}(\widetilde\bg\rho_o)^p)\\
    &\,\subset
    K_{\widetilde\bg\rho_o}(x_o)\times
    (t_o-\tfrac12(T-t_o),t_o+\tfrac12(T-t_o))
    \Subset E_T.
  \end{align*}
  Therefore, the gradient estimate \eqref{grad-est-intro} is available on the cylinder
  $Q_o$. By combining this estimate with the already established bound for $u$, we infer 
  \begin{equation*}
    |\nabla u(x_o,t_o)|
    \le
    \frac{C u_o}{\rho_o}
    \le
    \frac{C}{d(x_o)}\bigg[\frac{T-t_o}{d^p(x_o)}\bigg]^{\frac{1}{q+1-p}}.
  \end{equation*}
Moreover, for $r\in(0,\rho_o)$ we obtain from \eqref{diff-u-intro} that
\begin{equation*}
    \osc_{K_r(x_o)} u(\cdot,t_o) 
    \le
   \frac{C ru_o}{\rho_o}
   \le
   C\frac{ r}{d(x_o)}\bigg[\frac{T-t_o}{d^p(x_o)}\bigg]^{\frac{1}{q+1-p}}
\end{equation*}
and from \eqref{C1alph-intro} that
\begin{align*}
    \osc_{K_r(x_o)} \nabla u(\cdot,t_o) 
    \le 
    C\Big(\frac{r}{\rho_o}\Big)^{\al_o}\,\frac{u_o}{\rho_o}
    \le 
    \frac{C}{d(x_o)}\bigg[\frac{r}{d(x_o)}\bigg]^{\al_o}\bigg[\frac{T-t_o}{d^p(x_o)}\bigg]^{\frac{1}{q+1-p}},
\end{align*}
which completes the proof. 
\end{proof}

Concerning the extinction time $T$ we refer to Corollary~\ref{Cor:12:4}, where an upper bound for it can be given in a simple but nevertheless interesting situation. Consider the Cauchy-Dirichlet Problem
\begin{equation*}
\left\{
\begin{array}{cl}
    \partial_t \big(|u|^{q-1}u\big) - \Div (|\nabla u|^{p-2}\nabla u)=0  & \quad \mbox{in  $E_T$,}\\[6pt]
    u =0 &\quad \mbox{on $\partial E\times(0,T]$,}\\[6pt]
    u(\cdot,0)= u_o &\quad\mbox{in $E$,}
\end{array}
\right.
\end{equation*}
where $T\in (0,\infty]$ and the initial datum $u_o\ge 0$ is non-negative. For finite $T<\infty$ the existence of a non-negative solution to the Cauchy-Dirichlet Problem is discussed in \cite[Chapter~4, Remark~4.5]{BDGLS-23}. The case $T=\infty$ follows by the uniqueness from \cite[Chapter~4, Corollary~4.7]{BDGLS-23}.

\begin{proposition}
Let $0<p-1<q<\frac{N(p-1)+p}{(N-p)_+}$,
$E\subset\R^N$  a bounded domain, and $u_o\in L^{q+1}(E)$ with $u_o\ge0$. Moreover, let
\[
    u\in C\big([0,+\infty);L^{q+1}(E)\big)
    \cap L^p\big(0,+\infty;W^{1,p}_0(E)\big)
\]
be the unique non-negative weak solution to the Cauchy-Dirichlet problem. There exists a finite time $T>0$ depending only on $N$, $p$, $q$, $|E|$, and $u_o$, such that 
\[
    u(\cdot,t)=0\quad\text{ for all }\,\,t\ge T.
\]
Moreover, we have
\[
0<T\le \frac {q\, C^p}{q+1-p}\, |E|^{\frac{\lambda_{q+1}}{N(q+1)}}\|u_o\|_{L^{q+1}(E)}^{q+1-p},
\]
where $C$ is a constant 
that depends on $N$ and $p$,
and $\lambda_{q+1}:=N(p-q-1)+p(q+1)$.
\end{proposition}

\begin{proof} 
Taking $u$ as test function in the weak formulation of the Cauchy-Dirichlet Problem, \cite[Lemma~1.5]{AL} yields
\[
\|u(\cdot,t)\|_{L^{q+1}(E)}^{q+1}-\|u_o\|_{L^{q+1}(E)}^{q+1}+\frac{q+1}q\int_0^t\|\nabla u(\cdot,\tau)\|_{L^p(E)}^p\,\d\tau=0
\]
for any $t>0$. Let us first assume that $1<p<N$; the ranges of $p$ and $q$ ensure that $q+1<\frac{Np}{N-p}=:p_\ast$, so that the H\"older and the Sobolev inequalities give
\begin{equation}\label{Eq:11:1}
\|u(\cdot,t)\|_{L^{q+1}(E)}\le |E|^{\frac{\lambda_{q+1}}{Np(q+1)}}\,\|u(\cdot,t)\|_{L^{p_*}(E)}\le C|E|^{\frac{\lambda_{q+1}}{Np(q+1)}}\,\|\nabla u(\cdot,t)\|_{L^p(E)},
\end{equation}
where $C=C(N,p)$ is the optimal constant of the Sobolev inequality. 

Next, we consider
the case $p\ge N\ge1$. 
Let us first suppose $N\ge2$ and take $s\in(1,N)$ that satisfies $s_*:=\frac{Ns}{N-s}>q+1$; by H\"older's inequality, we have
\[
\|u(\cdot,t)\|_{L^{q+1}(E)}\le |E|^{\frac1{q+1}-\frac{N-s}{Ns}}\|u(\cdot,t)\|_{L^{s_*}(E)}.
\]
By the Sobolev and H\"older inequalities, we estimate
\[
\|u(\cdot,t)\|_{L^{s_*}(E)}\le C\|\nabla u(\cdot,t)\|_{L^s(E)}\le C |E|^{\frac1s -\frac1p}\|\nabla u(\cdot,t)\|_{L^p(E)},
\]
where again $C=C(N,s)$ is the optimal constant of the Sobolev inequality.
Joining the last two inequalities yields inequality \eqref{Eq:11:1}
in the case $p\ge N\ge2$. If $N=1$, by the Sobolev inequality we directly obtain that
\[
\|u(\cdot,t)\|_{L^{q+1}(E)}\le |E|^{\frac1{q+1}}\|u(\cdot,t)\|_{L^{\infty}(E)}\le C |E|^{\frac1{q+1}+1 -\frac1p}\|\nabla u(\cdot,t)\|_{L^p(E)},
\]
and once more we end up with \eqref{Eq:11:1}.

Consequently, combining the previous estimates yields in any case that
\[
\|u(\cdot,t)\|_{L^{q+1}(E)}^{q+1}-\|u_o\|_{L^{q+1}(E)}^{q+1}+\frac{q+1}{q C^p |E|^{\frac{\lambda_{q+1}}{N(q+1)}}}\int_0^t\|u(\cdot,\tau)\|_{L^{q+1}(E)}^{p}\,\d\tau\le0
\]
for every $t>0$.
Setting $\displaystyle v(t):=\|u(\cdot,t)\|_{L^{q+1}(E)}^{q+1}$, the previous inequality can be rewritten as
\[
v(t)-v(0)+\mu \int_0^t [v(\tau)]^{\frac p{q+1}}\,\d\tau\le0,
\]
where $v(0)=\|u_o\|_{L^{q+1}(E)}^{q+1}$ and 
\[
\mu:=\frac{q+1}{q C^p |E|^{\frac{\lambda_{q+1}}{N(q+1)}}}.
\]
By the regularity of $u$, $v$ is a continuous function of $t$. 
Just as before, we can prove that 
\[
    v(t_2)-v(t_1)\le-\mu\int_{t_1}^{t_2}[v(\tau)]^{\frac p{q+1}}\,\d\tau
    \qquad\forall\,t_2>t_1\ge0.
\]
Since $v\ge0$ by definition, it is apparent that $v$ is strictly decreasing wherever it is positive, $\lim_{t\to \infty}v(t)=0$, and if there exists $T>0$ such that $v(T)=0$, then $v(t)=0$ for any $t\ge T$.

Now, let $w:[0,+\infty)\to\R$ be the unique solution to the Cauchy Problem
\[
\left\{
\begin{array}{c}
w'(t)=-\mu[w(t)]^{\frac p{q+1}},\\[6pt]
w(0)=v(0),
\end{array}
\right.
\]
which can be computed as 
\[
w(t)=v(0)\left[1-\frac{\mu(q+1-p)}{(q+1)[v(0)]^{\frac{q+1-p}{q+1}}}\,t\right]_+^{\frac{q+1}{q+1-p}}.
\]
We obviously have
\[
    w(t_2)-w(t_1)=-\mu\int_{t_1}^{t_2}[w(\tau)]^{\frac p{q+1}}\,\d\tau
    \qquad \forall\,t_2>t_1\ge0.
\]
Suppose there exists $\tilde t\in(0,\infty)$ such that $v(\tilde t)>w(\tilde t)$ and let
\[
t_o:=\sup\big\{t\in(0,\tilde t):\, v(t)\le w(t)\big\}.
\]
We have $v(t_o)=w(t_o)$, and for any $t\in[t_o,\tilde t\,]$ there holds
\[
v(t)-v(t_o)\le-\mu\int_{t_o}^t [v(\tau)]^{\frac p{q+1}}\,\d\tau,
\quad 
w(t)-v(t_o)=-\mu\int_{t_o}^t [w(\tau)]^{\frac p{q+1}}\,\d\tau.
\]
Subtracting from one another yields that for any $t\in(t_o,\tilde t]$
\[
v(t)-w(t)\le-\mu\int_{t_o}^t\left([v(\tau)]^{\frac p{q+1}}-[w(\tau)]^{\frac p{q+1}}\right)\,\d\tau,
\]
which is a contradiction, since the left-hand side is positive, whereas the right-hand side is negative. Hence, we necessarily conclude that 
\[
v(t)\le w(t)\qquad\forall\,t\in[0,\infty) 
\]
which gives
\[
v(t)\le v(0)\left[1-\frac{\mu(q+1-p)}{(q+1)[v(0)]^{\frac{q+1-p}{q+1}}}\,t\right]_+^{\frac{q+1}{q+1-p}}.
\]
From this, reverting to $\|u(\cdot,t)||_{L^{q+1}(E)}$, we have
\[
\|u(\cdot,t)\|_{L^{q+1}(E)}\le\|u_o\|_{L^{q+1}(E)}\left[1-\frac{q+1-p}{q C^p |E|^{\frac{\lambda_{q+1}}{N(q+1)}} \|u_o\|_{L^{q+1}(E)}^{q+1-p}}\,t\right]_+^{\frac1{q+1-p}},
\]
and
\[
0<T\le \frac {q\, C^p}{q+1-p}\, |E|^{\frac{\lambda_{q+1}}{N(q+1)}}\|u_o\|_{L^{q+1}(E)}^{q+1-p}.
\]
This proves the claim.
\end{proof}

\begin{remark}\upshape
The same estimate of the extinction time $T$ is given in \cite[Proposition~3.5]{MKS}, assuming $1<p<N$, $q\ge1$, $p-1<q<\frac{N(p-1)+p}{N-p}$. In this paper there is also an estimate from below in the same ranges for $p$ and $q$. Indeed, in \cite[Proposition~4.2 and Corollary~4.3]{MKS} Misawa \& Nakamura \& Sarkar prove that when $u_o\in W^{1,p}_0(E)$ one has
\[
T\ge\frac{q}{q+1-p}\frac{\|u_o\|_{L^{q+1}(E)}^{q+1}}{\|D u_o\|_{L^p(E)}^p}.
\]
\end{remark}

\end{document}